\documentclass[12pt]{article}
\usepackage[utf8x]{inputenc}

\usepackage{latexsym}
\usepackage{mathrsfs}
\usepackage{amssymb}
\usepackage{amsfonts}
\usepackage{amsmath}
\usepackage{amsthm}
\usepackage{amscd}
\usepackage[usenames,dvipsnames]{color}
\usepackage{bbm}
\usepackage{stackrel}
\usepackage[colorlinks=true, linkcolor=blue]{hyperref}
\usepackage{textcomp}
\usepackage{mathabx}
\usepackage{enumerate}

\newtheorem{teo}{Theorem}[section]

\newtheorem{df}{Definition}[section]

\newtheorem{prop}{Proposition}[section]
\newtheorem{lm}{Lemma}[section]

\newtheorem{rem}{Remark}[section]

\newtheorem{coroll}{Corollary}[section]
\newtheorem*{cond_}{Condition}

\def\N{\mathbb N}
\def\P{\mathbb P}
\def\R{\mathbb R}
\def\E{\mathbb E}

\def\H{\mathcal H}
\def\Fc{\mathcal F}
\def\At{\mathcal{A}_{t,T}^\Theta}
\def\A0{\mathcal{A}_{0,T}^\Theta}
\def\Hs{\mathsf H}
\def\Xs{\mathsf X}
\def\Ys{\mathsf Y}
\def\Us{\mathsf U}

\def\L{\mathbb L}

\def\Cp{\mathbf{C}_{\,p.Lip}}
\def\Cbp{\mathbf{C}_{\,b,\,p.Lip}}
\def\F{\mathbf F}
\def\ind{\mathbbm 1}
\def\fee{\varphi}
\def\barw{\overline{\phantom{::}}}
\def\absz{\phantom{.}}
\def\Tr{\mathrm{Tr}}
\def\Trunc{\mathrm{Trunc}}
\def\rk{\mathrm{rk}}

\newcommand{\Ar}{\ \ \Rightarrow\ \ }

\newcommand{\tdot}{\raisebox{2pt}{$\scriptstyle\times$}}
\newcommand{\pdot}{\phantom.\!\!\cdot\!\!\phantom.}
\newcommand{\dperp}{\perp\negthickspace\!\!\perp}

\title{G-expectations in infinite dimensional spaces and related PDEs}
\author{Anton Ibragimov \thanks{Universit\`{a} degli Studi di Milano-Bicocca, Dipartimento di matematica e applicazioni, via R.Cozzi,~53, Milan, Italy; ibrahimov.ag@gmail.com} 
}

\begin{document}
\maketitle

\textbf{Abstract.} 
In this paper, we extend the $G$-expectation theory to infinite dimensions. Such notions as a covariation set of $G$-normal distributed random variables, viscosity solution, a stochastic integral driven by $G$-Brownian motion are introduced and described in the given infinite dimensional case. We also give a probabilistic representation of the unique viscosity solution to the fully nonlinear parabolic PDE with unbounded first order term in Hilbert space in terms of $G$-expectation theory.


\section{Introduction}

This paper is based on the author's Ph.D. thesis (see \cite{ibragimov1}) and devoted to a study of the theory of $G$-expectations in infinite dimensions. Using $G$-expectations as a probabilistic tool, we also study  equations of the form:
\begin{equation}\label{eq_1}
\begin{cases}
\partial_t u+\langle Ax,D_x u\rangle+G(D^2_{xx} u)=0 \,, \ \ \ t\in[0,T)\,, \ x\in\Hs;\\
u(T,x)=f(x) \,.
\end{cases}
\end{equation}
We call this equation a $G$-PDE, because of the occurrence of the nonlinear coefficient $G$.  $G$ is a certain sublinear functional which is connected to a $G$-expectation $\E$ by the formula $G(\, \pdot\, )=\dfrac{1}{2}\E\big[\langle \, \pdot\,  X,X\rangle\big]$. The term $A$ in the PDE is a given generator of a $C_0$-semigroup $\big(e^{tA}\big)$: the occurrence of this unbounded, not everywhere defined  term is important for the applications, but it requires to face additional difficulties.
The solution to equation \eqref{eq_1}  will be understood  in the sense of viscosity solutions. But treating viscosity solutions in the infinite dimensions also requires to overcome special difficulties (see, i.e., \cite{crandall_lions2, crandall_lions3, crandall_lions4, crandall_lions5, crandall_lions6}). 

\'{S}wi\c{e}ch (see \cite{kelome_swiech1, swiech1}) was the first author to include the ``unbounded'' term $\langle Ax,D_x u\rangle$ in the second order PDE. Together with Kelome (see \cite{kelome1, kelome_swiech1}) he proved a comparison principle and existence and uniqueness results for a nonlinear second order PDE. We will make use of their results on uniqueness of the solution to equation \eqref{eq_1}. In order to prove existence we will use a probabilistic representation which is entirely different from the  method of Kelome and \'{S}wi\c{e}ch.
 The probabilistic representation of the solution of equation \eqref{eq_1} is formally analogous to the classical case. To this aim we consider an associated stochastic differential equation:
\begin{equation}\label{eq_2}
\begin{cases}
d X_\tau=AX_\tau+dB_\tau\,,\ \ \ \tau\in(t,T]\subset[0,T]\ ;\\
X_t=x,
\end{cases}
\end{equation}
where, however, $B_\tau$ is a so called  $G$-Brownian motion in the Hilbert space $\Hs$, i.e. a Brownian motion related to a $G$-expectation that we introduce in an appropriate way.

The solution of equation \eqref{eq_2} is the following process, formally analogue to the Ornstein-Uhlenbeck process:
$$X_\tau:=X_\tau^{t,x}=e^{(\tau-t)A}x+\int\limits_t^\tau e^{(\tau-\sigma)A}dB_\sigma.$$
We will see that the formula \ $u(t,x):=\E \big[f(X_T^{t,x})\big]$ gives the required representation of the unique viscosity solution of equation \eqref{eq_1}.

In the definition of $X_\tau$ we are naturally led to considering a stochastic integral, which can be of the more general form $$\int\limits_0^\tau \Phi(\sigma)\,dB_\sigma.$$
So, we need to define a stochastic integral with respect to a $G$-Brownian motion. Its investigation gets us related    properties and results. Special attention will be devoted to the identifying a suitable class of integrand processes $\Phi$, with values in an appropriate space of linear operators that will be introduced to this purpose.

\section{Sublinear functionals and distributions}

\subsection{G-functional}\label{subsec_G_funct}

The notion of $G$-functional we need in order to use some sublinear functionals in infinite dimensions. Mainly it applies to the second order term in the heat equation
\begin{equation}
\hspace{30mm}\begin{cases}
\partial_t u+G(D^2_{xx} u)=0 \,, \ \ \ t\in[0,T)\,, \ x\in\Hs;\\
u(T,x)=f(x) \,.
\end{cases}
\end{equation}
Also, in the $G$-expectation theory $G$-functional plays a very important role as a tool of characterization $G$-normal distributed random variables.
\\
So, let us consider $\Hs$ is a real separable Hilbert space and $\{e_i\,, i\geq1\}$ be an orthonormal basis  on it.

Keeping the standard notations, define the following sets in this way:

$L(\Hs):=\big\{A:\Hs\to\Hs\mid A\text{ -- linear, continuous in $L(\Hs)$-topology}\big\}$;

$K(\Hs):=\big\{A\in L(\Hs)\mid A\text{ -- compact}\big\}$;

$L_S(\Hs):=\big\{A\in L(\Hs)\mid A=A^*\big\}$\ \ \ and\ \ \ $K_S(\Hs):=\big\{A\in K(\Hs)\mid A=A^*\big\}$.

\begin{df}\label{df_g_func}
A monotone, sublinear, continuous (in the operator norm) functional \ \
$G:D\subset L_S(\Hs)\rightarrow\R$\ \ is said to be \textbf{$\mathbf{G}$-functional}.

That is, $G(\pdot)$ is required to satisfy the following conditions:\\
\hspace*{5mm}\parbox[c][\height]{16cm}{
1)$A\geq \bar{A} \ \Rightarrow\ G(A)\geq G(\bar{A})$.\\
2)$G(A+\bar{A})\leq G(A)+G(\bar{A})$;\\
3)$G(\lambda A)=\lambda G(A),\ \lambda\geq0$;\\
4)$G$ is $L(\Hs)$-continuous.
}
\end{df}
 
 \begin{teo}\label{th_sublinear_F_sup_linear_f}
Let $\Xs$ be a linear space.

$F:\Xs\to\R$ is a sublinear functional, i.e.

\hspace{5mm}1) $F(x+y)\leq F(x)+F(y)$;

\hspace{5mm}2) $F(\lambda x)=\lambda F(x)\ ,\ \ \lambda\geq0$.

Then there exists a family of linear functionals $\big\{f_\theta:\Xs\to\R,\ \theta\in\Theta\big\}$, such that:
\begin{equation}\label{eq_F_sup_f}
F(x)=\sup\limits_{\theta\in\Theta}f_\theta(x),\ x\in\Xs.
\end{equation}

Moreover, (a) If $F$ is continuous $\Ar f_\theta$ in \eqref{eq_F_sup_f} are continuous.
\\
\phantom{Moreover,} (b) \parbox[t][\height]{11cm}{If $F$ is a monotone, sublinear functional, that preserves constants
 (such functional we call \textbf{sublinear expectation})\\ 
$\Rightarrow\ \ \  f_\theta$ in \eqref{eq_F_sup_f} are linear expectations.}
\end{teo}

\begin{proof}
 $\absz$\\
The original proof in a finite dimensional case you can see in \cite[Th.2.1]{peng10} but some parts in infinite dimensions requires additional passages, so that the full proof you can consult in the author's thesis \cite{ibragimov1}.

\end{proof}

\begin{df}\label{von_neumann-schatten_cl}
 The von Neumann-Schatten classes of operators are defined as follows: 
\\
$\begin{array}{l}
C_p(\Hs):=\big\{A\in L(\Hs)\mid\sum\limits_{j\geq1}\big|\langle A e_j, e_j\rangle\big|^p<\infty\big\}\ ,\ \  1\leq p<\infty;\\ 
C_\infty(\Hs):=L(\Hs).
\end{array}$
\end{df}
 
 Introducing a norm for $A\in C_p(\Hs), \ \ 1\leq p\leq\infty:\\
 \|A\|_{C_p}:=\Big[\Tr\big(A\!\cdot A^*\big)^{p/2}\Big]^{1/p},\ 1\leq p<\infty,\\
\|A\|_{C_\infty}:=\|A\|_{L(\Hs)}$;
\\
Also we know that $\big(C_p(\Hs),\|\pdot\|_{C_p}\big),\ 1\leq p\leq\infty$ is a Banach space 

\hfill (see \cite[2.1]{ringrose1}).

\begin{rem}
In the sequel we will call the classes $C_1(\Hs)$ and $C_2(\Hs)$ as \textbf{trace-class} and \textbf{Hilbert-Schmidt class} of operators respectively,
and denote them in this way:
\\
$\begin{array}{l}
C_1(\Hs)\equiv L_1(\Hs):=\big\{A\in L(\Hs)\mid \Tr\big[(A\cdot A^*)^{1/2}\big]<\infty\big\};\\
C_2(\Hs)\equiv L_2(\Hs):=\big\{A\in L(\Hs)\mid \Tr\big[A\cdot A^*\big]<\infty\big\}.
\end{array}$
\end{rem}

Now we would like to give the representation of the $G$-functional defined on the set of compact symmetric operators on $\Hs$. Actually, we can't give the same representation result in the general case with the domain of linear bounded operators. But afterwards we discuss about extension $G$-functional on $L_S(\Hs)$.

\begin{teo}\label{th_charact_G_func}
 Let  $G:K_S(\Hs)\rightarrow\R$ be a $G$-functional.

Then there exists a set $\Sigma$ such that:

\hspace*{5mm}1)$\Sigma\subset C_1(\Hs)$;\\
\hspace*{5mm}2)$\forall\,B\in\Sigma\ \Rightarrow\ B=B^*,\ B\geq0$;\\
\hspace*{5mm}3)$\Sigma$ is convex;\\
\hspace*{5mm}4)$\Sigma$ is closed subspace of $C_1(\Hs)$;\\
\hspace*{5mm}5) $G(A)=\dfrac{1}{2}\sup\limits_{B\in\Sigma}{\Tr}\big[A\cdot B],\ \  \forall\,A\in K_S(\Hs)$.
\end{teo}

\begin{proof}
 $\absz$\\
For every sublinear continuous functional $F:K(\Hs)\to\R$ according to \textbf{Th.\ref{th_sublinear_F_sup_linear_f}}, there exist a family of linear continuous (in $L(\Hs)$-topology) functionals 
$\big\{f_\theta,\ \theta\in\Theta\big\}$,\ \ \ 
such that\ \ $F(A)=\sup\limits_{\theta\in\Theta}f_\theta(A)\,,\ \ \ A\in K(\Hs)$.

Let us fix $\theta$ and consider $f=f_\theta\in L\big(K(\Hs),\R\big)$.
\\
For \ \  $x,y\in\Hs$ \ \  we define a bounded linear operator $x\otimes y$ on $\Hs$ as follows:

$(x\otimes y)z:=\langle z, y\rangle\, x\,,\ z\in\Hs$.
\\
It is clear that $\rk (x\otimes y)=1$, if $x\neq0,\,y\neq0$, and
\\
$\Tr[x\otimes y]=\sum\limits_{j\geq1}\big\langle (x\otimes y)e_j,e_j\big\rangle
=\sum\limits_{j\geq1}\langle e_j,y\rangle \langle x,e_j\rangle=\langle x,y\rangle $, and

the norm \ \ $\|x\otimes y\|_{L(\Hs)}:=\|x\|_\Hs\cdot\| y\|_\Hs\,,\ \ \ 1\leq p\leq\infty$.
\\
The bilinear form $L(x,y):=f(x\otimes y)$ satisfies

$\big|L(x,y)\big|=\big|f(x\otimes y)\big|
\leq\|f\|_{L(L(\Hs),\Hs)}\cdot\|x\otimes y\|_{L(\Hs)}
\leq\|f\|_{L(L(\Hs),\Hs)}\cdot\|x\|_\Hs\cdot\| y\|_\Hs$,

So that, it 
is bounded.

But for every bounded bilinear form there exists a unique $B\in L(\Hs)$\ such that\ $L(x,y)=\langle Bx,y\rangle$ \
(see \cite[Th.1.7.1]{ringrose1}).
\\
It follows that \ $f(x\otimes y)=\langle Bx,y\rangle=\langle x,B^* y\rangle=\Tr[x\otimes B^* y]$

\hfill$=\Tr[(x\otimes y)B]$, 

because $(x\otimes B^* y)z=\langle z,B^* y\rangle\, x=\langle B z,y\rangle\, x=(x\otimes y)Bz$. 

So, we can conclude that \ $f(A)=\Tr\big[A B\big]\,,$\ \ where \ $ A:=x\otimes y$. {\makebox[20mm][r]{$(\,\ast\,)$}}
\\
Denote $\Fc$ to be the set of all operators of finite rank on $\Hs$.

But every $A\in\Fc$ can be represented in the form \ $A=\sum\limits_{j=1}^m \alpha_j \psi_j\otimes\fee_j$ 
\\
(see \cite[\ Th.1.9.3]{ringrose1}),
where\ \ $(\alpha_j)$ \ is uniquely determined sequence of the real elements, \ and \ $(\psi_j),(\fee_j)$ \ are two orthonormal systems in $\Hs$.

For this reason \ $(\,\ast\,)\ $ holds for every \ $A\in\Fc$.
\\
In order to show that \ $B\in C_1(\Hs)$ we use the following lemma:
 \begin{lm}[Ringrose, \cite{ringrose1}, Lm.2.3.7]\label{lm_Cq_caracterization}
$\absz$\\
 Let\ \ $1\leq p\leq \infty$\,,\  and $q$ be a \textbf{conjugate exponent} of $p$ 
(i.e. $\frac{1}{p}+\frac{1}{q}=1$).

Then $T\in C_q(\Hs)$ \ if and only if \ $\sup\Big\{\big|\,\Tr[ S T]\,\big|:\ S\in\Fc,\ \|S\|_{C_p}\leq 1\Big\}<\infty$.

And in such a case the value of the supremum is equal to $\|T\|_{C_q}$.

\end{lm}
 
According to this lemma we get that

\hfill$\sup\Big\{\big|\,\Tr[ A B]\,\big|:\ B\in\Fc,\ \|B\|_{C_\infty}\leq 1\Big\}\leq\|f\|<\infty$;
\\
So that for every \ $A\in \Fc$ \ we have \ $f(A)=\Tr\big[A B\big]\,,\ \ B\in C_1(\Hs)$.
\\
Both sides in this equality are $L(\Hs)$-continuous, 
and since the set $\Fc$ is dense 
in $K(\Hs)$, \ this implies that for every \ $A\in K(\Hs)$ \ we have also \ $f(A)=\Tr\big[A B\big]\,,\ \ B\in C_1(\Hs)$.
\\
In the case when \  $A\in K_S(\Hs)$ \ therefore there exists \ $ B_0\in C_1(\Hs)$, \ such that \ $f(A)=\Tr\big[A B_0\big]$.

Consider an operator \  $B:=\dfrac{B_0+B^*_0}{2}$,\ \ which is symmetric and  trace-class;
\\
Hence: 
\\
$\Tr\big[A B\big]=\Tr\bigg[\dfrac{A(B_0+B^*_0)}{2}\bigg]=
\dfrac{1}{2}\Tr\big[A B_0\big]+\dfrac{1}{2}\Tr\big[A B_0^*\big]\\
\hspace*{14mm}=
\dfrac{1}{2}\Tr\big[A B_0\big]+\dfrac{1}{2}\Tr\big[B_0 A^*\big]=
\dfrac{1}{2}\Tr\big[A B_0\big]+\dfrac{1}{2}\Tr\big[B_0 A\big]\\
\hspace*{14mm}=
\dfrac{1}{2}\Tr\big[A B_0\big]+\dfrac{1}{2}\Tr\big[A B_0\big]=\Tr\big[A B_0\big]=f(A)$.
\\
So we can conclude that for every \ $A\in K_S(\Hs)$ \ we have \ $f(A)=\Tr\big[A B\big]$, $B=B^*\in C_1(\Hs)$.
\\
Now let us take the following sublinear functional \ $F(A):=2\,G(A)$, defined on the set $K_S(\Hs)$. 

Applying \textbf{Th.\ref{th_sublinear_F_sup_linear_f}} we have
\\
$F(A)=\sup\limits_{\theta\in\Theta}f_\theta(A)=\sup\limits_{\theta\in\Theta}\Tr\big[A B_\theta\big]\,,\ \ B_\theta=B_\theta^*\in C_1(\Hs)\,,\ \theta\in\Theta$;

Hence there exists a set of operators $ \big\{B_\theta,\, \theta\in\Theta\big|\,B_\theta=B_\theta^*\in C_1(\Hs)\big\}$, such that
$G(A)=\dfrac{1}{2}\sup\limits_{\theta\in\Theta}\Tr\big[A B_\theta\big]\,,\ 
A\in K_S(\Hs)$.

$G$ is monotone. For this reason we can only take that $B_\theta$ which are positive:
if \ $x\in\Hs$ \ then \ $\langle B_\theta x,x\rangle\geq0$.

Indeed, by the definition of supremum we have that:

$\forall\, \varepsilon>0\ \ \ \exists\, \theta=\theta_\varepsilon\in\Theta:\ \ \ 
\dfrac{1}{2}\Tr\big[A B_\theta\big]\leq G(A)<\dfrac{1}{2}\Tr\big[A B_\theta\big]+\varepsilon$.{\makebox[30mm][r]{$(\,\#\,)$}}

Assume that there exists a basis vector \ $e_i$,\ such that \ $\langle B_\theta e_i,e_i\rangle:=\beta_\varepsilon<0$.

then consider an operator \ $A_0$ \ such that \ \ $\langle A_0 e_j,e_j\rangle=\ind_{\{j=i\}}\,,\ j\geq1$.

It is clear that $A_0$ is symmetric and compact.

Also if \ $A_0\geq0$ \ yields \ $G(A_0)\geq G(0)=0$.

From $(\,\#\,)$ \ we have \ $\dfrac{1}{2}\beta_\theta\leq G(A)<\dfrac{1}{2}\beta_\theta+\varepsilon$.

And passing to the limit\ \ $\varepsilon\to0$: \ gives \ $G(A)=\dfrac{1}{2}\beta_0<0$,\ \  a contradiction.

And finally it may be concluded that \ $G(A)=\dfrac{1}{2}\sup\limits_{\theta\in\Theta}\Tr\big[A B_\theta\big]\,, \ A\in K_S(\Hs)$.

 where operators  \ $\big\{B_\theta\,,\ \theta\in\Theta\big\}$ \ satisfy required properties:

\hfill$B_\theta=B_\theta^*\geq0\,,\ B_\theta\in C_1(\Hs)$.
\\
And now we are going to finish the proof by the construction of the set $\Sigma$:

Let us define the following sets:

$\widetilde{\Sigma}:=\big\{B\in C_1(\Hs)\mid B=B^*\geq0\big\}$;

$\Sigma^\prime=\Sigma^\prime_G:=\big\{B\in \widetilde{\Sigma}\mid\forall\, \varepsilon>0\ \ 
\dfrac{1}{2}\Tr\big[A B\big]\leq G(A)<\dfrac{1}{2}\Tr\big[A B\big]+\varepsilon\big\}$;

Take $\Sigma:=\overline{\mathrm{conv}(\Sigma^\prime)}$\ \ (closure in $C_1$-norm).

Since \ $\sup\limits_{\theta\in\Theta}\Tr\big[A B_\theta\big]=\sup\limits_{B\in\Sigma^\prime}\Tr\big[A B\big]=\sup\limits_{B\in\Sigma}\Tr\big[A B\big]$, \
and \ $\Sigma\subseteq\widetilde{\Sigma}\subset  C_1(\Hs)$,

we can set  \ $G(A):=\dfrac{1}{2}\sup\limits_{B\in\Sigma}{\Tr}\big[A\cdot B],\ \  \forall\,A\in K_S(\Hs)$.

Moreover, by a such construction one can see that $G$ defines the $\Sigma$-set uniquely.

\end{proof}

\begin{rem}
For every functional \ $\ f_B(A):L(\Hs)\to\R\,,\  \ B\in C_1(\Hs)$,

 such that \
$ f_B(A)=\Tr\big[A B\big]\ \ \
 \text{we have}\ \ \ f_B\in L\big(L(H),\R\big)$. 

Moreover the mapping \ $B\to f_B$ \ defines an isometric isomorphism from $C_1(\Hs)$ to $\big(K(\Hs)\big)^*$.
\end{rem}

\begin{proof}
 $\absz$\\
We have \ $AB\in C_1(\Hs)$ \ and \ $\Tr\big[A B\big]=\|AB\|_{C_1}\leq\|A\|_{C_\infty}\|B\|_{C_1}$ \ 

\hfill (see \cite[\ Th.2.3.10]{ringrose1}).

So that, $f_B$ is a linear continuous functional.
\\
Moreover, $\|f_B\|\leq\|B\|_{C_1}$;

But according to \textbf{Lm.\ref{lm_Cq_caracterization}}:
 
$\|f_B\|\geq\sup\Big\{\big|\,f_B(A)\,\big|:\ A\in\Fc,\ \|A\|_{C_\infty}\leq 1\Big\}$

\hfill$
=
\sup\Big\{\big|\,\Tr[ A B]\,\big|:\ A\in\Fc,\ \|A\|_{C_\infty}\leq 1\Big\}=\|B\|_{C_1}$;

whence it follows that\ \ \ $\|f_B\|=\|B\|_{C_1}$. 

\end{proof}

 It is easy to see that if \ $\Sigma$ \ is a convex, closed (in $C_1(\Hs)$-topology) set of symmetrical, non-negative, trace-class operators,
then functional\\
$G(A):=\dfrac{1}{2}\sup\limits_{B\in\Sigma}{\Tr}\big[A\cdot B]$\ \ is a $G$-functional, \ $A\in K_S(\Hs)$.

Furthermore, we see that between $G$ and $\Sigma$ is settled a mutual correspondence, such that from the one we can get another one:\ \ $G\leftrightarrow\Sigma$.

\begin{prop}\label{prop_sigma_ker_g}
For the $G$-functional \ $G(A)=\dfrac{1}{2}\sup\limits_{B\in\Sigma}{\Tr}\big[A\cdot B]$ \
the set $\Sigma\equiv\mathrm{ker}\,G^*$.
\end{prop}

\begin{proof}
 $\absz$\\
Consider the indicator functional \ $f(B):=
\begin{cases}
0,\ \  B\in\Sigma\,;\\
\infty,\ \  B\in C_1(\Hs)\smallsetminus\Sigma \,.
\end{cases}$

According to the topological properties of $\Sigma$ (closed, convex and nonempty set) it follows that $\Sigma$ is a proper, convex and lower semicontinuous set.

For every \ $A\in K_S(\Hs)$ we have 

$
G(A)=\dfrac{1}{2}\sup\limits_{B\in\Sigma}\Tr\big[A B\big]=
\dfrac{1}{2}\sup\limits_{B\in C_1(\Hs)}\Big\{\Tr\big[A B\big]-f(B)\Big\}=\dfrac{1}{2}f^*(B)$,\\ 
where $f^*$ is the Legendre transform of $f$.

We thus get \ $2G=f^*$ \ and hence \ $2G^*=f^{**}=f$\ \ by the Fenchel-Moreau theorem.

Therefore \ $\Sigma=\big\{B\in\Sigma\big|\ G^*(B)=0\big\}=\mathrm{ker}\,G^*$.

\end{proof}

\subsection[Some remarks regarding the extension of the G-functional]{Some remarks regarding the extension of the\\ $G$-functional  to $L_S(\Hs)$}\label{subs_G_ext}

 In general case the extension of $G$ to the space $L_S(\Hs)$ (of linear, bounded and symmetric operators) is not unique. To see it we consider such a functional\ \ \ 
$\widetilde{G}(A):=\dfrac{1}{2}\sup\limits_{B\in\Sigma}{\Tr}\big[A\cdot B]+\rho(A)\,,\ \ \ A\in L_S(\Hs)$;

where $\rho(A)=\max\limits_{\lambda\in\sigma_{ess}(A)}\lambda$, \ the maximal point of essential spectrum $\sigma_{ess}(A)$, 
which is defined as\ \ 
$\sigma_{ess}(A)=\big\{\lambda\in\sigma(A)\mid\forall\varepsilon>0\ \ \ \exists\mu\in\sigma(A)\,:$

\hfill$ |\lambda-\mu|<\varepsilon\big\}$. 

That is $\sigma_{ess}(A)$ consists of all unisolated points of the spectrum $\sigma(A)$.
\\
If $A\in K_S(\Hs)$ \ yields that \ $\rho(A)=0$,

because in such a case every\ $\lambda\in\sigma(A)\smallsetminus \{0\}$\  is an isolated point, 

i.e. 
$\lambda\notin\sigma_{ess}(A)$ (see \cite[Claim 7.6]{conway1}).
\\
It follows that \ $\widetilde{G}(A)-G(A)=\rho(A)\cdot\ind_{\{A\in L_S(\Hs)\setminus K_S(\Hs)\}}$.

But in order to to be sure that $\widetilde{G}(A)$ is $G$-functional indeed we set the following lemma:
\begin{lm}
$\rho(A)$ is a continuous, sublinear and monotone functional.
\end{lm}

\begin{proof}
 $\absz$\\
The proof you can find in the author's thesis \cite{ibragimov1}.

\end{proof}

Now let us call a \textbf{canonical extension} of $G$ the following $G$-functional on $L_S(\Hs)$:

According to  \textbf{Th.\ref{th_charact_G_func}} the given $G$ defines $\Sigma$;

For $ A\in L_S(\Hs)$ \ we can define \ $ 
\overline{G}(A)=\overline{G}_\Sigma(A):=\dfrac{1}{2}\sup\limits_{B\in\Sigma}{\Tr}\big[A\cdot B]$, the canonical extension of $G$.

\subsection{Sublinear expectation}\label{subsec_sublin_expect}

In this chapter we discuss the notion of sublinear expectation. The material mainly was taken from \cite{peng10}, which we have generalized to infinite dimensional case.

Let $(\Omega,{\cal F},P)$ is an ordinary probability space and $(\Xs,\|\pdot\|_{\Xs})$, $(\Ys,\|\pdot\|_{\Ys})$ are normed space.

Let us define the class of Lipschitz functions with polynomial growth as follows:

\parbox{16cm}{$\Cp\,(\Xs,\Ys):=\Big\{\fee:\Xs\to\Ys\ \Big|\ \|\fee(x)-\fee(y)\|_{\Ys}
 \leq C\cdot\big(1+\|x\|^m_{\Xs}+\|y\|^m_{\Xs}\big)
\cdot\|x-y\|_{\Xs}\Big\}$.}

$\Cp\,(\Xs):=\Cp\,(\Xs,\R)$.

\begin{df}
We define the class $\H^0$ to be the linear space that satisfies the following conditions:

\parbox[t][{1.1\height}]{15cm}{
1)\ If $c\in\R$ \ then \ $c\in\H^0$;\\
2)\ If $\xi: \Omega\to\R$ is a random variable on $(\Omega,{\cal F},\P)$ \ then \ $\xi\in\H^0$;\\
3)\ If $\xi_1, \xi_2,\ldots,\xi_n \in \H^0$ \ then \ $\fee(\xi_1, \xi_2,\ldots,\xi_n)\in \H^0
$\,,\  
for every \ $\fee\in \Cp(\R^n)$.}
\end{df}

\begin{df}\label{df_H_sp} Let us set
\\ $\H:=\big\{X: \Omega\to \Xs \text{ -- r.v. on } (\Omega,{\cal F},\P) \ |\ \psi(X)\in\H^0\ \  \forall \psi\in\Cp(\Xs)\big\}$.
\end{df}

\begin{rem}
 If $X\in\H$ \ then \ $\|X\|^m\in\H^0\,,\ m\geq1$

\end{rem}

%
%

\begin{df}\label{df_subl_exp}

A functional $\E: \H^0\to\R$ is called a \textbf{sublinear expectation} if it satisfies the following conditions:

1)\ \textbf{Monotonicity}:\ \ \ if $X\geq Y$ \ then \ $\E[X]\geq \E[Y]$;

2)\ \textbf{Constant preserving}:\ \ \ $c\in\R$ \ then \ $\E[c]=c$;

3)\  \textbf{Sub-additivity}:\ \ \ $\E[X+Y]\leq\E[X]+\E[Y]$; 

4)\ \textbf{Positive homogeneity}:\ \ \ $\lambda\geq 0$ \ then \ $\E[\lambda X]=\lambda\E[X]$.

\end{df}

The triple $(\Omega, \H,\E)$ we shall call a \textbf{sublinear expectation space}.

\begin{df}
 A functional $\F_X [\fee]:=\E [\fee(X)]$ is said to be a \textbf{distribution} of $X\in\H$, \ for \ $\fee \in \Cp\,(\Xs)$.
\end{df}

\begin{prop} (Some elementary properties of $\E$)\label{prop_elem_prop}
$ $\\
1) \ \ \ \! $\E[\alpha X] = \alpha^{+}\E[ X]+\alpha^{-}\E[ -X],\ \  \Big(\alpha:=\alpha^+-\alpha^-\Big)$;

2.a) \ $-\E[X]\leq\E[-X]$;

2.b) \ $\big|\E[X]\big|\leq\E[\big|X\big|]$;

2.c) \ $\E[Z]-\E[X]\leq\E[Z-X]$;

2.d) \ $\big|\E[Z]-\E[X]\big|\leq\big|\E[Z-X]\big|\leq\E[\big|Z-X\big|]$;

3.a) \ $X\geq0\Rightarrow\ \E[X]\geq0$,

3.b) \ $X\leq0\Rightarrow\ \E[X]\leq0$;

4) \ \ \ \! $\lambda\leq0\Ar \E[\lambda X]\geq\lambda\E[X]$;

5.a) \ $c\in\R\Ar \E[X+c]=\E[X]+c$; 

5.b) \ $Y: \E[Y]=\E[-Y]=0 \Ar \E[X+Y]=\E[X]$;

5.c) \ $Y: \E[Y]=\E[-Y] \Ar \E[X+\alpha Y]=\E[X]+\alpha\,\E[Y]$;

\end{prop}

%
%
%
%
%
%
%
%
%
%
%

\begin{rem}
Note the following false implications for a sublinear expectation:
 $$\E[X+Y]=\E[X]\ \ \nRightarrow\ \ \E[Y]=0\ \ \nRightarrow\ \ \E[-Y]=0.$$

\end{rem}

\begin{prop}[Cauchy–Bunyakovsky–Schwarz inequality]\label{prop_cb_ineq}
$\absz$\\
Let $X, Y\in\H^0$, then
$$\E \big[XY\big] \leq \Big( \E \big[X^2\big]\cdot\E \big[Y^2\big] \Big)^\frac{1}{2}.$$
\end{prop}

\begin{proof}
$\absz$\\
The proof is trivial and based on the classical CBS inequality and on representation theorem for sublinear functional \textbf{Th.\ref{th_sublinear_F_sup_linear_f}}.
%
%

\end{proof}

\subsection{$G$-normal distribution}

We again refer us to Peng \cite{peng10} in order to introduce a notion of the $G$-normal distribution. 
All the definitions can be carried just from the 1-dimensional case to the infinite dimensional case.

Note that in this chapter and later on we imply that all the used random variables are defined on the sublinear expectation space $(\Omega, \H,\E)$.

\begin{df}\label{df_idd_ind}
$\absz$\\
We say that random variables $X$ and $Y$ have \textbf{identical distribution} and denote $X\sim Y$ if their distributions coincide,

i.e., for every \ $\fee\in\Cp(\Hs)\ \ \ \E \big[\fee(X)\big]=\E \big[\fee(Y)\big]$.
\\
We say that random variables $Y$ is \textbf{independent} from random variable~$X$ and denote 
$Y \dperp X$ 
if they satisfy the following equality:

$\E\big[\fee(X,Y)\big]=\E\Big[\E\big[\fee(x,Y)\big]_{x=X}\Big]\,,\ \ \fee\in\Cp(\Hs\tdot\Hs)$.

\end{df}

\begin{rem}\label{rem_sum_ind}
Peng has already mentioned that $Y \dperp X$ does not imply that $X \dperp Y$ (see \cite[Rem.3.12]{peng10}).

In the our turn when $Y \dperp X $ we put some obvious properties which we shall use later:\\
\textbf{1)} $\E[\fee_1(X)+\fee_2(Y)]=\E[\fee_1(X)]+\E[\fee_2(Y)],\ \ \ \fee\in\Cp(\Hs)$.

\textbf{2)} $\E[XY]=\E\Big[\E[x,Y]_{x=X}\Big]=\E\Big[X^+\cdot\E[Y]+X^-\cdot\E[-Y]\Big]$

\hfill $=\E[X^+]\cdot\E[Y]^++\E[-X^+]\cdot\E[Y]^-+\E[X^-]\cdot\E[-Y]^++\E[-X^-]\cdot\E[-Y]^-$.
\end{rem}

\begin{df}\label{df_g_norm}
Random variable $X$ on the $(\Omega, \H,\E)$ is said to be \textbf{$\mathbf{G}$-normal distributed}
if for every \ $\bar{X}$ which is and independent copy of\  $X$ (i.e. has identical distribution):

$$aX+b\bar{X}\sim \sqrt{a^2+b^2}\ X, \text{ where }\ a,b>0 .$$
\end{df}

Let us consider a $G$-functional defined on the $K_S(\Hs)$ in such a way:

\begin{equation}\label{eq_G_funct_exp}
G(A):=\dfrac{1}{2}\,\E\big[\langle  AX,X \rangle\big]
\end{equation}

In fact, it is clear that the sublinear expectation provides the fulfillment of all $G$-functional's properties.

Assume that $X$ is a $G$-normal distributed random variable with respect to the sublinear expectation $\E$, and $G(\pdot)$ is defined in $\eqref{eq_G_funct_exp}$.

Later (\textbf{Th.\ref{th_exist_gd}} ) we will see that for every \ $G$-functional $\widetilde{G}(\pdot)$ \ there exists \ $G$-normal distributed random variable $\widetilde{X}$ and a sublinear expectation $\widetilde{\E}$, such that 
$\widetilde{G}(A)=\dfrac{1}{2}\,\widetilde{\E}\big[\langle  A\widetilde{X},\widetilde{X} \rangle\big].$

In view of such fact we shortly make a following notation
$$X\sim N_G(0,\Sigma).$$

According to \textbf{Prop.\ref{prop_sigma_ker_g}} $G$ defines a $\Sigma$-set, which we will call a \textbf{covariance set}, about what we will discuss more precisely in the chapter \ref{subsec_covariance_set}.
 
\begin{rem}\label{rem_eq_rv_gf}
If random variables have the same distribution then their $G$-functionals coincide, i.e.: \ 
if \ $X\sim Y\sim N_G(0,\Sigma)$ \ then \ $G_X=G_Y$.
\end{rem}

%

\begin{prop}\label{prop_moments}
For a $G$-normal distributed random variable \ \mbox{$X\sim N_G( 0,\Sigma )$} \ we have the following estimation of the moments for $m\geq1$: 

$$c_m\cdot\sup\limits_{Q\in\Sigma}\Tr\big[Q^m\big]\leq\E\big[\|X\|^{2m}_\Hs\big]\leq C_m\cdot\sup\limits_{Q\in\Sigma}\Big(\Tr\big[Q\big]\Big)^m\,,$$
where $c_m$ are real constants dependent just of $m$.

In particular, if \ $m=1$ \ then \ $c_m=1$ \ and \ $\E\big[\|X\|^2_\Hs\big]=\sup\limits_{Q\in\Sigma}\Tr\big[Q\big]$.
\end{prop}

\begin{proof}
$\absz$\\
Fix \ $Q\in\Sigma$ \ and let \ $\{\lambda_i\,,\ i\geq1\}$ \ be eigenvalues of $Q$.

Consider a gaussian measure $\mu:=N_Q$.
\\
Then according to \textbf{Th.\ref{th_sublinear_F_sup_linear_f}} we have
\\
$\E\big[\|X\|^{2m}_\Hs\big]=\sup\limits_{Q\in\Sigma}E_Q\big[\|X\|^{2m}_\Hs\big]=\sup\limits_{Q\in\Sigma}\int\limits_\Hs |x|^{2m}\mu(dx)$.\hfill $(\,\ast\,)$
\\
Define $J_m:=\int\limits_\Hs |x|^{2m}\mu(dx)$
\ \  and\ \  $F(\varepsilon):=\int\limits_\Hs e^{\frac{\varepsilon}{2}|x|^2}\mu(dx)$.

Then \ $F^{(m)}(\varepsilon)=\int\limits_\Hs e^{\frac{\varepsilon}{2}|x|^2}\cdot\frac{|x|^{2m}}{2^m\phantom{x}}\mu(dx)$.

And therefore \ $J_m=2^m F^{(m)}(0)$.\hfill $(\,\ast\ast\,)$
\\
Take $\varepsilon$ small, namely $\varepsilon<\min\limits_k\frac{1}{\lambda_k}$, so we have:

$F(\varepsilon)=\int\limits_\Hs e^{\frac{\varepsilon}{2}|x|^2}\mu(dx)
=\prod\limits_{k\geq1}\int\limits_{-\infty}^{+\infty}e^{\frac{\varepsilon}{2}\,x_k^2}\, N_{\lambda_k}(dx)
=\prod\limits_{k\geq1}\int\limits_{-\infty}^{+\infty}e^{\frac{\varepsilon}{2}\,x_k^2}\frac{1}{\sqrt{2\pi\lambda_k}} 
\,e^{\frac{-x_k^2}{2\lambda_k}}\,dx\\
=\prod\limits_{k\geq1}\int\limits_{-\infty}^{+\infty}\frac{1}{\sqrt{2\pi\lambda_k}}e^{-x_k^2\big/ \big(\frac{2\lambda_k}{1-\varepsilon_k\lambda_k}\big)}\, dx
\overset{\text{Euler-Poisson integral}}{=}\prod\limits_{k\geq1}\frac{1}{\sqrt{2\pi\lambda_k}}\,\sqrt{\frac{2\lambda_k}{1-\varepsilon_k\lambda_k}}\,\sqrt{\pi}\\
=\prod\limits_{k\geq1}(1-\varepsilon_k\lambda_k)^{-\frac{1}{2}}
=\Big(\prod\limits_{k\geq1}(1-\varepsilon_k\lambda_k)\Big)^{-\frac{1}{2}}
=\Big[\det(1-\varepsilon Q)\Big]^{-\frac{1}{2}}
$.
\\
Recall the following elementary formula:

 $\Big(\prod\limits_{k\geq1}f_k\Big)^{\prime}
=f_1^{\, '} \prod\limits_{k\geq2}f_k+f_1 \Big(\prod\limits_{k\geq2}f_k\Big)^{\prime}
=\dfrac{f_1^{\, '}}{f_1} \prod\limits_{k\geq1}f_k+f_1 f_2^{\, '} \prod\limits_{k\geq3}f_k+f_1 f_2 \Big(\prod\limits_{k\geq1}f_k\Big)^{\prime}\\
=\dfrac{f_1^{\, '}}{f_1} \prod\limits_{k\geq1}f_k+\dfrac{f_2^{\, '}}{f_2} \prod\limits_{k\geq1}f_k
+\dfrac{f_3^{\, '}}{f_3} \prod\limits_{k\geq1}f_k+\ldots
=\prod\limits_{k\geq1}f_k\cdot \sum\limits_{k\geq1}\dfrac{f_k^{\, '}}{f_k}$.
\\
Therefore
$F'(\varepsilon)=\Bigg(\Big(\prod\limits_{k\geq1}(1-\varepsilon_k\lambda_k)\Big)^{-\frac{1}{2}}\Bigg)^{\prime}$

\hfill$
=-\dfrac{1}{2}\Big(\prod\limits_{k\geq1}(1-\varepsilon_k\lambda_k)\Big)^{-\frac{3}{2}}\cdot \prod\limits_{k\geq1}(1-\varepsilon_k\lambda_k)\cdot
\sum\limits_{k\geq1}\dfrac{-\lambda_k}{1-\varepsilon_k\lambda_k}\phantom{xxll}$

\hfill$
=\dfrac{1}{2}\Big(\prod\limits_{k\geq1}(1-\varepsilon_k\lambda_k)\Big)^{-\frac{1}{2}}\cdot
\sum\limits_{k\geq1}\underbrace{\dfrac{\lambda_k}{1-\varepsilon_k\lambda_k}}_{\phantom{xxxxxx}=:g_k(\varepsilon)}
=\dfrac{1}{2}\sum\limits_{k\geq1}F(\varepsilon)g_k(\varepsilon)
$.

We define \ $G(\varepsilon):=\sum\limits_{k\geq1}F(\varepsilon)g_k(\varepsilon)$. 

Then let us compute the $ m$-th derivative of $G$, we have

$G^{(m)}(\varepsilon):=\sum\limits_{k\geq1}\big(F(\varepsilon)g_k(\varepsilon)\big)^{(m)}
=\sum\limits_{k\geq1}\sum\limits_{j=0}^m{m \choose j}F^{(j)}(\varepsilon)g_k^{(m-j)}(\varepsilon)\\
$

\hfill$=\sum\limits_{k\geq1}\sum\limits_{j=0}^m\dfrac{m!}{(m-j)!j!}\,(m-j)!\,F^{(j)}(\varepsilon)\,g_k^{m+1-j}(\varepsilon)
$,

since \ \ $g_k^{(p)}(\varepsilon)=\Big(\dfrac{\lambda_k}{1-\varepsilon_k\lambda_k}\Big)^{(p)}
=\Big(\lambda_k \cdot\dfrac{\lambda_k}{(1-\varepsilon_k\lambda_k)^2}\Big)^{(p-1)}\\
=\Big(2\lambda_k \cdot\dfrac{\lambda_k^2}{(1-\varepsilon_k\lambda_k)^3}\Big)^{(p-2)}
=\ldots
=p!\cdot\dfrac{\lambda_k^{p+1}}{(1-\varepsilon_k\lambda_k)^{p+1}}
=p!\cdot g_k^{p+1}(\varepsilon)
$,

then \ \ $F^{(m)}(0)=\dfrac{1}{2}\sum\limits_{k\geq1}\sum\limits_{j=0}^{m-1}\dfrac{(m-1)!}{j!}\,F^{(j)}(0)\,g_k^{m-j}(0)$

\hfill$
=\dfrac{(m-1)!}{2}\cdot\sum\limits_{j=0}^{m-1}\dfrac{F^{(j)}(0)}{j!}\cdot\sum\limits_{k\geq1}\lambda_k^{m-j}
=\dfrac{(m-1)!}{2}\cdot\sum\limits_{j=0}^{m-1}\dfrac{F^{(j)}(0)}{j!}\cdot\Tr\big[Q^{m-j}\big]
$.

Since \ $F(0)=1$ \ it follows that \ $F^{(1)}(0)=\dfrac{1}{2}\cdot \Tr\big[Q\big]$.

Hence \ $c_1\cdot\Tr\big[Q^1\big]\leq F^{(1)}(0)\leq c_1\cdot\Tr\big[Q\big]^1$.
\\
To finish the proof we use the induction method:

Let for all \ $j<m$ \ required estimation \ $c_j\cdot\Tr\big[Q^j\big]\leq F^{(1)}(0)\leq c_j\cdot\Tr\big[Q\big]^j$ \ holds.
\\
Since \ $\Tr\big[Q^m\big]
\leq\Tr\big[Q^j\big]\Tr\big[Q^{m-j}\big]
\leq\Tr\big[Q\big]^j\Tr\big[Q^{m-j}\big]
\leq\Tr\big[Q\big]^m$,

we have \ \ $$c_m\cdot\Tr\big[Q^m\big]\leq F^{(m)}(0)\leq c_m\cdot\Tr\big[Q\big]^m.$$

But according to \  $(\,\ast\,)$ \ and \  $(\,\ast\ast\,)$ \ this gives us what we need.

\end{proof}

\begin{rem}
If we consider a canonical extension on $L_S(\Hs)$ for a such particular $G(\pdot)=\dfrac{1}{2}\,\E\big[\langle  \pdot X,X \rangle\big]$,
there emerges a question if we are able to show that such $G$ satisfies the representation given in \textbf{Th.\ref{th_charact_G_func}},

i.e. $\exists$ certain set $\Sigma$, s.t.\ \ \  $G(A)=\dfrac{1}{2}\sup\limits_{B\in\Sigma}{\Tr}\big[A\cdot B\big]\,,\ \ \ A\in L_S(\Hs)$?

\end{rem}

According to (\textbf{\ref{subs_G_ext}}) in the general case we do not have uniqueness.

But, it is clear that for every \ $A\in L_S(\Hs)$ \ there exist $A_n\in K_S(\Hs)$ such that \ \ 
$\Big|\langle  AX,X \rangle-\langle  A_nX,X \rangle\Big|\xrightarrow[n\to\infty]{}0$.{\makebox[80mm][r]{$(\,\ast\,)$}}

To show this we 
take a projection operator \ 
$P_n x:=\sum\limits_{i=1}^n \langle x,e_i\rangle e_i$,\ 

where \ $x\in \Hs$ \ and \  $(e_i)$ \ is a basis of \ $\Hs$.

Taking \  $A_n:=\dfrac{A P_n+P_n A}{2}$  which is actually compact (since has a finite range) and symmetric.
\\
And $(\,\ast\,)$ may be concluded from the following convergence and Cauchy-Bunyakovsky-Schwarz inequality:

$\|Ax-A_n x\|_\Hs
\leq \dfrac{1}{2}\Big(\|Ax-A P_n x\|_\Hs+\|Ax-P_n A x\|_\Hs\Big)\\
\leq\dfrac{1}{2}\Big(\|A\|_{L(\Hs)}\cdot\|x-P_n x\|_\Hs+\|Ax-P_n A x\|_\Hs\Big)\\
\leq\dfrac{1}{2}\Big(\|A\|_{L(\Hs)}\cdot\|\underbrace{\sum\limits_{i=n+1}^\infty \langle x,e_i\rangle e_i}_{\searrow0}\|_\Hs
+\|\underbrace{\sum\limits_{i=n+1}^\infty \langle Ax,e_i\rangle e_i}_{\searrow0}\|_\Hs\Big)
\xrightarrow[n\to\infty]{}0$.
\\
So, the question remains just to understand under which conditions the following expression tends to zero:
\\
$\Big|\E\big[\langle AX,X\rangle\big] - \sup\limits_{B\in\Sigma}{\Tr}\big[A_n\cdot B\big]\Big|
=\Big|\E\big[\langle AX,X\rangle\big] - \E\big[\langle A_n X,X\rangle\big] \Big|$

\hfill$\leq\E\Big[\Big|\langle AX,X\rangle-\langle A_n X,X\rangle\Big| \Big]
\xrightarrow[n\to\infty]{?}0$.
\\
\subsection{Covariance set under sublinear expectation}\label{subsec_covariance_set}

Let us describe the notion of the covariance operators for a random variable $X\sim N_G(0,\Sigma)$ under the sublinear expectation $\E$. Actually, we will see that $\Sigma$ is a set of operators.
\\
If $\E$ is linear (denote it as $E$), then:\\
$Cov(X)=Q$, where $Q$ is defined as: $\left\langle Q h,k \right\rangle=E[\left\langle X,h \right\rangle\left\langle X,k \right\rangle]$.

Now we fix the family of linear functionals $\Big\{E_{\theta},\enskip \theta\in\Theta\Big\}$ for given sublinear functional $\E$, and let $\{e_i,i\geq1\}$ be a basis in $\Hs$. Then we have:
$$\E\big[\langle A X,X \rangle\big]=\E\Big[ \sum_{i\geq1}\langle A X,e_i \rangle\langle X,e_i \rangle\Big]=
\sup\limits_{\theta\in\,\Theta}\,E_\theta\Big[ \sum_{i\geq1}\langle A X, e_i \rangle\langle X,e_i \rangle\Big]
$$

We can change the order of integration in the last term. For check it we formulate the following lemma.

\begin{lm}\label{lm_change_integration}
 For $A$ a linear bounded operator and square-integrable r.v.~$X$ under the linear expectation $E$, i.e. 
$E\big[\|X\|^2_\Hs\big]<\infty$, the following equality holds:

$$E\Big[ \sum_{i\geq1}\langle A X, e_i \rangle\langle X,e_i \rangle\Big]
= \sum_{i\geq1}E\Big[\langle A X, e_i \rangle\langle X,e_i \rangle\Big],$$
where $\{e_i,i\geq1\}$ is a basis in Hilbert space $\Hs$.

\end{lm}

\begin{proof}
 $\absz$\\
We know that we can change the finite sums with a linear expectation, i.e.

$E\Big[ \sum\limits_{i=1}^N\langle A X, e_i \rangle\langle X,e_i \rangle\Big]
= \sum\limits_{i=1}^N E\Big[\langle A X, e_i \rangle\langle X,e_i \rangle\Big]
$.

So, we need to take just limits when $N\to\infty$ from the left and right part respectively, and see that they coincide.
\\
\textbf{(a)}\ \ Define $S_N$ as a partial sum of the given series.

$S_N:=\sum\limits_{i=1}^N\langle A X, e_i \rangle\langle X,e_i \rangle$.

We have that for every $\omega\in\Omega\ \ \ S_N\xrightarrow[N\to\infty]{}\sum\limits_{i=1}^\infty\langle A X, e_i \rangle\langle X,e_i \rangle
=\langle AX,X\rangle$.

Take the projection operator $P_N:=Proj(e_1,..,e_N)$, that is it is noting else but \ 
$P_N=P^*_N=P_N^2$\ \  and\ \ $\|P_N\|_\Hs=1$.
\\
Then we have

$S_N
=\sum\limits_{i=1}^N\langle A X, e_i \rangle\langle X,e_i \rangle
=\sum\limits_{i=1}^N\langle A X, P_N e_i \rangle\langle X,P_N e_i \rangle 
=\sum\limits_{i=1}^\infty\langle A X, P_N e_i \rangle\langle X,P_N e_i \rangle\\
=\sum\limits_{i=1}^\infty\langle P_N A X,  e_i \rangle\langle P_N X, e_i \rangle
=\langle P_N A X, P_N X \rangle
=\langle P^2_N A X, X \rangle
=\langle P_N A X, X \rangle
$.
\\
And we get that \ \ 
$|S_N|\leq\|P_N\|_{L(\Hs)}\cdot\|A\|_{L(\Hs)}\cdot\|X\|^2_\Hs
=\|A\|_{L(\Hs)}\cdot\|X\|^2_\Hs\in L^1(P)
$,

 where $P$ is a probability for the integral that $E$.

And it follows by the dominated convergence theorem that 
$$E\big[S_N]\xrightarrow[N\to\infty]{}E\big[\langle AX,X\rangle].$$
\\
\textbf{(b)}\ \ Let $Q$ be a covariance set of operators of $X$ ($Q$ is a trace-class operator), then\ \ \
$\sum\limits_{i=1}^N E\Big[\langle A X, e_i \rangle\langle X,e_i \rangle\Big]
=\sum\limits_{i=1}^N \langle Q A^* e_i, e_i\rangle
$.

The last term has a limit\ \ \ $\sum\limits_{i=1}^\infty \langle Q A^* e_i, e_i\rangle\equiv\Tr\big[QA^*]$, 

because 
$\sum\limits_{i=1}^\infty \big|\langle Q A^* e_i, e_i\rangle\big|<\infty$\,,\ \ since\ \ $QA^*$\ \ is also a trace-class operator.

So that, 

\hfill$\sum\limits_{i=1}^N E\Big[\langle A X, e_i \rangle\langle X,e_i \rangle\Big]\xrightarrow[N\to\infty]{}
\sum\limits_{i=1}^\infty \langle Q A^* e_i, e_i\rangle
=\sum\limits_{i=1}^\infty E\Big[\langle A X, e_i \rangle\langle X,e_i \rangle\Big]
$.

\end{proof}

Surely, we can use \textbf{Lm.\ref{lm_change_integration}} in our case, because by \textbf{Prop.\ref{prop_moments}} and \textbf{Th.\ref{th_sublinear_F_sup_linear_f}} we get that 
$$\infty>\E\big[\|X\|^2_\Hs\big]=\sup\limits_{\theta\in\,\Theta}E_\theta\big[\|X\|^2_\Hs\big].$$

It means that, 

$\E\big[\langle A X,X \rangle\big]
=\sup\limits_{\theta\in\,\Theta} \sum\limits_{i\geq1}E_\theta\Big[\langle AX, e_i \rangle\langle X,e_i \rangle\Big]=
\sup\limits_{\theta\in\,\Theta} \sum\limits_{i\geq1}\langle Q_\theta A e_i,e_i \rangle$

\hfill$=
\sup\limits_{\theta\in\,\Theta}\ \Tr[Q_\theta\cdot A].
$

In the same manner we get that
\\
$\E\big[\langle X,h \rangle\langle X,k \rangle\big]=
\E\Big[ \sum\limits_{i\geq1}\langle X,e_i \rangle\langle h,e_i \rangle\sum\limits_{j\geq1}\langle X,e_j \rangle\langle k,e_j \rangle\Big]$

\hfill$=\sup\limits_{\theta\in\Theta}\,E_\theta\Big[ \sum\limits_{i,j\geq1}\langle X,e_i \rangle\langle h,e_i \rangle\langle X,e_j \rangle\langle k,e_j \rangle\Big]$

Using the same idea as in the proof of \textbf{Lm.\ref{lm_change_integration}} permits us also to change the integration sums and we obtain that
\\
$
\E\big[\langle X,h \rangle\langle X,k \rangle\big]=\sup\limits_{\theta\in\Theta}\sum\limits_{i,j\geq1}\,E_\theta\Big[ \langle X,e_i \rangle\langle X,e_j \rangle\Big]\langle h,e_i \rangle\langle k,e_j \rangle$

\hfill$=
\sup\limits_{\theta\in\Theta}\sum\limits_{i,j\geq1}\,\langle Q_\theta e_i,e_j \rangle\langle h,e_i \rangle\langle k,e_j \rangle=
\sup\limits_{\theta\in\Theta}\,\langle Q_\theta h,k\rangle.
$

So, we have that
\begin{equation}
\Sigma:=\{Q_\theta\mbox{ -- covariation of $X$ under $E_\theta$}\,,\ \theta\in\Theta\}=Cov(X) 
\end{equation}
\begin{equation}\label{eq_cov_gnd}
\boxed{\E\big[\langle X,h \rangle\langle X,k \rangle\big]=\sup_{Q\in\Sigma}\,\langle Q h,k\rangle\!\!\!\!\!\phantom{\dfrac{1}{2}}}
\end{equation}
 
\begin{rem} In order to better understand the nature of the covariance set of operators under sublinear expectation, we list the following obvious properties:
 $\phantom{.}$\\
\textbf{1)}$\E\big[-\langle A X,X \rangle\big]=\sup\limits_{Q\in\Sigma}\Tr[-AQ]
=\sup\limits_{Q\in\Sigma^{\barw}}\Tr[AQ]$,

\hfill where $\Sigma^{\,\barw}:=\Big\{-A \bigl|\ A\in\Sigma\Bigr.\Big\}$.

\textbf{2)} $\E\big[-\langle A X,X \rangle\big]\geq-\E\big[\langle A X,X \rangle\big]$.
\\
\textbf{3)} $\E\big[-\langle X,h \rangle\langle X,k \rangle\big]=\sup\limits_{Q\in\Sigma}\langle -Qh,k \rangle
=\sup\limits_{Q\in\Sigma^{\barw}}\langle Qh,k \rangle$.

\textbf{4)} $X_1\sim N_G(0,\Sigma),\ X_2\sim N_G(0,\Sigma^{\,\barw})\ \ \Rightarrow\ \ 
G_{X_1}(A)=G_{X_2}(-A)$.
\end{rem}

Also we can show that one-dimensional projection of the $G$-normal distributed random variable in the Hilbert space is also $G$-normal distributed.
 \begin{prop}\label{prop_proj_gn}
 Let $X$ be a $G$-normal distributed random variable in the Hilbert space $\Hs$, then for every \ \ $h\in\Hs$\ \ \ 
$\left\langle X,h\right\rangle$ \ is $G_h$-normal~distributed, 

where
\parbox[t]{10cm}{
$G_h(\alpha)=\dfrac{1}{2}\big(\alpha^+\,\overline{\sigma}^2(h)-\alpha^-\,\underline{\sigma}^2(h)\big);$\\
$\overline{\sigma}^2(h)=\E[\left\langle  X,h \right\rangle^2]=2 G(h\cdot h^T),\\ \underline{\sigma}^2(h)=-\E[-\left\langle  X,h \right\rangle^2]=-2 G(-h\cdot h^T)$.}

So, we keep the notation and can write that $X\sim N_G( 0,[\underline{\sigma}^2(h), \overline{\sigma}^2(h)] )$.
\end{prop}

\begin{proof}
$\absz$\\
The proof you can find in the author's thesis \cite{ibragimov1}.
%
%
%
%

\end{proof}
 
\begin{rem} We also recall the following obvious fact settled in the finite-dimensions:
\\
if\ \ $X\in\mathcal{H}^0$ \ \ then\ \ $\E[X^2]:=\overline{\sigma}^2\geq\underline{\sigma}^2=:-\E[-X^2]\geq0$.
\end{rem}
 
Some algebraical properties of the $G$-normal distributed random variables are listed below.

\begin{prop}\label{prop_ar_op_gnd}
$\absz$\\
1) Let \ $X\sim N_G( 0,\Sigma )$ \ then \ $a X\sim N_G( 0,a^2\Sigma )\,,\ a\in\R$;

2) Let \ $Y=X_1+X_2$\,, \ where \ $X_i\sim N_G( 0,\Sigma_i )$ \ are reciprocally independent,

then \ $Y\sim N_G( 0,\Sigma_Y)$ \ with a covariance set \ 

\hfill$\Sigma_Y:=\big\{Q_1+Q_2\big|\, Q_i\in\Sigma_i\,,\ i=\overline{1,2}\big\}$;

3) Let \ $Z=SX\,,\ S\in L(\Us,\Hs)\,,\ Z\in\Hs\,,\ X\in\Us\,,\ X\sim N_G( 0,\Sigma )$ \ 
then \ $Z\sim N_G( 0,\Sigma_Z)$ \ with a covariance set $\Sigma_Z:=\big\{SQS^*\big|\, Q\in\Sigma\big\}.\phantom{////}$
\end{prop}
 
\begin{proof}
$\absz$\\
The proof you can find in the author's thesis \cite{ibragimov1}.

\end{proof}

\newpage
\section{Viscosity solutions}

In this chapter we describe the notion of the viscosity solution for a fully nonlinear infinite-dimensional parabolic PDEs. Mainly the material (definitions and results) was taken from Kelome \cite{kelome1}. In infinite-dimensions for a viscosity solutions Kelome uses a particular notion of $B$-continuity which we also describe below. We apply his results of comparison principle and uniqueness of viscosity solution to our theory where in the following chapters we will solve parabolic PDEs in infinite dimensions with a probabilistic tools of sublinear expectation. 

\subsection{B-continuity}

Consider a fully nonlinear infinite-dimensional parabolic PDE:
\begin{equation}\tag{P}\label{eq_p}
\hspace{20mm}\begin{cases}
\partial_t u+\langle Ax,D_x u\rangle+G(D^2_{xx} u)=0 \,, \ \ \ t\in[0,T)\,, \ x\in\Hs;\\
u(T,x)=f(x) \,.
\end{cases}
\end{equation}

$u:[0,T]\times\Hs\to\R$;

$f\in\Cp(\Hs)$;

$G:L_S(\Hs)\rightarrow\R$ \ is a canonical extension of a $G$-functional defined 
on $K_S(\Hs)$ and denoted by the same symbol $G$;

$A:D(A)\to\Hs$ is a generator of $C_0$-semigroup $\big(e^{tA}\big)$.
\\
Recall that \ $\Cp$ \ is a space of Lipschitz functions with polynomial growth (see \ref{subsec_sublin_expect}), and \ $K_S(\Hs)$ \ is a space of compact symmetric operators (see \ref{subsec_G_funct}).
\\
We also impose the following condition on the operator $A$:
\begin{cond_}
There exists \ $B\in L_S(\Hs)$ such that:

\hspace{5mm}1)$B>0$;

\hspace{5mm}2)$A^*B\in L(\Hs)$;

\hspace{5mm}2)$-A^*B+c_0 B\geq I$,\ \ for some $c_0>0$.
\end{cond_}

\begin{rem}
The  $im ( B)$ should belong to the set $D(A^*)$. If it happens that $D(A^*)\subset\Hs$ compactly, then it is necessarily 
$B$ be a compact operator.
\end{rem}

\begin{proof}
 $\absz$\\
In fact, let $\{x_n\,,\ n\geq1\}\subset\Hs$ and $\|x_n\|_{\Hs}\leq c\ \ \forall n\geq1$.

So then \ $\|A^*Bx_n\|_{\Hs}\leq\|A^*B\|_{L(\Hs)}\cdot\|x_n\|_{\Hs}\leq c\cdot\|A^*B\|_{L(\Hs)}$.

Since we assume that $D(A^*)$ is a compact embedding in $\Hs$, we have

$\{Bx_n\,,\ n\geq\}$ is bounded in $D(A^*)$.

Thus there exists a subsequence \ $\{x_{n_k}\,,\ k\geq1\}$, \ such that \ 
$\{Bx_{n_k}\,,\ n\geq1\}$ is convergent in $\Hs$.

And we conclude that \ $ B\in K(\Hs)$.

\end{proof}

\begin{rem}
If $A$ is self-adjoint, maximal dissipative operator then we can take $B:=(I-A)^{-1}$ with $c_0:=1$ which satisfies the condition imposed above.

Usually, in applications $A=\Delta$, so such condition for finding the correspondent $B$ is not too much strict.
\end{rem}

Later we need a space $\Hs_{-1}$ which is defined to be a completion of $\Hs$ under the norm \
 $\|x\|_{-1}^2:=\langle Bx,x\rangle=\langle B^{\frac{1}{2}}x,B^{\frac{1}{2}}x\rangle=\|B^{\frac{1}{2}}x\|_{\Hs}^2$.

Fix \ $\{\widetilde{e}_j\,,\ j\geq1\}$ to be a basis of $\Hs_{-1}$ made of elements of $\Hs$.

(Hence in such a  case \ $\{B^{\frac{1}{2}}\widetilde{e}_j\,,\ j\geq1\}$ \ is a basis of $\Hs$).

Define $\Hs_N:=span\{\widetilde{e}_1,..,\widetilde{e}_N\}\,,\ N\geq1$.

And let \ $P_N$ \ be an orthonormal projection $\Hs_{-1}$ onto $\Hs:$

$P_N x:=\sum\limits_{j=1}^N\widetilde{e}_j\langle x,\widetilde{e}_j\rangle_{-1}\,, \ x\in\Hs_{-1}$.

Also we define the following operator \ $Q_N:=I-P_N$.

\begin{df}
 Let $u,\,v:[0,T]\times\Hs\to\R$.

$u$ is said to be $B$-l.s.c. ($B$-lower semicontinuous) 

\hspace{20mm}if \ $u(t,x)\leq\varliminf\limits_{n\to\infty}u(t_n,x_n)$;

And $v$ is said to be $B$-u.s.c. ($B$-upper semicontinuous) 

\hspace{20mm}if \ $u(t,x)\geq\varlimsup\limits_{n\to\infty}u(t_n,x_n)$,

\hfill whenever $x_n\xrightarrow{w}x\,,\ t_n\to t\,,\ Bx_n\xrightarrow{s}Bx$.
\end{df}

\begin{df}
 The function which is $B$-l.s.c. and $B$-u.s.c. simultaneously is called $B$-continuous.
\end{df}

\begin{rem}
 Note that $B$-continuity means that function $u(t,x)$ is continuous on the bounded sets  of $[0,T]\times\Hs$ for the $[0,T]\times\Hs_{-1}$-topology.
\end{rem}

\begin{df}
The function $u(t,x)$ is locally uniformly $B$-continuous if it is uniformly continuous on the bounded sets  of $[0,T]\times\Hs$ for the $[0,T]\times\Hs_{-1}$-topology.
\end{df}

\begin{rem}
In some cases $B$ is a compact operator.

If it is so then from the convergence \ $x_n\xrightarrow{w}x$ \ it follows that \ $Bx_n\xrightarrow{s}Bx$.

And notions ``$B$-continuity``, ``locally uniformly $B$-continuity`` and ``weak continuity'' are the same.
\end{rem}

\subsection{Test functions and viscosity solutions}

\begin{df}
 The function $\psi:(0,T)\times\Hs\to\R$ is said to be a test function if it admits the representation $\psi=\fee+\chi$, such that:

\textbf{1)} $\fee\in C^{1,\,2}\Big((0,T)\times\Hs\to\R\Big)$;

\hspace{6mm}$\fee$ is $B$-continuous;

\hspace{6mm}$\Big\{\partial_t\fee,\,A^*D_x\fee,\,D_x\fee,\,D^2_{xx}\fee\Big\}$ are locally uniformly continuous on 

\hfill$(0,T)\times\Hs$;

\textbf{2)} $\chi:(0,T)\times\Hs\to\R$ \ and has the following representation 
 
\hspace{6mm}$\chi(t,x)=\xi(t)\cdot\eta(x)$, such that:
\\
\hspace*{6mm}$\xi\in C^1\Big((0,T)\to(0,\,+\infty)\Big)$;

\hspace{6mm}$\eta\uparrow,\ \eta\in C_p^2(\Hs)$ -- i.e. the derivatives have polynomial growth: 

\hfill$\|D\eta\|_\Hs, \|D^2\eta\|_{L(\Hs)}\leq C\big(1+|x|^m\big)$;

\hspace{6mm}$\eta(x)=\eta(y)\ \ \text{ whenever } |x|=|y|$;

\hspace{6mm}$\Big\{D\eta,\,D^2\eta\Big\}$ are locally uniformly continuous on $(0,T)\times\Hs$ and have 

\hfill polynomial growth.

\end{df}

\begin{df}\label{df_test_func}
 Let $u,\,v:[0,T]\times\Hs\to\R$.

$u$ is said to be a \textbf{viscosity subsolution} of \eqref{eq_p} at the point $(t_0,x_0)$ if:

\hspace{5mm}1) $u$ is $B$-u.s.c. on $[0,T]\times\Hs$;

\hspace{5mm}2) for every test function $\psi$:

\hspace{10mm}$u\leq\psi$;

\hspace{10mm}$u(t_0,x_0)=\psi(t_0,x_0)$;

\hspace{10mm}$\Big[\partial_t\psi+\langle x,A^*D_x\fee\rangle+G(D^{2}_{xx}\psi)\Big](t_0,x_0)\geq0$;

\hspace{10mm}$u(T,x)\leq f(x)$.
\\
Analogously, $v$ is said to be a \textbf{viscosity supersolution} of \eqref{eq_p} at the point $(t_0,x_0)$ if:

\hspace{5mm}1) $u$ is $B$-\textcolor{Blue}{l.}s.c. on $[0,T]\times\Hs$;

\hspace{5mm}2) $\forall$test function $\psi$:

\hspace{10mm}$v\textcolor{Blue}{\geq}\psi$;

\hspace{10mm}$v(t_0,x_0)=\psi(t_0,x_0)$;

\hspace{10mm}$\Big[\partial_t\psi+\langle x,A^*D_x\fee\rangle+G(D^{2}_{xx}\psi)\Big](t_0,x_0)\textcolor{Blue}{\leq}0$;

\hspace{10mm}$v(T,x)\textcolor{Blue}{\geq} f(x)$.
\end{df}

\begin{df}
 The function which at the point $(t_0,x_0)$ is viscosity sub- and supersolution simultaneously is called viscosity solution.
\end{df}

\begin{rem}\label{rem_add_cond_test_func}
Note that in the definition we imply that $D_x\fee\in D(A^*)$.

\end{rem}

\begin{rem} Actually the functional $G$ in equation \eqref{eq_p} can be considered only on a compact set of operators. Because when we solve this equation the
compactness of the operator $D^2_{xx}\psi$ of test function is constrained only by such a thing that we are looking for only such functions as solutions which have compact second Fréchet derivative, i.e. on the domain of $G$-functional. 
This fact is subjected to only the above described requirement. In fact, the functions $\fee$ and $\chi$ are built in the following way (see \cite[p.14]{kelome1}):
\\
we take a test function \ $\widetilde{\psi}=\widetilde{\fee}+ \widetilde{\chi}$ defined on a \   $(0,T)\times\widetilde{\Hs}_N$ \ and \ $\widetilde{\fee}\,,\ \widetilde{\chi}$ \ are bounded.
$\widetilde{\Hs}_N$ is defined as a space $\Hs_N$ with $\Hs_{-1}$-topology.
Note that $dim \Hs_N<\infty$.
And for a test function we take 
\parbox[t]{90mm}{$\fee(t,x):=\widetilde{\fee}(t, P_Nx);\\
\phantom{.}\chi(t,x):=\widetilde{\chi}(t, P_Nx)$.}

It is clear that such $\fee$ and $\chi$ have a compact second derivative, so a test function \ $\psi=\fee+\chi$ \ satisfies required condition.

\end{rem}

\subsection{Comparison principle}

The following result is called a comparison principle which is obtained by Kelome and \'{S}wi\c{e}ch (original formulation one can see in \cite[Th.3.1]{kelome1}).

\begin{teo}[Comparison principle]\label{th_comp_princ}
$\absz$\\
 Let $u$ and $v$ be respectively sub- and super- viscosity solutions of the \eqref{eq_p}, such that:

\hspace{3mm} 1) there exists a positive $M$, such that 

\hspace{9mm} $u\leq M$,

\hspace{9mm} $v\geq -M$;

\hspace{3mm} 2) $f$ is bounded, locally uniformly $B$-continuous;

\hspace{3mm} 3) $G$satisfies the following conditions:

\hspace{9mm} (i) if \ $A_1\geq A_2$ \ then \ $G(A_1)\geq G(A_2)$;

\hspace{9mm} (ii) there exists a radial, increasing, linearly growing function 

\hspace{17mm} $\mu\in C^2(\Hs\to\R)$  with bounded first and second derivatives, such 

\hspace{17mm} that for all \ $\alpha>0\,:\ \ \ \Big|G\big(A+\alpha D^2\mu(x)\big)-G\big(A\big)\Big|\leq C \big(1+|x|)\cdot \alpha$;
\\
\hspace*{9mm} (iii) $\sup\Big\{\,\Big|G\big(A+\lambda B Q_N\big)-G\big(A\big)\Big|:
\ \|A\|_{L(\Hs)}<p,\,|\lambda|<p,$

\hfill$ A=P^*_N A P_N\,,\ p\in\R\Big\}\xrightarrow[N\to\infty]{}0$.
\\
Then $u\leq v$.
\end{teo}

\begin{rem}
 If $G$ is a $G$-functional then condition \textbf{3)} of \textbf{Th.\ref{th_comp_princ}} holds.
\end{rem}

\begin{proof}
 $\absz$\\
The proof you can find in the author's thesis \cite{ibragimov1}.

\end{proof}

The following proposition is a comparison principle for the functions with polynomial growth. In fact, we will use this result applied to the functions of the $\Cp$ class.

\begin{prop}\label{prop_comp_princ_polyn_growth}
 Let $u$ is a sub- and $v$ is a super- viscosity solution of the PDE  \eqref{eq_p}, such that for every $t\in[0,T]$ and for every $x\in\Hs$ there satisfy:

\ \ \ \parbox[c]{120mm}{
$\phantom{-}u(t,x)\leq C(1+|x|^m);\\
-v(t,x)\leq C(1+|x|^m);\\
\phantom{-,}|f(x)|\leq C(1+|x|^m).
$
}$(\,\#\,)$
\\
And the following conditions hold:

\hspace{10mm} (i) there exists continuous \ $\delta_\varepsilon:[0,T]\to[0,\infty)$, such that:

\hspace{17mm} at the every point $t\in[0,T]$: \ $\delta_\varepsilon(t)\xrightarrow[\varepsilon\downarrow0]{}0$;

\hspace{10mm} (ii) there exists a radial, increasing, function \ $g_0$, such 

\hspace{19mm} that:

\hspace{19mm}  $g_0, Dg_0, D^2 g_0$ are locally uniformly continuous,

\hspace{18mm} $\varliminf\limits_{|x|\to\infty}\dfrac{g_0(x)}{|x|^{m+1}}>0$;

\hspace{10mm} (iii) $\delta^\prime_\varepsilon(t)\cdot g_0(x)+\Big|G\big(A+\delta_\varepsilon(t)\cdot D^2g_0(x)\big)-G\big(A\big)\Big|\leq0$.

Then the functions

\hspace{5mm}$u_\varepsilon(t,x):=u(t,x)-\delta_\varepsilon(t)\cdot g_0(x)$;

\hspace{5mm}$v_\varepsilon(s,y):=v(s,y)+\delta_\varepsilon(s)\cdot g_0(y)$;

are respectively sub- and super- viscosity solutions of the PDE  \eqref{eq_p}, such that $u_\varepsilon\leq v_\varepsilon$. 
\\
Moreover, if \ $\varepsilon\to0 $ \ then \ $ u\leq v$. 
\end{prop}

\begin{proof}
 $\absz$\\
The proof you can find in the author's thesis \cite{ibragimov1}.

\end{proof}

\subsection{Uniqueness of viscosity solution}

\begin{teo} \label{th_uniq_vs}
If PDE  \eqref{eq_p} has a $B$-continuous viscosity solution on~$[0,T]\tdot\Hs$ that has a polynomial growth, then such a solution is unique.
\end{teo}

\begin{proof}
 $\absz$\\
The proof you can find in the author's thesis \cite{ibragimov1}.

\end{proof}

\newpage
\section{G-expectations}

In this chapter we describe a main notion of the theory, the notion of a $G$-expectation, the special case of a sublinear expectation. It was introduced by Peng \cite{peng10} and we by analogy carry it to the infinite-dimensions. Also we touch upon such objects as $G$-Brownian motion, capacity and upper expectation.

\subsection{G-Brownian motion}

\begin{df}
 $X_t:\Omega\tdot\R_+\to\Xs$ is called a stochastic process 
if $X_t$ is random variable on a sublinear expectation space $(\Omega,\H,\E)$ for every nonnegative $t$.
\end{df}

\begin{df}
 Stochastic process is called a G-Brownian motion if:\\

1) \ $B_0=0$;\\
2) \ for all \ $t,s\geq0 \ \ \ ( B_{t+s}-B_t)\sim N_G( 0,s\Sigma  )$;\\
3) \ for all \ $t,s\geq0 \ \ \ ( B_{t+s}-B_t)$ \ is independent from \ $(B_{t_1},..,B_{t_n} )$ 

\hfill for every \ $n\in \N,\ 0\leq t_1\leq \ldots\leq t_n\leq t$ . 

\end{df}
 
\begin{rem}\label{rem_bm_ind}
 $\big(B_{t_{k+1}}-B_{t_k}\big)$ is independent from 
$\fee\big(B_{t_1},..,B_{t_k}\big)$,

\hfill $0\leq t_1\leq \ldots\leq t_k\leq t_{k+1}\,,\ \fee\in \Cp\,(\Xs^k\rightarrow\R).\phantom{........}$
\end{rem}

%
%
%

\begin{prop}\label{prop_gbm}
 Let $B_t$ be a stochastic process and:

\hspace{30pt}$B_t\sim\sqrt{t}B_1$;

\hspace{30pt}$B_1$ is $G$-norm. distributed.

Then $B_1$ is $G$-Brownian motion
\end{prop}

\begin{proof}
 $\absz$\\
\textbf{1)}$B_0=0$;

\textbf{2)}$B_{t+s}-B_t\sim\sqrt{t+s}B_1-\sqrt{t}B_1\sim\sqrt{s}B_1\sim N_G\Big(0,\ s\Sigma\Big)$;

\textbf{3)}$B_{t+s}-B_t\sim\sqrt{s}B_1$;

\hspace{12pt}$\big(B_{t_1},\ldots,B_{t_n}\big)\sim\big(\sqrt{t_1}B_{1}^{(1)},\ldots,\sqrt{t_n}B_{1}^{(n)}\big)$, such that $B_1,\, B_{1}^{(1)},\ldots,B_{1}^{(n)}$ are independent copies.

So that \ $ B_1 \text{ is independent form }\big(B_{1}^{(1)},\ldots,B_{1}^{(n)}\big)$.

Therefore \ $ \big(B_{t+s}-B_t\big) \text{ is independent form }\big(B_{t_1},\ldots,B_{t_n}\big)$.

\end{proof}

\begin{coroll}\label{cor_Bth_bm}
 $B_t^h:=\langle B_t, h\rangle$ is a 1-dimensional G-B.m.
\end{coroll}
\begin{proof}
 $\absz$\\
 It follows from \textbf{Prop.\ref{prop_gbm}} and \textbf{Prop.\ref{prop_proj_gn}}.

\end{proof}

\begin{rem}\label{rem_lin_gf_gbm}
If $B_t$ is a $G$-Brownian motion then \ $G(A)=\dfrac{1}{2}\,\E\big[\langle  A B_1,B_1 \rangle\big]$.

Moreover:

\hspace{2cm}\textbf{1)} $\dfrac{1}{2}\,\E\big[\langle  A B_t,B_t \rangle\big]=t\,G(A)$.

\hspace{2cm}\textbf{2)} for every \ $k,r \ \ \ \E\big[B_t^k\cdot B_t^r\big]=t\,\E\big[B_1^k\cdot B_1^r\big]$.
\end{rem}

\subsection{Capacity and upper expectation}

The notion of capacity, i.e. a supremum measure, was introduced by Choquet \cite{choquet1}. Since in our framework there are no fixed probability measures, so we will use exactly this object. The classical notion ``almost surely'' becomes ``quasi surely''.

So, let $\big(\Omega,\mathcal{B}(\Omega)\big)$ be a complete separable metric space with Borel \mbox{$\sigma$-algebra} on it.
$\mathcal{M}$ is the collection of all probability measures on $\big(\Omega,\mathcal{B}(\Omega)\big)$.
Take a fixed $\mathcal{P}\subseteq\mathcal{M}$.
\begin{df}
Let us define capacity to be
\begin{equation}
 c(A)=c_{\mathcal{P}}(A):=\sup\limits_{P\in\mathcal{P}}P(A),\ \  A\in\mathcal{B}(\Omega).
\end{equation}
\end{df}

\begin{rem}
It it obvious that $c(A)$ is a Choquet capacity (see \cite{choquet1}), i.e.

\hspace{10mm}\parbox{150mm}{
\begin{enumerate}[(1)]
  \item $c(A)\in[0,1],\ \  A\in\Omega$.
  \item if \ $A\subset B$ \ then \ $c(A)\leq c(B)$.
  \item if \ $\big\{A_n,\ n\geq1\big\}\subset\mathcal{B}(\Omega)$ \ then \ $c(\bigcup\limits_{n\geq1}A_n)\leq \sum\limits_{n\geq1}c(A_n)$.
  \item if \ $\big\{A_n,\ n\geq1\big\}\subset\mathcal{B}(\Omega):\ A_n\uparrow A=\bigcup\limits_{n\geq1}A_n$\,, \ then  \ $c(\bigcup\limits_{n\geq1}A_n)=\lim\limits_{n\to\infty}c(A_n)$.
\end{enumerate}}
\end{rem}

\begin{df}
 The set $A$ is called polar if $c(A)=0$.

The property holds \textbf{quasi surely} (q.s.) if it holds outside a polar set.
\end{df}

\begin{df} We define an \textbf{upper expectation} as follows:
\begin{equation*}
 \bar{\E}[\pdot]:=\sup\limits_{P\in\mathcal{P}}E_P[\pdot].
\end{equation*}
\end{df}

\begin{lm}\label{lm_X0_qs}
If \ $\bar{\E}\big[\|X\|^p\big]=0$ \ then \ $X=0$ quasi surely.
\end{lm}

\begin{proof}
$\absz$\\
 $\bar{\E}\big[\|X\|^p\big]=0$ \ that is for every \ $P\in\mathcal{P} \ \  \ E_P\big[\|X\|^p\big]=0$ \ hence \ \mbox{$  X=0\ \  P$-a.s.} \ for every $P\in\mathcal{P}$.

And finally \ $c\big(\{X=0\}\big)=\sup\limits_{P\in\mathcal{P}}P\big(\{X=0\}\big)=1>0$.

\end{proof}

\subsection{Solving a fully nonlinear heat equation}

Let $B_t$ is $G$-Brownian motion with correspondent $G$-functional $G(\pdot)$, such that \ \ 
 $G(A)=\dfrac{1}{2t}\,\E\Big[\langle A B_t, B_t\rangle\Big]$.

Consider equation \eqref{eq_p} with $A=0$, i.e. the following parabolic PDE:
\begin{equation}\tag{P0}\label{eq_p0}
\hspace{30mm}\begin{cases}
\partial_t u+G(D^2_{xx} u)=0 \,, \ \ \ t\in[0,T)\,, \ x\in\Hs;\\
u(T,x)=f(x) \,.
\end{cases}
\end{equation}
In order to proof the existence of a viscosity solution of equation \eqref{eq_p0} we need Taylor's expansion for the function $u(t,x)$. Such a trivial statement will show us a following lemma:
\begin{lm}[Taylor's formula]\label{lm_taylor_formula}
$\absz$\\
Let $\psi\in C^2(\R\tdot\Hs\rightarrow\R)$\,, \ $(\delta,\Delta x)$ and
$(t, x)$ are in $(\R, \Hs)$, such that $(t+\delta,x+\Delta x)\in D(\psi)$.

Then:

$\psi\bigl(t+\delta,x+\Delta x\bigr)=\psi(t,x)+\delta\cdot\partial_t\psi(t,x)+\langle D_x\psi(t,x),\Delta x\rangle
+\dfrac{1}{2}\,\delta^{\,2}\cdot\partial^2_{tt}\,\psi(t,x)$

\hfill$
+\delta\cdot\partial_t\bigl[\langle D_x\psi(t,x),\Delta x\rangle\bigr]
+\dfrac{1}{2}\langle D_{xx}^{\,2}\psi(t,x)\Delta x,\Delta x\rangle+o(\delta^{\,2}+\|\Delta x^{\,2}\|)$.
\end{lm}

%
%
%
%
%
%
%

\begin{teo} \label{th_exist_vs_p0}
Let $f$ is a $B$-continuous of \ $\Cp(\Hs)$-class real function. Then $u(t,x):=\E \big[f(x+B_{T-t})\big]$ is a unique viscosity solution of equation \eqref{eq_p0}:
\begin{equation}\tag{P0}\label{eq_ppp0}
\hspace{30mm}\begin{cases}
\partial_t u+G(D^2_{xx} u)=0 \,, \ \ \ t\in[0,T)\,, \ x\in\Hs;\\
u(T,x)=f(x) \,,
\end{cases}
\end{equation}
where $B_t$ is a $G$-Brownian motion with a correspondent $G$-functional $G(\pdot)$.

\end{teo}

\begin{proof}
$\absz$\\
Let $\psi$ be a test function, and for every fixed point $(t,x)\in[0,T]\tdot\Hs$ we have:  
\parbox[t][\height]{5cm}{$u\leq \psi\,;\\ 
\phantom{|}\!u(t,x)=\psi(t,x)\,.$}
\\
Taking a small enough $\delta$ yields:
\\
$\psi(t,x)=u(t,x)
=\E[ f(x+B_{T-t})]
=\E[ f(x+\underbrace{B_\delta+B_{T-(t+\delta)}}_{\text{indep.}})]\\
=
\E\Bigg[\E\Big[ f(x+\beta+B_{T-(t+\delta)})\Big]_{\beta=B_\delta}\Bigg]=
\E\Big[u(t+\delta,x+\beta)\Big|_{\beta=B_\delta}\Big]$

\hfill$
=\E\big[u(t+\delta,x+B_\delta)\big]
\leq\E\big[\psi(t+\delta,x+B_\delta)\big]$.
\\
Using the Taylor formula (\textbf{Lm.\ref{lm_taylor_formula}}):
\\
$\psi(t+\delta,x+B_\delta)=\psi(t,x)+\delta\cdot\partial_t\psi(t,x)+\left\langle D_x\psi(t,x),B_\delta\right\rangle
+\dfrac{1}{2}\,\delta^{\,2}\cdot\partial^{\,2}_{tt}\psi(t,x)$

\hfill$+\delta\cdot\partial_t\big[\left\langle D_x\psi(t,x),B_\delta\right\rangle\big]+\dfrac{1}{2}\,\left\langle D^{2}_{xx}\psi(t,x)B_\delta,B_\delta\right\rangle
+o\big(\delta^2+\|B_\delta\|^2\big)$.
\\
Then we have \ 

$0\leq\E\Big[\psi(t+\delta,x+B_\delta)\Big]-\psi(t,x)
\leq\E\Big[\psi(t+\delta,x+B_\delta)-\psi(t,x)\Big]$

\hfill$
=\delta\,\partial_t\psi(t,x)+\dfrac{1}{2}\,\delta^{\,2}\,\partial^{\,2}_{tt}\psi(t,x)+\delta\cdot G(D^{2}_{xx}\psi)(t,x)
+o\big(\delta^2+\|B_\delta\|^2\big)$.
\\
Letting $\delta\to0$ \ yields \ $\Big[\partial_t\psi+ G(D^{2}_{xx}\psi)\Big](t,x)\geq0$.
\\
Note, that $u$ is continuous at $(t,x)\in[0,T]\tdot\Hs$.

In fact, let us show that $\ \  \E\big[f(x+B_s)\big]\xrightarrow[s\to t]{}\E\big[f(x+B_t)\big]\ ,\ \ \ t\in[0,T]\ :$

$0\leq\Big|\E\big[f(x+B_s)\big]-\E\big[f(x+B_t)\big]\Big|
\leq\E\Big[\big|f(x+B_s)-f(x+B_t)\big|\Big]\\
\leq\E\Big[\big(1+\|x+B_s\|^m+\|x+B_t\|^m\big)\cdot\|B_s-B_t\|\Big]\\
\leq\Bigg(\E\Big[\big(1+\|x+B_s\|^m+\|x+B_t\|^m\big)^2\Big]\cdot\E\Big[\|B_s-B_t\|^2\Big]\Bigg)^{\frac{1}{2}}\\
\leq\Bigg(\underbrace{\E\Big[\big(1+2m\|x\|^m+m\cdot s^{\frac{m}{2}}\|\bar{X}\|^m+m\cdot t^{\frac{m}{2}}\|\bar{\bar{X}}\|^m\big)^2\Big]}
_{\text{bdd. by \textbf{Prop.\ref{prop_moments}}}}
\cdot(s-t)\underbrace{\E\Big[\|X\|^2\Big]}_{\text{bdd.}}\Bigg)^{\frac{1}{2}}$

\hfill$
\phantom{x}_{s\to t}\!\!\!\searrow0,
$
\\
\hfill where $X\,\ \bar{X}\,,\ \bar{\bar{X}}$ are independent copies of $B_1$.
\\
$u(T,x)=\E\big[f(x)\big]=f(x)\leq f(x)$.
\\
So we see that $u$ is a viscosity subsolution of equation \eqref{eq_ppp0}.

In the same way one can prove that $u$ is a viscosity supersolution, and the existence is proved.
\\
It is clear that if $f$ is $B$-continuous and has a polynomial growth that $u$ is also $B$-continuous and has a polynomial growth, because a sublinear expectation $\E$ does not influence on it. So we can conclude that $u$ is a unique viscosity solution by \textbf{Th.\ref{th_uniq_vs}}.

\end {proof}

\subsection{Basic space constructions}

Let $\Us, \Hs$ be Hilbert spaces, and $\{e_i\}$ and $\{f_j\}$ are systems of orthonormal bases on them respectively. We consider $G$-Brownian motion $B_t$ with values in $\Us$. Assume that random variable $B_1\sim N_G\big(0,\Sigma\big)$ with a covariance set $\Sigma$. Recall that $\Sigma$ is a convex set of linear bounded non-negative symmetric trace-class operators.

Now we are going to define a Banach space of operators $\Phi$ with values in $\Hs$ and defined in an appropriate space of $\Us$, 
such that $dom\,\Phi\subset\Us$,
endowed with the norm $\|\Phi\|^2_{L_2^\Sigma}:=\sup\limits_{Q\in\,\Sigma}{\Tr}\big[\Phi Q \Phi^*\big]
=\sup\limits_{Q\in\,\Sigma}\|\Phi Q^{1/2}\|^2_{L_2(\Us,\Hs)}$.


\begin{prop}
$\|\pdot\|^2_{L_2^\Sigma}$ is a norm.
\end{prop}

%
%
%
%
%
%

It needs that \ \ $ dom\,\Phi\supset\bigcup\limits_{Q\in\Sigma}Q^{1/2}(\Us)$.

Take \ $ dom\,\Phi=\Us_\Sigma
:=\big\{u\in\Us\ \big|\ \exists u_i\in Q_i^{1/2}(\Us),\ Q_i\in \Sigma:\ u\!\!\!\!\overset{_{\phantom{U}}\|\, \cdot\, \|_{\Us}}{=}\!\!\sum\limits_{i=1}^\infty u_i,$

\hfill$  \sum\limits_{i=1}^\infty \|u_i\|_{Q_i^{1/2}(\Us)}<\infty \big\}$.

and $\|u\|_{\Us_\Sigma}:=\inf\big\{\sum\limits_{i=1}^m \|u_i\|_{Q_i^{1/2}(\Us)}\ \big|\ 
u\!\!\!\!\overset{_{\phantom{U}}\|\, \cdot\, \|_{\Us}}{=}\!\!\sum\limits_{i=1}^\infty u_i,\  u_i\in Q_i^{1/2}(\Us),\ Q_i\in \Sigma\big\}$.
\\
Note that $\|u_i\|_{Q_i^{1/2}(\Us)}=\|Q_i^{-1/2}u_i\|_\Us,\ u_i\in Q_i^{1/2}(\Us),\ Q_i\in\Sigma$;

and it is known that 
$\Big(Q_i^{1/2}(\Us),\,\|\pdot\|_{Q_i^{1/2}(\Us)}\Big)$ is a Banach space (see \cite[4.2]{da_prato_zabcz3}).

\begin{rem}\label{rem_conv_U_U_Sigma_norm}
 In the definition of $\Us_\Sigma$: \

from \ $u\!\!\!\!\overset{_{\phantom{U}}\|\, \cdot\, \|_{\Us}}{=}\!\!\sum\limits_{i=1}^\infty u_i$ \ \ follows \ \ $u\!\!\!\!\overset{_{\phantom{U}}\|\, \cdot\, \|_{\Us_\Sigma}}{=}\!\!\sum\limits_{i=1}^\infty u_i$.
\end{rem}

\begin{proof}
  $\absz$\\
$\|u-\sum\limits_{i=1}^K u^i\|_{U_\Sigma}
=\|\sum\limits_{i=K+1}^\infty u^i\|_{U_\Sigma}
\leq\sum\limits_{i=K+1}^\infty  \|u_i\|_{Q_i^{1/2}(\Us)}
\xrightarrow[K\to\infty]{}0$.

\end{proof}

\begin{lm}\label{lm_U_Sigma_cont_embed}
 $\Us_\Sigma\hookrightarrow\Us$ ( $\Us_\Sigma$ is a continuously embedded in $\Us$), i.e.:

\hspace{27mm}$\Us_\Sigma\subset\Us$;

\hspace{27mm}$\|\pdot\|_\Us\leq C\cdot\|\pdot\|_{\Us_\Sigma}$.
  
In particular $\forall Q\in\Sigma,\ \ Q^{1/2}(\Us)\hookrightarrow\Us$.
\end{lm}

\begin{proof}
 $\absz$\\
Let $Q\in\Sigma$.

It is well known that $Q^{1/2}\in L_2(\Us)$(see \cite[Prop.2.3.4]{prevot_roeckner1}).

Therefore \ $Q^{1/2}(\Us)\subset\Us$.

So that for all \ $u\in Q^{1/2}\,:\ \ \|u\|_\Us
=\|Q^{1/2}Q^{-1/2}u\|_\Us
\leq\|Q^{1/2}\|_\Us\cdot\|u\|_{Q^{1/2}(\Us)}$.{\makebox[10mm][r]{$(\,\ast\,)$}}
\\
Every $ Q_i\in\Sigma$ is bounded in $L(\Us)$ \ and let \ $\|Q_i\|_\Us\leq C^2$.
\\
Also we have that $C^2\geq\|Q_i\|_\Us=\|Q_i^{1/2}Q_i^{1/2}\|_\Us=\|Q_i^{1/2}\|^2_\Us$.

From $(\,\ast\,)$ it follows that \ $\|u\|_\Us\leq C \cdot\|u\|_{Q_i^{1/2}(\Us)}$

Therefore \ $Q^{1/2}(\Us)\hookrightarrow\Us$, \ $\forall Q_i\in\Sigma$.\hfill$(\,\#\,)$
\\
Also it is clear that $\Us_\Sigma\subset\Us$.

Hence for all \ $u\in\Us_\Sigma,\ u=\sum\limits_{i\geq1} u_i \ :$

\hfill$
\|u\|_\Us
=\|\sum\limits_{i\geq1} u_i\|_\Us
\leq \sum\limits_{i\geq1} \|u_i\|_\Us
\overset{(\,\#\,)}{\leq}C\cdot\sum\limits_{i\geq1} \|u_i\|_{Q_i^{1/2}(\Us)}
$.

Taking $\inf$ \ we obtain \ $\|u\|_\Us\leq C\cdot\|u\|_{\Us_\Sigma}$.
\\
So we have \ $\Us_\Sigma\hookrightarrow\Us$.

\end{proof}

\begin{lm}
  $L(\Us,\Sigma)\hookrightarrow L(\Us_\Sigma,\Hs)$.
\end{lm}

\begin{proof}
$\absz$\\
Since $\|\Phi u\|_\Hs\leq \|\Phi \|_{L(\Us,\Hs)}\cdot\|u\|_\Us
\overset{\textbf{Lm.\ref{lm_U_Sigma_cont_embed}}}{\leq} C\cdot \|\Phi \|_{L(\Us,\Hs)}\cdot \|u\|_{\Us_\Sigma}$\ \ \ and\ \ \ 
$\Us_\Sigma\subset\Us$,
\\
then we have \ $L(\Us,\Sigma)\subset L(\Us_\Sigma,\Hs)$\ \ \ and\ \ \ 
$\|\Phi \|_{L(\Us_\Sigma,\Hs)}\leq C\cdot \|\Phi \|_{L(\Us,\Hs)}$.

\end{proof}

Below we will use a criterion of completeness of a normed space (see \cite[Lm.2.2.1]{berg_loefstroem1}) which we have formulated in the following lemma 
\begin{lm}\label{lm_cauch_sq_berg_loefstr}
$\absz$\\
Normed space $A$ is complete if and only if \ 
when from \ $\sum\limits_{n\geq1}\!\|a_n\|_A<\infty$ \ it follows that there exists an element \ $ a\in A$ \ such that \ \ $
a\!\!\!\!\overset{_{\phantom{U}}\|\, \cdot\, \|_A}{=}\!\!\sum\limits_{n\geq1} a_n$.
\end{lm}

\begin{prop}
$\big(\Us_\Sigma,\|\pdot\|_{\Us_\Sigma}\big)$ is a Banach space.
\end{prop}

\begin{proof}
$\absz$\\
Let a sequence $\big\{u^n\big\}\subset\Us_\Sigma$ such that $\sum\limits_{n=1}^\infty \|u^n\|_{\Us_\Sigma}<\infty$.

For every element  $u^n$  there exists a subsequence \ $u^n_{i_n}\in Q_{i_n}^{1/2}(\Us)$, where \ $ Q_{i_n}\in \Sigma$, such that \ $  u^n=\sum\limits_{i_n=1}^\infty u^n_{i_n}$.

Since $\Sigma$ is no countable, but we can always reduce to the case with a countable number of $Q\in\Sigma$:

Taking $\Sigma':=\bigcup\limits_{n\geq1}\big\{Q_{i_n}\big\}=\big\{Q_i\big\}_{i\geq1}$, we can see that it is countable.

So we can use here the set $\Sigma'$ instead of $\Sigma$, and 
in the same time we can suppose that $u^n=\sum\limits_{i=1}^\infty u^n_i\,,\ u^n_i\in Q_i^{1/2}(\Us)$,
where $u_i^n=u_{i_n}^n\cdot\ind_{\{i=i_n\}}$, i.e. 
we renumber the elements and putting some zeros. It means that for every \ $u^n$ \ there exists \ $u^n_i\in Q_i^{1/2}(\Us)$, where \ $ Q_i\in \Sigma$ \ such that \ $u^n=\sum\limits_{i=1}^\infty u^n_i$.
\\
So that, there exists a sequence $\big\{\varepsilon_n\big\}$ \ with \ $\sum\limits_{n\geq1}\varepsilon_n <\infty$, such that:

$\varepsilon_n+\|u^n\|_{\Us_\Sigma}\geq\sum\limits_{i=1}^\infty\|u^n_i\|_{Q_i^{1/2}(\Us)}\geq\|u^n_i\|_{Q_i^{1/2}(\Us)}$ .
\hfill$(\,\ast\,)$

Therefore \ $\sum\limits_{n=1}^\infty\|u^n_i\|_{Q_i^{1/2}(\Us)}<\infty\,,\ u_i^n\in Q_i^{1/2}(\Us)$ -- Banach.

Then by \textbf{Lm.\ref{lm_cauch_sq_berg_loefstr}} it follows that there exists \ $ v_i\in Q_i^\frac{1}{2}(\Us)$, such that  

$ 
v_i\!\!\!\overset{_{\phantom{U}}\|\, \cdot\, \|_{Q_i^{1/2}(\Us)}}{=}\sum\limits_{n=1}^\infty u^n_i$, \ \ for every $i$.

So that \ $v_i\!\!\!\!\overset{_{\phantom{U}}\|\, \cdot\, \|_\Us}{=}\!\!\sum\limits_{n=1}^\infty u^n_i$.

Taking $v\!\!\!\overset{_{\phantom{U}}\|\, \cdot\, \|_{\Us}}{:\,=}\!\sum\limits_{i=1}^\infty v^i$, let us show that this series really converges in the $\Us$-norm:

From $(\,\ast\,)$ we have: \ \ 

$
\infty>\sum\limits_{n=1}^\infty\sum\limits_{i=1}^\infty\|u^n_i\|_{Q_i^{1/2}(\Us)}
=\sum\limits_{i=1}^\infty\sum\limits_{n=1}^\infty\|u^n_i\|_{Q_i^{1/2}(\Us)}
\overset{\textbf{Lm.\ref{lm_U_Sigma_cont_embed}}}{\geq}C\cdot\sum\limits_{i=1}^\infty\sum\limits_{n=1}^\infty\|u^n_i\|_\Us$

\hfill$
\geq C\cdot\sum\limits_{i=1}^\infty\|v_i\|_\Us
$.

By \textbf{Lm.\ref{lm_cauch_sq_berg_loefstr}} we deduce that \ $\sum\limits_{i=1}^\infty v^i$ converges in $\|\pdot\|_\Us$.

Also we can show that \ \ $\sum\limits_{i=1}^\infty \|v^i\|_{Q_i^{1/2}(\Us)}$:

In fact, from $(\,\ast\,)$ follows:

$
\infty>\sum\limits_{n=1}^\infty\sum\limits_{i=1}^\infty\|u^n_i\|_{Q_i^{1/2}(\Us)}
=\sum\limits_{i=1}^\infty\sum\limits_{n=1}^\infty\|u^n_i\|_{Q_i^{1/2}(\Us)}
\geq \sum\limits_{i=1}^\infty\|v_i\|_{Q_i^{1/2}(\Us)}
$.

Therefore \ $v\in\Us_\Sigma$.
\\
Consider\ \ \  $\|v-\sum\limits_{n=1}^K u^n\|_{U_\Sigma}
=\|\sum\limits_{i=1}^\infty v_i-\sum\limits_{i=1}^\infty\sum\limits_{n=1}^K u^n_i\|_{U_\Sigma}
\leq\sum\limits_{i=1}^\infty\|\sum\limits_{n=K+1}^\infty u_i^n\|_{Q_i^{1/2}(\Us)}\\
\leq\sum\limits_{i=1}^\infty\sum\limits_{n=K+1}^\infty \|u_i^n\|_{Q_i^{1/2}(\Us)}
\xrightarrow[K\to\infty]{}0$, \ \ \ by the monotone convergence theorem.

Hence \ $v\!\!\!\overset{_{\phantom{U}}\|\, \cdot\, \|_{\Us_\Sigma}}{=}\!\sum\limits_{n=1}^\infty u^n$.

According to \textbf{Lm.\ref{lm_cauch_sq_berg_loefstr}} we can conclude that \ $  \big(\Us_\Sigma,\|\pdot\|_{\Us_\Sigma}\big)$ is complete.

\end{proof}

In the very beginning of this chapter we would like to describe a Banach space with a norm $\|\pdot\|^2_{L_2^\Sigma}$. So, we are ready to define it. This space is closely related to the construction of stochastic integral in the following chapter \ref{sec_stoch_int}.
\begin{df}
$L_2^\Sigma=\big\{\Phi\in L(\Us_\Sigma, \Hs)\ \big|\ \|\Phi\|_{L_2^\Sigma}<\infty\big\}$.
\end{df}

\begin{rem} From the definition above follows that if  \ 
$Q\in\Sigma$ \ then \ 

$\Phi Q^{1/2} \in L_2(\Us,\Hs) $
\end{rem}

\begin{lm}\label{lm_L_Sigma_2_L_U_Sigma_H_cont_embed}
 $L^\Sigma_2\hookrightarrow L(\Us_\Sigma,\Hs)$.

Moreover \ $\|\pdot\|_{L(\Us_\Sigma,\Hs)}\leq \|\pdot\|_{L^\Sigma_2}$ (note that here a constant $C=1$).
\end{lm}

\begin{proof}
$\absz$\\
Let $\Phi\in L^\Sigma_2$ \ then \ $\Phi\in L(\Us_\Sigma,\Hs)$.

$u\in\Us_\Sigma$ \ it means that \ $  
u\!\!\!\!\overset{_{\phantom{U}}\|\, \cdot\, \|_{\Us}}{=}\!\!\sum\limits_{i=1}^\infty u_i,\  u_i\in  Q_i^{1/2}(\Us),\ Q_i\in \Sigma$.

Note also that $\|\Phi\|_{L_2^\Sigma}\geq c_i:=\|\Phi Q_i^{1/2}\|_{L_2(\Us,\Hs)}$.

Therefore \ $\|\Phi u\|_\Hs=\|\Phi (\sum\limits_{i=1} ^\infty u_i)\|_\Hs=\|\sum\limits_{i=1} ^\infty \Phi u_i\|_\Hs=
\|\sum\limits_{i=1} ^\infty \Phi Q_i^{1/2} Q_i^{-1/2}u_i\|_\Hs\\
\leq
\sum\limits_{i=1} ^\infty \|\Phi Q_i^{1/2}\|_{L_2(\Us,\Hs)} \| Q_i^{-1/2}u_i\|_\Us=
\sum\limits_{i=1} ^\infty c_i \| u_i\|_{Q_i^{1/2}(\Us)}\leq
\sum\limits_{i=1} ^\infty \|\Phi\|_{L_2^\Sigma} \| u_i\|_{Q_i^{1/2}(\Us)}$.
\\
Taking \  $\inf$ \ for all such $u_i$ we get \ 
$\|\Phi u\|_\Hs\leq\|\Phi\|_{L_2^\Sigma} \| u\|_{\Us_\Sigma}$.
\\
Hence \ $ \|\Phi\|_{L(\Us_\Sigma,\Hs)}\leq \|\Phi\|_{L^\Sigma_2}$.
\\
So we can deduce that \ $L^\Sigma_2\hookrightarrow L(\Us_\Sigma,\Hs)$.

\end{proof}

\begin{prop}
$\big(L_2^\Sigma,\|\pdot\|_{L_2^\Sigma}\big)$ is a Banach space.
\end{prop}

\begin{proof}
$\absz$\\
Let $\big\{\Phi_n\big\}\subset L_2^\Sigma$ be a Cauchy sequence: \ 
$\|\Phi_n-\Phi_m\|_{L_2^\Sigma}\xrightarrow[n,\,m\to\infty]{}0$.

Or we can rewrite it as \ 
$\sup\limits_{Q\in\,\Sigma}\|\Phi_n Q^{1/2}-\Phi_m Q^{1/2}\|^2_{L_2(\Us,\Hs)}\xrightarrow[n,\,m\to\infty]{}0$.

Let $Q\in\Sigma$ be an arbitrary fixed.

In the classical case the space $L^0_2=L_2\big(Q^{1/2}(\Us),\,\Hs\big)$ \ 
with the norm 
 
$\|\Phi\|^2_{L_2^0}:=\Tr\big[\Phi Q \Phi^*\big]=\|\Phi Q^{1/2}\|^2_{L_2(\Us,\Hs)}$
is Banach (see \cite[4.2]{da_prato_zabcz3}).

It implies that \ $\Phi_n\xrightarrow[n\to\infty]{\|\, \cdot\, \|_{L_2^0}}\Phi^Q$,\ \ i.e.:\ \  
$\|\Phi_n Q^{1/2}-\Phi^Q Q^{1/2}\|_{L_2(\Us,\Hs)}\xrightarrow[n\to\infty]{}0$.

Therefore \ $\sup\limits_{Q\in\Sigma}\|\Phi^Q Q^{1/2}\|_{L_2(\Us,\Hs)}\leq C_0$,

and also \ \ $\sup\limits_{Q\in\Sigma}\|\Phi_n Q^{1/2}\|_{L_2(\Us,\Hs)}\leq C_0$.
\hfill$(\,\ast\,)$

If $u\in\Us_\Sigma$ \ then by the definition we have:

$u\!\!\!\!\overset{_{\phantom{U}}\|\, \cdot\, \|_{\Us}}{=}\!\!\sum\limits_{i=1}^\infty u_i\ ,\ \ u_i\in Q_i^{1/2}(\Us),Q_i\in \Sigma, \ \ \sum\limits_{i=1}^\infty \|u_i\|_{Q_i^{1/2}(\Us)}<\infty$. 

It follows that \ $\Phi_n u_i\xrightarrow[n\to\infty]{\|\, \cdot\, \|_\Hs}\Phi^{Q_i} u_i$, and we can get that
 
\hfill$\|\Phi_n u_i\|_\Hs\xrightarrow[n\to\infty]{}\|\Phi^{Q_i} u_i\|_\Hs$.
\\
Define $\Phi u\!\!\!\overset{_{\phantom{U}}\|\, \cdot\, \|_{\Hs}}{:\,=}\sum\limits_{i=1}^\infty \Phi^{Q_i} u_i$. 

For this reason we need to show that the series converges in the $\Hs$-norm and that $\Phi u$ is well defined:

\textbf{(a)} Consider $\sum\limits_{i=1}^\infty \|\Phi^{Q_i} u_i\|_\Hs
\overset{\textbf{Lm. Fatou}}{\leq}\varliminf\limits_{n\to\infty}\sum\limits_{i=1}^\infty \|\Phi_n u_i\|_\Hs\\
=\varliminf\limits_{n\to\infty}\sum\limits_{i=1}^\infty \|\Phi_n Q_i^{1/2} Q_i^{-1/2}u_i\|_\Hs
\leq\varliminf\limits_{n\to\infty}\sum\limits_{i=1}^\infty \|\Phi_n Q_i^{1/2}\|_{L_2(\Us,\Hs)} \cdot\|Q_i^{-1/2}u_i\|_\Us$

\hfill$
\overset{(\,\ast\,)}{\leq} C_0\cdot \sum\limits_{i=1}^\infty\|u_i\|_{Q_i^{1/2}(\Us)}<\infty
$.

Taking $\inf$ we get that \ $\|\Phi u\|_\Hs\leq C_0\cdot\|u\|_{\Us_\Sigma}$.
\\
\textbf{(b)} Let  $ u\!\!\!\!\overset{_{\phantom{U}}\|\, \cdot\, \|_{\Us}}{=}\!\!\sum\limits_{i=1}^\infty v_i$ is another representation.

Note that $\sum\limits_{i=1}^\infty (u_i-v_i)=0$ in the $\Us$-norm, 

and by \textbf{Rem.\ref{rem_conv_U_U_Sigma_norm}} the series also converges to zero in the $\Us_\Sigma$-norm.

It means that \ $\|\sum\limits_{i=1}^n (u_i-v_i)\|_{\Us_\Sigma}=\|\sum\limits_{i=n+1}^\infty (u_i-v_i)\|_{\Us_\Sigma}\xrightarrow[n\to\infty]{}0$.
\hfill$(\,\#\,)$

Consider \ $\|\sum\limits_{i=1}^\infty \Phi^{Q_i} u_i-\sum\limits_{i=1}^\infty \Phi^{Q_i} v_i\|_\Hs
=\|\sum\limits_{i=1}^\infty \Phi^{Q_i} (u_i- v_i)\|_\Hs\\
\leq\underbrace{\sum\limits_{i=n+1}^\infty \|\Phi^{Q_i} (u_i- v_i)\|_\Hs}_{\phantom{xxxx}\textbf{(a)}\,:\ \ \searrow0,\ n\to\infty}
+\|\sum\limits_{i=1}^n \Phi^{Q_i} (u_i- v_i)\|_\Hs
\leq \varepsilon+\lim\limits_{m\to\infty}\|\sum\limits_{i=1}^n \Phi_m (u_i- v_i)\|_\Hs\\
\leq \varepsilon+\sup\limits_{m\geq1}\|\Phi_m\|_{L(\Us_\Sigma,\Hs)}\cdot\|\sum\limits_{i=1}^n  (u_i- v_i)\|_{\Us_\Sigma}
\overset{\textbf{Lm.\ref{lm_L_Sigma_2_L_U_Sigma_H_cont_embed}},\ (\,\#\,)}{\leq} \varepsilon
+\sup\limits_{m\geq1}\|\Phi_m\|_{L^\Sigma_2}\cdot\varepsilon
\leq\varepsilon+C\cdot\varepsilon
$.

So we have \ $\sum\limits_{i=1}^\infty \Phi^{Q_i} u_i=\sum\limits_{i=1}^\infty \Phi^{Q_i} v_i$, \ and \ $\Phi u$ does not depend on the representation.
\\
\textbf{(c)} Also \ \ $\Phi\big|_{Q^{1/2}(\Us)}=\Phi^Q$. 

Because, if $v\in  Q_1^{1/2}(\Us)\cap Q_2^{1/2}(\Us)\,,\ Q_1,\,Q_2\in\Sigma$ \ then \ 

\hfill$
\Phi^{Q_1}v=\lim\limits_{n\to\infty}\Phi_n v=\Phi^{Q_2}v$.

So we can conclude that for every \ $u_i\in  Q_i^{1/2}(\Us)\,:\ \Phi^{Q_i}u_i=\Phi\, u_i$.
\\
And we have that $\Phi $  is defined correctly.

It is also clear that $\Phi$ is linear, and boundedness follows from \textbf{(a)}. 
\\
Now let us turn back to the Cauchy sequence $\big\{\Phi_n\big\}$, so that:\

$\forall\varepsilon>0 \ \ \exists N \ \ \forall n,m>N\,: \ \ \|\Phi_n Q^{1/2}-\Phi_m Q^{1/2}\|_{L_2(\Us,\Hs)}<\varepsilon \ \  \forall Q\in\Sigma$.

Letting \ 
$m\to\infty$ \ yields \ $\|\Phi_n Q^{1/2}-\Phi^Q Q^{1/2}\|_{L_2(\Us,\Hs)}\leq\varepsilon\ \  \forall Q\in\Sigma$.

Or we can rewrite and obtain that:

$\forall\varepsilon>0\ \exists N\ \forall n,m>N\,:\ 
\sup\limits_{Q\in\Sigma}\|\Phi_n Q^{1/2}-\Phi^Q Q^{1/2}\|_{L_2(\Us,\Hs)}\leq\varepsilon\ \ $.

What implies that:

$\forall\varepsilon>0\ \exists N\ \forall n,m>N\,:\ 
\sup\limits_{Q\in\Sigma}\|\Phi_n Q^{1/2}-\Phi Q^{1/2}\|_{L_2(\Us,\Hs)}\leq\varepsilon\ \ $.

And it is the same as \ $\|\Phi_n-\Phi\|_{L_2^\Sigma}\xrightarrow[n\to\infty]{}0$.

\end{proof}

\subsection{Existence of G-normal distribution}

Consider the $G$-PDE \eqref{eq_p0}. For the following calculations we change a time $t\longrightarrow T-t$ and obtain the analogous $G$-PDE. 
\begin{equation}\tag{P$'$0}\label{eq_pp0}
\hspace{30mm}\begin{cases}
\partial_t u-G(D^2_{xx} u)=0 \,, \ \ \ t\in(0,T]\,, \ x\in\Hs;\\
u(0,x)=f(x) \,.
\end{cases}
\end{equation}
Note that all results for viscosity solution of the $G$-PDE \eqref{eq_p0} are also the same as for the \eqref{eq_pp0}. So, 
we can fix $u=u^f(t,x)$ as a unique viscosity solution of \eqref{eq_pp0}.

\begin{lm}[Some properties of the viscosity solutions of \eqref{eq_pp0}]\label{lm_prop_vs}{\mbox{ }}\\
\hspace{10mm}\parbox[t][{1\height}]{15cm}{1)\ $u^{\!\!\!\!{\phantom{\frac{1}{1}} u}^{\scriptscriptstyle f}(s,\,\cdot)}(t,x)=u^{ f}(t+s,x)$;\\
2)\ if \ $ f\equiv C$ \ then \ $u^{ f}\equiv C$;\\
3)\ $u^{ f (\bar{x}+\sqrt{\lambda}\,\circ)}(1,0)=u^{ f\,(\sqrt{\lambda}\,\circ)}(1,\frac{\bar{x}}{\sqrt{\lambda}})=u^{ f\,(\circ)}(\lambda,\bar{x})\,,\ \ \lambda\geq0$;\\
4)\ $u^{\lambda  f}=\lambda u^{ f}\,,\ \ \lambda\geq0$;\\
5)\ if \ $ f\leq g$ \ then \ $u^{ f}\leq u^{ g}$;\\
6)\ $u^{ f+ g}\leq u^{ f}+u^{ g}$.}
\end{lm}

\begin{proof}
$\absz$\\
We will prove just the {\bf {1)}}, because the rest items can be straightforwardly checked implying that 
$u^ f (t,x)=\E[(x+B_t)]$.

Let us consider a boundary condition \ $u(0,x)=u^ f (s,x)$.

We can carry initial point $''0''$ to $''t''$ so that $u(t,x)=u^ f (s+t,x)$.
On the other hand \ $u(t,x)=u^{\!\!\!\!{\phantom{\frac{1}{1}} u}^{\scriptscriptstyle f}(s,\,\cdot)}(t,x)$.

What implies that 
$u^{\!\!\!\!{\phantom{\frac{1}{1}} u}^{\scriptscriptstyle f}(s,\,\cdot)}(t,x)=u^{ f}(t+s,x)$.

\end{proof}

\begin{teo}\label{th_exist_gd}
 Let $G(\pdot)$ is a given $G$-functional.

Then:

\hspace{5mm}1) There exists \ $ \xi\sim N_G(0,\Sigma)$;

\hspace{5mm}2) There exist \ a sequence $\{\xi_i\,,\ i\geq1\}$ \ such that for every $i\ \ \ $

\hfill $\xi_i\sim N_G(0,\Sigma)$ \ and \ $  \xi_{i+1}\dperp(\xi_1,\ldots,\xi_i)$.
\end{teo}

\begin{proof}
 $\absz$\\
You can see the original proof in finite dimensions in \cite[II.2]{peng10} or the full proof for the infinite dimensional case in the author's thesis \cite{ibragimov1}.

\end{proof}

\begin{teo}\label{th_criterion_gd}
 If $G$ -- is a given $G$-functional and 

$u(t,x):=\E f(x+\sqrt{T-t} X)$ is a viscosity solution of equation \eqref{eq_p0}.

Then $X\sim N_G(0,\Sigma_G)$.
\end{teo}

\begin{proof}
  $\absz$\\
Since $u$ is a unique viscosity solution of equation \eqref{eq_p0} by \textbf{Th.\ref{th_exist_vs_p0}} then
$u(t,x)=\E f(x+B_{T-t})=\E f(x+\sqrt{T-t} X)$, \ 
where $B_t$ is a $G$-Brownian motion with a covariation set $\Sigma_G$. Note that  \ $\sqrt{T-t} X\sim B_{T-t}$, \ where \ $X\sim B_1\sim N_G(0,\Sigma_G)$.

\end{proof}

\begin{rem}
$\absz$\\
Let $X\sim N_G\big(0,\Sigma\big)$, then $(-X)\sim X$.
\end{rem}

%
%

\subsection{Existence of G-Brownian motion and notion of G-expectation}

Let $\Us, \Hs$ will be Banach spaces. Recall (see \ref{subsec_sublin_expect}) that 

$\Cp\,(\Us,\Hs):=\Big\{\fee:\Us\to\Hs\ \Big|\ \|\fee(x)-\fee(y)\|_\Hs
$

\hfill$\leq C\cdot\big(1+\|x\|^m_{\Us}+\|y\|^m_{\Us}\big)
\cdot\|x-y\|_{\Us}\Big\}$.

Later on we will need a bounded subspace: 

$\Cbp\,(\Us,\Hs):=\Cp\,(\Us,\Hs)\cap{\mathbf{C}_{\,b}}\,(\Us,\Hs)$.
\\
Let $\Omega=\big\{\omega_t:[0,+\infty)\rightarrow\Us\mid\omega_0=0,\ \omega_t {\mbox{--continuous}} \big\} $;\ \ 
\\
$G(\pdot):\, K_S(\Hs)\to\R$ is a $G$-functional.

$B_t(\omega_s):=\bigl.\omega_s\bigr|_{s=t}$ is a canonical process;

$Lip( \Omega_t ):=\Big\{ \fee(B_{t_1},..,B_{t_n})\mid n\geq1,\ 
t_1,..,t_n\in[0,t],\  \fee\in \Cp\,(\Us^n, L_2^\Sigma)\Big\} $;

$Lip( \Omega ):=\bigcup\limits_{n\geq1} Lip( \Omega_t )$.

\begin{rem}\label{rem_finite_norm_LpGH}
 If \  $X\in Lip( \Omega )$ \ then \ $ \E\big[\|X\|^p_{L_2^\Sigma}\big]<\infty\,,\ \ p>0$.
\end{rem}

\begin{proof}
 $\absz$\\
We will use such a trivial fact:

$\forall a>0\ \ \ \forall m>0\ \ \ \exists M>m\ \ \ \exists C>0\ :\ \ \ a(1+a^m)\leq C(1+a^M)$.
\\
It means if \ $\fee\in\Cp\,(\Us^n, L_2^\Sigma)$ \ then

$
\|\fee(x_1,..,x_n)\|_{L_2^\Sigma}
\leq C\cdot\big(1+\|(x_1,..,x_n)\|^M_{\Us^n}\big)
= C\cdot\Big(1+\big(\sum\limits_{k=1}^n\|x_k\|^2_{\Us}\big)^{M/2}\Big)$

\hfill$
\leq C\cdot\Big(1+\sum\limits_{k=1}^n\|x_k\|^M_{\Us}\Big)$.

So that, $\fee$ is a polynomial growth function.
\\
We also have 

 $X=\fee(B_{t_1},..,B_{t_n})
=\fee(\sqrt{t_1}X^{(1)},\ldots,\sqrt{t_n}X^{(n)}\big)
=\widetilde{\fee}(X^{(1)},\ldots,X^{(n)}\big)$,

\hfill where $B_1\sim X^{(k)}\sim N_G(0,\Sigma)$.
\\
Hence \ $\|X\|^p_{L_2^\Sigma}=\|\widetilde{\fee}\|^p_{L_2^\Sigma}(X^{(1)},\ldots,X^{(n)}\big)$,

and \ \ \ $\|\widetilde{\fee}\|^p_{L_2^\Sigma}:\Us^n\to\R$\ \ \ is a polynomial growth function.
\\
Therefore \ $\E\big[\|X\|^p_{L_2^\Sigma}\big]
\leq C\cdot\E\Big[\big(1+\sum\limits_{k=1}^n\|X^{(k)}\|^M_{\Us}\big)\Big]
\leq C\cdot\big(1+\sum\limits_{k=1}^n\E\big[\|X^{(k)}\|^M_{\Us}\big]\big)$

\hfill$
\overset{\textbf{Prop.\ref{prop_moments}}}{<}\infty
$.

\end{proof}

\begin{prop} Under settled above conditions there exists a sublinear expectation 
$\E\big[\psi(\,\pdot\,)\big]=\F_\cdot\big[\psi\big]:Lip( \Omega )\rightarrow\R$,
such that $B_t$ is a $G$-Brownian motion on $\big( \Omega,\ Lip( \Omega),\ \E\big)$,
where $\psi\in\Cp\,(L^\Sigma_2)$ .
\end{prop}
 
\begin{proof}
$\absz$\\
 You can see the original proof in finite dimensions in \cite[III.2]{peng10} or the full proof for the infinite dimensional case in the author's thesis \cite{ibragimov1}.

\end{proof}
 \begin{df}
 Such sublinear expectation $\E$ will be called $G$-expectation.
\end{df}

 In the same manner we can define related conditional $G$-expectation with respect to $\Omega_t$:

For $X\in Lip( \Omega )$ \ which can be represented in the following form

$X=\fee(B_{t_1}-B_{t_0},..,B_{t_n}-B_{t_{n-1}})$, \ \ where\ \ $\fee\in \Cp\,(\Us^n,L^\Sigma_2)\,,$

\hfill $0=t_0\leq t_1\leq \ldots\leq t_n<\infty$,

we have the next definition:

\begin{df}
 Conditional $G$-expectation $\E[\pdot\mid \Omega_{t_j}]$ is defined as
$$\E[\psi(X)\mid \Omega_{t_j}]:=\chi(B_{t_1}-B_{t_0},..,B_{t_j}-B_{t_{j-1}}),$$
where $\chi(x_1,..,x_j)=\tilde{\E}\big[\psi\circ\fee(x_1,..,x_j,\sqrt{t_{j+1}-t_j}\,\xi_j,..,\sqrt{t_n-t_{n-1}}\,\xi_n)\big]$,

\hfill $\psi\in\Cp\,(L^\Sigma_2)$. 
\end{df}

\begin{rem}
Since $\fee$ and $\psi$ are Lipschitz with a polynomial growth then $\chi$ is also Lipschitz with a polynomial growth.

Hence \ $\chi(B_{t_1}-B_{t_0},..,B_{t_j}-B_{t_{j-1}})\in\H_0$.

Let us define a space $L^1_G(\Omega)$ to be a completion of $\H$ under the norm 

$\|\pdot\|=\E\big[\pdot\big]$.

Thus we will consider an unconditional $G$-expectation with values in the complete space:
$\E[\psi(X)\mid \Omega_{t_j}]:Lip(\Omega)\to L^1_G(\Omega)$.

\end{rem}

\begin{rem}
Actually, we apply the definitions of conditional and unconditional $G$-expectation only when $\psi=\|\cdot\|^p_{L_2^\Sigma}$.
\end{rem}

Let us define the following space: 

$Lip( \Omega^t ):=\Big\{ \fee(B_{t_2}-B_{t_1},..,B_{t_{n+1}}-B_{t_n})\mid n\geq1,\ 
t_1,..,t_n\in\R\!\overset{+}{\phantom :},$

\hfill$\fee\in \Cp\,(\Us^n, L_2^\Sigma)\Big\} $.
\\
Since a conditional $G$-expectation $\E[\,\cdot\,\mid \Omega_{t}]$ relates to a sublinear expectation $\tilde{\E}[\pdot]$, we can conclude that
$\E[\,\cdot\,\mid \Omega_{t}]$ is also a sublinear expectation. 

The following properties are trivial consequences of the definition:
 
\begin{prop}  \label{prop_propert_cond_exp}
Let $X\in Lip( \Omega ),\ \  Z_1\in Lip( \Omega_t ),\ \  Z_2\in Lip( \Omega^t )$,

\hfill$\fee\in \Cp\,(\Us^n,L^\Sigma_2)$.

Then the following equalities hold:
 \begin{enumerate}[(1)]
\item $\E\big[\E[\fee(X)\mid \Omega_{t}]\mid \Omega_{s}\big]=\E[\fee(X)\mid \Omega_{t \wedge s}]$;

$\E\big[\E[\fee(X)\mid \Omega_{t}]\big]=\E[\fee(X)]$;
\item $\E[\fee(Z_1)\mid \Omega_{t}]=\fee(Z_1)$;
\item $\E[\fee(Z_2)\mid \Omega_{t}]=\E[\fee(Z_2)]$.
 \end{enumerate}
\end{prop}

\begin{rem}
 Due to this  \textbf{Prop.\ref{prop_propert_cond_exp}} and its properties $(2)$ and $(3)$ respectively we will call the variable $Z_1$ as \textbf{$\Omega_t$-measurable}, and $Z_2$ as \textbf{independent from $\Omega_t$}.
\end{rem}

Now we are going to describe an extension of the $G$-expectation on the completion of the space $Lip( \Omega )$. We will make such completion under the norm
$\|X\|^p_{\Sigma,\,p}=\E\Big[\|X\|^p_{L_2^\Sigma}\Big]$, since this norm on $Lip( \Omega )$ is finite by \textbf{Rem.\ref{rem_finite_norm_LpGH}}. 

Denote ${^\Hs\! L_G^p\big( \Omega \big)}=\mathrm{compl}\Big( Lip( \Omega ),\,\E\big[\|\pdot\|^p_{L_2^\Sigma}\big]\Big)$.

\begin{prop}\label{prop_extens_G_exp}
$G$-expectation
$\E\big[\|\,\pdot\,\|^p_{L_2^\Sigma}\big]:Lip(\Omega)\to\R$
as well as conditional one $\E\big[\|\cdot\|^p_{L_2^\Sigma}\mid \Omega_t\big]:Lip(\Omega)\to L_G^1(\Omega)$
can be continuously extended to the space
 ${^\Hs\! L_G^p\big( \Omega \big)}$.
\end{prop}

\begin{proof}
 $\absz$\\
If\ \ $X\in{^\Hs\! L_G^p\big( \Omega \big)}$ \ then there exist \ $X_n\in Lip( \Omega )\,, $ such that\\ 
$\E\Big[\|X-X_n\|^p_{L_2^\Sigma}\Big]\xrightarrow[n\to\infty]{}0$.

Here we take $p\geq2$ because the case $p=1$ is trivial. 
\\
\textbf{1)} Firstly we show the following convergence (which we use for the extension):
\\
$\Big|\E\Big[\|X\|^p_{L_2^\Sigma}\Big]-\E\Big[\|X_n\|^p_{L_2^\Sigma}\Big]\Big|
\leq \E\Big[\Big|\|X\|^p_{L_2^\Sigma}-\|X_n\|^p_{L_2^\Sigma}\Big|\Big]$

\hfill$
\leq \E\Big[\|X-X_n\|_{L_2^\Sigma}\cdot\big(\|X\|^{p-1}_{L_2^\Sigma}+\ldots+\|X_n\|^{p-1}_{L_2^\Sigma}\big)\Big]$

\hfill$
\leq \Bigg(\E\Big[\|X-X_n\|^2_{L_2^\Sigma}\Big]\Bigg)^{1/2}
\cdot C\cdot\Bigg(\E\Big[\big(\|X\|^{2(p-1)}_{L_2^\Sigma}+\|X_n\|^{2(p-1)}_{L_2^\Sigma}\big)\Big]\Bigg)^{1/2}$

\hfill$
\leq \widetilde{C}\cdot\Bigg(\E\Big[\|X-X_n\|^p_{L_2^\Sigma}\Big]\Bigg)^{1/p}
\xrightarrow[n\to\infty]{}0$.

\textbf{2)} Now it is enough to only take the more general case with a conditional $G$-expectation.
We consider an operator $T$ on the space $Lip( \Omega )$, such that\ \
$TX_n=\E\big[\|X_n\|^p_{L_2^\Sigma}\mid \Omega_t\big]$. 

Let us define $TX:=\lim\limits_{n\to\infty}TX_n$ \ if \ $\E\Big[\big|TX-TX_n\big|\Big]\xrightarrow[n\to\infty]{}0$.

Now we are going to show that the sequence $\big(TX_n\big)$ is Cauchy in $L_G^1(\Omega)$.

Consider \ \ $\E\Big[\Big|\E\big[\|X_n\|^p_{L_2^\Sigma}\mid \Omega_t\big]-\E\big[\|X_m\|^p_{L_2^\Sigma}\mid \Omega_t\big]\Big|\Big]\\
\leq \E\Bigg[\E\Big[\big|\|X_n\|^p_{L_2^\Sigma}-\|X_m\|^p_{L_2^\Sigma}\big|\mid \Omega_t\Big]\Bigg]
\overset{\textbf{Prop.\ref{prop_propert_cond_exp},1)}}{=}\E\Big[\big|\|X_n\|^p_{L_2^\Sigma}-\|X_m\|^p_{L_2^\Sigma}\big|\Big]$

\hfill$
\overset{\textbf{1)}}{\leq}\widetilde{C}\cdot\Bigg(\E\Big[\|X_n-X_m\|^p_{L_2^\Sigma}\Big]\Bigg)^{1/p}
\xrightarrow[n,m\to\infty]{}0$, 

since the sequence $(X_n)$ is convergent in ${^\Hs\! L_G^p\big( \Omega \big)}$.
\\
So, we have that \ $\lim\limits_{n\to\infty}TX_n$ \ exists, because \ $L_G^1(\Omega)$ \ is complete.

If we take another sequence $(Y_n)$ of elements from $Lip(\Omega)$ such that $(TY_n)$ converges to $TX$, then we have

 $\E\Big[\big|TX_n-TY_n\big|\Big]
\leq\E\Big[\big|TX_n-TX\big|\Big]+\E\Big[\big|TY_n-TX\big|\Big]
\xrightarrow[n\to\infty]{}0
$.

And by 
\textbf{Lm.\ref{lm_X0_qs}} we see that \ $TX_n=TY_n$ \ q.s.

Thus we have an extension of conditional $G$ expectation to \ ${^\Hs\! L_G^p\big( \Omega \big)}$:

$\E\big[\|X\|^p_{L_2^\Sigma}\mid \Omega_t\big]:=\lim\limits_{n\to\infty}\E\big[\|X_n\|^p_{L_2^\Sigma}\mid \Omega_t\big]$ \ in \ $L_G^1(\Omega)$ \ norm.

\end{proof}

\begin{rem}\label{rem_prop_cond_G_exp}
From the definition of the operator $T$ in the proof of \textbf{Prop.\ref{prop_extens_G_exp}} and by passing to the limit, one can easily seen that \textbf{Prop.\ref{prop_propert_cond_exp}} holds for extended conditional $G$-expectation on \ ${^\Hs\! L_G^p\big( \Omega \big)}$ \ for \ $\fee=\|\pdot\|^p_{L_2^\Sigma}$. 
\end{rem}

\subsection{G-expectation and upper expectation}

In this chapter we consider a notion of an upper expectation and compare it with a $G$-expectation in order to see how they are related to one with another.
Later on we deduce see that they coincide on a considered above Banach space ${^\Hs\! L_G^p\big( \Omega \big)}$. The material for this theory is taken mainly from \cite{denis_hu_peng1}, where was considered only one-dimensional case. So, we follow this reference in order to give detailed description of the material in infinite dimensions. The proofs will be given only those ones that differ from the one-dimensional case, in the same time we will only point out those ones that can be just repeated.

In the next chapter dedicated to the construction of stochastic integral with respect to $G$-Brownian motion we widely use such a space ${^\Hs\! L_G^p\big( \Omega \big)}$ for it. There for a correct representation we shall be sure that ${^\Hs\! L_G^p\big( \Omega \big)}$ is not just an abstract completion of $Lip( \Omega )$ under the norm $\E\big[\|\pdot\|^p_{L_2^\Sigma}\big]$, but a space of random variables.
\\
Let $(\Omega,\Fc,P)$ be a probability space, 

where 
$\Omega=\big\{\omega_t:[0,+\infty)\rightarrow\Us\mid\omega_0=0,\ \omega_t {\mbox{--continuous}} \big\} $, \ the space of continuous trajectories as mentioned above.
$W_t$ is a cylindrical Wiener process in $\Us$ under the measure $P$.
$\Fc_t:=\sigma\{W_u,\ 0\leq u\leq t\}\vee \mathcal{N}$, where $\mathcal{N}$ is the collection of $P$-null subsets.

Fixing $\Sigma$-set we define a $G$-functional, i.e. \ $G(A)=\dfrac{1}{2}\sup\limits_{Q\in\Sigma}{\Tr}\big[A\cdot Q]$. It follows that there exists a bounded closed set $\Theta$ of Hilbert-Schmidt operators, such that \ 
$\Sigma=\big\{Q\,|\ Q=\gamma\cdot\gamma^T,\ \gamma\in\Theta\big\}$ \ and 

\hfill$G(A)=\dfrac{1}{2}\sup\limits_{\gamma\in\Theta}{\Tr}\big[\gamma\gamma^T\cdot A\big]$.

Let us define a following set of random processes \ 

$\At:=\big\{\theta_s:\Omega\to\Theta\,|\ s\in[t,T]\subset[0,\infty),\, \theta_s - (\Fc_s)\text{-adopted}\big\}$.

Since $\theta_s\in\Theta$ \ it follows that \ $\int\limits_t^T\Tr[\theta_s\,\theta_s^T]\,ds<\infty$.

For a $B$-continuous function $\fee\in{\mathbf{C}_{\,p.Lip}}(\Us^n)$ \  
let us consider such a function \ $v(t,x):=\sup\limits_{\theta\in\At}E_P\big[\fee(x+B_T^{t,\theta})\big]$, \ \ \ 
where $B_T^{t,\theta}:=\int\limits_t^T\theta_s\,dW_s,\ \ \theta_s\in\At$.
\begin{rem}
 Let $P_\theta$ be a law of the process $B_t^{0,\theta}=\int\limits_0^t\theta_s\,dW_s$. 
$P_\theta$ is a probability on $\Omega=\big\{\omega_t:[0,+\infty)\rightarrow\Us\mid\omega_0=0,\ \omega_t {\mbox{--continuous}} \big\} $. 

The elements of \ $\Omega$ \ are \ $B_t(\omega)=\omega_t$.

So, we have 

 $$E_P\big[\fee(B_{t_1}^{0,\theta},\ldots,B_{t_n}^{0,\theta})\big]
=E_{P_\theta}\big[\fee(B_{t_1},\ldots,B_{t_n})\big]\quad.$$

Note also that $B_t^{s,\theta}=B_t^{0,\theta}-B_s^{0,\theta}$ by definition.
Also in the same manner we denote that $B_t^s:=B_t-B_s$.
\end{rem}

\begin{teo}
 $v$ is a viscosity solution of the G-heat equation:
\begin{equation}
\begin{cases}
\partial_t v+G(D^2_{xx} v)=0 \,;\\
v(T,x)=\fee(x) \,.
\end{cases}
\end{equation}
\end{teo}
\begin{proof}
$\absz$\\
The proof completely coincide with the finite dimensional case (see \cite[Th.47]{denis_hu_peng1}). 

\end{proof}

Here we have that: 

$v(0,x)=\sup\limits_{\theta\in\A0}E_P\big[\fee(x+B_T^{0,\theta})\big]
=\sup\limits_{\theta\in\At}E_{P_\theta}\big[\fee(x+B_T)\big]
\equiv\E\big[\fee(x+B_T)\big]$, since such a viscosity solution is unique by \textbf{Th.\ref{th_uniq_vs}}.

For the future work we need the following auxiliary result that generalizes dynamical programming principle:

\begin{lm}\label{lm_indep_DHP}
 Let $\zeta\in L^2(\Omega,\Fc_s,P;\Hs)\,,\ \ 0\leq s\leq t\leq T$.

Then $\forall\fee\in\Cp(\Hs\tdot\Us\tdot\Us\to\R)$:

$$ess \sup\limits_{\theta\in\mathcal{A}_{s,T}^\Theta}E_P\big[\fee(\zeta,B_t^{s,\theta},B_T^{t,\theta})\big|\Fc_s\big]=
ess \sup\limits_{\theta\in\mathcal{A}_{s,t}^\Theta}E_P\big[\psi(\zeta,B_t^{s,\theta})|\Fc_s\big]\quad,$$
where $\psi(x,y):=ess\sup\limits_{\overline{\theta}\in\mathcal{A}_{t,T}^\Theta}E_P\big[\fee(x,y,B_T^{t,\overline{\theta}})\big|\Fc_t\big]
=\sup\limits_{\overline{\theta}\in\mathcal{A}_{t,T}^\Theta}E_P\big[\fee(x,y,B_T^{t,\overline{\theta}})\big]$.
\end{lm}
 
\begin{proof}
$\absz$\\
If $\fee\in\Cbp$ then the proof of this result is completely according to the finite dimensional case, see \cite[ Th.44]{denis_hu_peng1}.

Let $\fee\in\Cp$. Hence, the truncation \ $\fee_n:=\fee\cdot\ind_{\{|\fee|\leq n\}}+n\cdot\dfrac{\fee}{|\fee|}\cdot\ind_{\{|\fee|>n\}}$\ \ satisfies the statement of the lemma.
Also, it is clear that $\fee\in L_1$ w.r.t.$P$ (because of the gaussianity of $B_t^{s,\theta}$).
\hfill$(\,\ast\,)$
\\
And it remains just to estimate the following expression:
\\
$\Big|E_P\big[\fee(\pdot)\big|\Fc_s\big]-E_P\big[\fee_n(\pdot)\big|\Fc_s\big]\Big|
\leq E_P\Big[\Big|\fee(\pdot)-\fee_n(\pdot)\Big|\,\big|\Fc_s\Big]\\
=E_P\Big[\big(|\fee|(\pdot)-n\big)\cdot\ind_{\{|\fee|>n\}}\big|\Fc_s\Big]
\leq E_P\Big[|\fee|(\pdot)\cdot\ind_{\{|\fee|>n\}}\big|\Fc_s\Big]
\xrightarrow[n\to\infty]{(\,\ast\,)}0$\,,\ \ \  by the dominated convergence theorem.

\end{proof}

\begin{prop}\label{prop_EG_Ebar}According to the imposed above notations there holds
\\
 $\E\big[\fee(B_{t_1}^0,\ldots,B_{t_n}^{t_n-1})\big]
=\sup\limits_{\theta\in\A0}E_P\big[\fee(B_{t_1}^{0,\theta},\ldots,B_{t_n}^{t_{n-1},\theta})\big]$

\hfill$
=\sup\limits_{\theta\in\A0}E_{P_\theta}\big[\fee(B_{t_1}^0,\ldots,B_{t_n}^{t_n-1})\big]$,

\end{prop}

\begin{proof}
$\absz$\\
The proof you can find in the author's thesis \cite{ibragimov1}.

\end{proof}
 
\begin{prop}\label{prop_P_tight}
 The family of probability measures $\big\{P_\theta\,,\ \theta\in\mathcal{A}_{0,\infty}^\Theta\big\}$ is tight.
\end{prop}

\begin{proof}
$\absz$\\
The proof is the same as in finite-dimensional case (see \cite[Prop.49]{denis_hu_peng1}).

\end{proof}

Now let us consider the following spaces:
\begin{itemize}
\item $L^0(\Omega)$: the space of $\mathcal{B}(\Omega)$-measurable real functions;
\item $L^0_{L_2^\Sigma}(\Omega)$: the space of $\mathcal{B}(L_2^\Sigma)/\mathcal{B}(\Omega)$-measurable $\Omega\to L_2^\Sigma$ mappings;
\item $B_{L_2^\Sigma}(\Omega)$: all bounded mappings in $L^0_{L_2^\Sigma}(\Omega)$;
\item $CB_{L_2^\Sigma}(\Omega)$: all bounded and continuous mappings in $L^0_{L_2^\Sigma}(\Omega)$.
\end{itemize}

\begin{df}
For a random variable $X\in L^0(\Omega)$ such that a linear (classical) expectation $E_P$ exists for all $P\in\mathcal{P}$ an \textbf{upper expectation} of $\mathcal{P}$ is defined as follows:
\begin{equation*}
 \bar{\E}[X]:=\sup\limits_{P\in\mathcal{P}}E_P[X].
\end{equation*}
\end{df}

Later on we consider such a set of probability measures

\hfill $\mathcal{P}:=\big\{P_\theta\,,\ \theta\in\mathcal{A}_{0,\infty}^\Theta\big\}$.

\begin{rem}
It is clear that on the expectation space $(\Omega, B_{L_2^\Sigma}(\Omega), \bar{\E})$ as well as on the $(\Omega, CB_{L_2^\Sigma}(\Omega), \bar{\E})$ the upper expectation $\bar{\E}[\pdot]$ is a sublinear expectation.
\end{rem}

\begin{rem}\label{rem_EG_Ebar_Lip}
For a $G$-expectation $\E$ by \textbf{Prop.\ref{prop_EG_Ebar}} we have that
$$\E[\psi(X)]=\bar{\E}[\psi(X)]\,, \ \ \ X\in Lip(\Omega)\,,\ \psi\in\Cp(L_2^\Sigma)\ .$$
For $X\in L^0_{L_2^\Sigma}(\Omega)$ we define the norm $\overline{\|X\|}^p_{\Sigma,\,p}:=\bar{\E}\big[\|X\|^p_{L_2^\Sigma}\big]$.
\\
Therefore \ $\|X\|^p_{\Sigma,\,p}:=\E\big[\|X\|^p_{L_2^\Sigma}\big]=\bar{\E}\big[\|X\|^p_{L_2^\Sigma}\big]=\overline{\|X\|}^p_{\Sigma,\,p}\ \,, \  X\in Lip(\Omega)$.
\end{rem}
 
Define the following spaces:
\begin{itemize}
\item $\mathcal{L}^p:=\big\{X\in L^0_{L_2^\Sigma}(\Omega)\mid\bar{\E}\|X\|^p_{L_2^\Sigma}<\infty\big\}$;
\item $\mathcal{N}:=\big\{X\in L^0_{L_2^\Sigma}(\Omega)\mid X=0\ \ \ c\text{ - q.s.}\big\}$;
\item $\L^p_{L_2^\Sigma}:=\mathcal{L}^p/\mathcal{N}$;
\item $\L^p_{B,L_2^\Sigma}$ is the completion of $B_{L_2^\Sigma}(\Omega)/\mathcal{N}$ under $\overline{\|\pdot\|}_{\Sigma,\,p}$;
\item $\L^p_{CB,L_2^\Sigma}$ is the completion of $CB_{L_2^\Sigma}(\Omega)/\mathcal{N}$ under $\overline{\|\pdot\|}_{\Sigma,\,p}$.
\end{itemize}

\begin{rem}
 Similar to the classical arguments of $L^p$-theory we can conclude that $(\L^p_{L_2^\Sigma},\,\overline{\|\pdot\|}_{\Sigma,\,p}\big)$ is a Banach space.

\end{rem}
 
Since\ \  $CB_{L_2^\Sigma}(\Omega)\subset B_{L_2^\Sigma}(\Omega)\subset \mathcal{L}^p$
$\ \ \ \Rightarrow\ \ \ CB_{L_2^\Sigma}(\Omega)/\mathcal{N}\subset B_{L_2^\Sigma}(\Omega)/\mathcal{N}\subset \L^p_{L_2^\Sigma}$.

So we have the following inclusions: $\L^p_{CB,L_2^\Sigma}\subset\L^p_{B,L_2^\Sigma}\subset\L^p_{L_2^\Sigma}$.
 \begin{prop}\label{prop_B}
 $$\L^p_{B,L_2^\Sigma}=
\big\{X\in \L^p_{L_2^\Sigma}(\Omega)\mid\lim\limits_{n\to\infty}\bar{\E}\big[\|X\|^p_{L_2^\Sigma}\cdot
\ind_{\{\|X\|_{L_2^\Sigma}>n\}}\big]=0\big\}.$$
\end{prop}

\begin{proof}
 $\absz$\\
There is no changes with a proof for a finite-dimensional case (see \cite[Prop.18]{denis_hu_peng1}).

\end{proof}

\begin{df}
 A mapping $X$ on $\Omega$ with values in a topological space is said to be \textbf{quasi-continuous} if for every \ $\varepsilon>0$ \ there exists \ an open set \ $O$ with capacity \ $c(O)<\varepsilon$, \ such that the restriction of mapping $X$ to the complement $O^c$ is continuous.
\end{df}

\begin{df}
 We say that mapping $X$ on $\Omega$ has a quasi-continuous version if there exists a quasi-continuous mapping $Y$ on $\Omega$, such that $X=Y$ quasi surely. 
\end{df}
 
\begin{prop}\label{prop_CB_qcv}
If $X\in\L^p_{CB,L_2^\Sigma}$ \ then \ $X$ has a quasi-continuous version.
\end{prop}

\begin{proof}
 $\absz$\\
There is also no changes with a proof for a finite-dimensional case (see \cite[Prop.24]{denis_hu_peng1}).

\end{proof}

\begin{prop}\label{prop_CB}
 $$\L^p_{CB,L_2^\Sigma}=
\big\{X\in \L^p_{L_2^\Sigma}(\Omega)\mid X \text{ has a q.c. vers., }
\lim\limits_{n\to\infty}\bar{\E}\big[\|X\|^p_{L_2^\Sigma}\cdot\ind_{\{\|X\|_{L_2^\Sigma}>n\}}\big]=~\!\!0\big\}.$$
\end{prop}

\begin{proof}
 $\absz$\\
For this proof we follows \cite[Prop.25]{denis_hu_peng1}, but here there are some difficulties when we pass to infinite-dimensions. So here we give the whole proof with all details.

Let us denote a set \  

$A:=\big\{X\in \L^0_{L_2^\Sigma}(\Omega)\mid X \text{ has a q.c. vers., }
\lim\limits_{n\to\infty}\bar{\E}\big[\|X\|^p_{L_2^\Sigma}\cdot\ind_{\{\|X\|_{L_2^\Sigma}>n\}}\big]=0\big\}$.

\textbf{1)} If $X\in \L^p_{CB,L_2^\Sigma}$ \ then by \textbf{Prop.\ref{prop_CB_qcv}} we have that  $X$ has a quasi-continuous version.

Also  for $X\in \L^p_{B,L_2^\Sigma}$ \ by \textbf{Prop.\ref{prop_B}} we have \ $
\lim\limits_{n\to\infty}\bar{\E}\big[\|X\|^p_{L_2^\Sigma}\cdot\ind_{\{\|X\|_{L_2^\Sigma}>n\}}\big]=0$.

So we can conclude that \ $ X\in A$.
\\
\textbf{2)} Every $X\in A$ is quasi-continuous.

Take the truncation 

$X^t_n:=\Trunc(X,n)\equiv X\cdot\ind_{\{\|X\|_{L_2^\Sigma}\leq n\}}+n\cdot\dfrac{X}{\phantom{x,}\|X\|_{L_2^\Sigma}}\cdot\ind_{\{\|X\|_{L_2^\Sigma}>n\}}$.
\\\parbox{140mm}{
Hence \ $\bar{\E}\big[\|X-X^t_n\|^p_{L_2^\Sigma}\big]
=\bar{\E}\Big[\|X\|^p_{L_2^\Sigma}\cdot\big(1-\dfrac{n}{\phantom{x,}\|X\|_{L_2^\Sigma}}\big)^p\cdot\ind_{\{\|X\|_{L_2^\Sigma}>n\}}\Big]$

\hfill$
\phantom{\dfrac{1}{1}}\leq\bar{\E}\big[\|X\|^p_{L_2^\Sigma}\cdot\ind_{\{\|X\|_{L_2^\Sigma}>n\}}\big]
\xrightarrow[n\to\infty]{}0$.}

Since $X^t_n$ is quasi-continuous then there exists a closed set $A_n$, such that $c(A_n^c)<\dfrac{1}{n^{p+1}}$\ \ \ and $X_n$
is continuous on $A_n$.

For the following step we need to use the Dugundji theorem (the infinite-dimensional version of the Tietze extension theorem): 

\begin{teo}[Dugundji]\label{th_Dugundji}\cite[Th.1.2.2.]{mill1}.\\
Let $L$ be a locally convex linear space and let $C\subseteq L$ be convex. Then for every space $\Omega$ with closed subspace $A$, every continuous function \mbox{$f:A\to C$} can be extended to a continuous function $\bar{f}:\Omega\to C$.
\\
As is known that every normable space is locally convex it follows that the Banach space $L_2^\Sigma$ is also locally convex.

\end{teo}

We have \  $\|X^t_n\|_{L_2^\Sigma}\leq \|X\|_{L_2^\Sigma}\cdot\ind_{\{\|X\|_{L_2^\Sigma}\leq n\}}+n\cdot\ind_{\{\|X\|_{L_2^\Sigma}>n\}}\leq n$.

Define a set \ $C_n:=\big\{Y\in L_2^\Sigma\mid \|Y\|_{L_2^\Sigma}\leq n\big\}$.

So that  \ $ X^t_n\big|_{A_n}:A_n\to C_n$ is continuous.

By the Dugundji theorem (\textbf{Th.\ref{th_Dugundji}}) it follows that there exists a continuous extension $\widetilde{X}_n:\Omega\to C_n$, which satisfies the following:

$\widetilde{X}_n\in CB_{L_2^\Sigma}(\Omega)$,

$X^t_n= \widetilde{X}_n$ on $A_n$,

$\|\widetilde{X}_n\|_{L_2^\Sigma}\leq n$.
\\
Hence \ $\bar{\E}\big[\|X^t_n- \widetilde{X}_n\|^p_{L_2^\Sigma}\big]
\leq \bar{\E}\Big[\Big(\|X^t_n\|^p_{L_2^\Sigma}+\|\widetilde{X}_n\|^p_{L_2^\Sigma}\Big)\cdot\ind_{\{C_n\setminus A_n\}}\Big]$

\hfill$
\leq 2n^p\cdot c\big(A_n^c\big)=\dfrac{2n^p}{n^{p+1}}=\dfrac{2}{n}\xrightarrow[n\to\infty]{}0$.
\\
And we have \ 

\hfill$\bar{\E}\big[\|X- \widetilde{X}_n\|^p_{L_2^\Sigma}\big]
\leq2^{p-1}\Big(\bar{\E}\big[\|X- X^t_n\|^p_{L_2^\Sigma}\big]+\bar{\E}\big[\|X^t_n- \widetilde{X}_n\|^p_{L_2^\Sigma}\big]\Big)
\xrightarrow[n\to\infty]{}0$.

So that \ $X\in \L^p_{CB,L_2^\Sigma}$.

\end{proof}

\newpage
\section{Stochastic Integral with respect to G-Brownian motion}\label{sec_stoch_int}

\subsection[Definition of the stochastic integral]{Definition of the stochastic integral for elementary integrated processes}\label{subs_def_stoch_int}

Recall that 

$Lip( \Omega_T ):=\Big\{ \fee(B_{t_1},..,B_{t_n})\mid n\geq1,\ 
t_1,..,t_n\in[0,T],\  \fee\in \Cp\,(\Us^n, L_2^\Sigma)\Big\} $.

For $p\geq1 $ \ the space ${^\Hs\! L_G^p\big( \Omega_T \big)}$ is a completion of $Lip( \Omega_T )$ 
under the norm  $\|X\|^p_{\Sigma,\,p}=\E\big[\|X\|^p_{L_2^\Sigma}\big]$.

Hold the notation $\|\pdot\|_\Sigma:=\|X\|_{\Sigma,\,2}$.

Since  ${^\Hs\! L_G^p\big( \Omega_T \big)}$ is an abstract completion we need to show that it is a space of random variables what provides the following theorem:

\begin{rem}
Let us denote the following subspace of $Lip( \Omega_T )$:

$BLip( \Omega_T ):=\Big\{ \fee(B_{t_1},..,B_{t_n})\mid n\geq1,\ 
t_1,..,t_n\in[0,T],\  \fee\in \Cbp\,(\Us^n, L_2^\Sigma)\Big\} $.

And the space ${^\Hs\! BL_G^p\big( \Omega_T \big)}$ \ be a completion of $BLip( \Omega_T )$ 
under the norm $\|X\|^p_{\Sigma,\,p}=\E\big[\|X\|^p_{L_2^\Sigma}\big]$.

Then we have that ${^\Hs\! BL_G^p\big( \Omega_T \big)}={^\Hs\! L_G^p\big( \Omega_T \big)}$.
\end{rem}

\begin{proof}
 $\absz$\\
If $X\in {^\Hs\! L_G^p\big( \Omega_T \big)}$ \ then there exists a sequence \ $(X_k)\subset Lip( \Omega_T )$, \ such that \ $\E\big[\|X-X_k\|^p_{L_2^\Sigma}\big]
\xrightarrow[n\to\infty]{}0$ \ and \ $\E\big[\|X_k\|^p_{L_2^\Sigma}\big]<\infty$.

Take $X_k^n:=\Trunc(X_k,n)\in BLip( \Omega_T )$.
\\
Therefore \ $\|X-X_k^n\|^p_{\Sigma,\,p}=\E\big[\|X-X_k^n\|^p_{L_2^\Sigma}\big]
\leq\E\big[\|X\|^p_{L_2^\Sigma}\cdot\ind_{\{\|X\|_{L_2^\Sigma}>n\}}\big]
\xrightarrow[n\to\infty]{}0$.

So we can deduce that $\mathrm{compl}\Big( BLip( \Omega_T ),\,\|\pdot\|_{\Sigma,\,p}\Big)
={^\Hs\! L_G^p\big( \Omega_T \big)}
$.

\end{proof}

\begin{teo}
$\absz$\\
  ${^\Hs\! L_G^p\big( \Omega_T \big)}=
\big\{X\in L^0_{L_2^\Sigma}(\Omega_T)\mid X \text{ has a q.c. vers., }$

\hfill$\lim\limits_{n\to\infty}\bar{\E}\big[\|X\|^p_{L_2^\Sigma}\cdot\ind_{\{\|X\|_{L_2^\Sigma}>n\}}\big]=0\big\}$.

Moreover, if $X\in{^\Hs\! L_G^p\big( \Omega_T \big)}$ \ then \ $\E\big[\|X\|^p_{L_2^\Sigma}\big]=\bar{\E}\big[\|X\|^p_{L_2^\Sigma}\big]$.

\end{teo}

\begin{proof}
 $\absz$

Let $\Phi\in CB_{L_2^\Sigma}(\Omega_T)$, i.e. $\Phi:\Omega_T\to L_2^\Sigma$ is bounded and continuous \mbox{$L_2^\Sigma$-valued} random variable.

Take a compact set $K\subset \Omega_T=\big\{\omega_t:[0,T]\to\Us\mid\omega_0=0,\ \omega_t {\mbox{--continuous}} \big\} $.
\\
We claim that $\Phi$ be uniformly continuous on $K$.

In fact, since $\Phi$ is bounded and continuous on the compact set $K$ \ then for every \ $\omega\in K$ \ there exists a subsequence \ $\omega_{n_j}\xrightarrow[j\to\infty]{}\omega$, \ such that 
\phantom{xxxxdddddd}\\  $\Phi(\omega_{n_j})\xrightarrow[j\to\infty]{\|\, \cdot\, \|_{L^\Sigma_2}}\Phi(\omega)$.\hfill$(\,\ast\,)$

If $\Phi$ is not uniformly continuous on $K$ then 
$\exists\varepsilon>0\ \ \ \forall\delta_n\downarrow0\ \ \ \mbox{$\forall\omega_n,\,\omega'_n\in\Omega_T$}$,\ such that \  $\|\omega_n-\omega'_n\|_\Omega<\delta$ \ and it follows \  
$\|\Phi(\omega_n)-\Phi(\omega'_n)\|_{L^\Sigma_2}<\varepsilon$.

Let us take two different subsequences:
 $\omega_{n_j}\xrightarrow[j\to\infty]{}\omega$ \ and \ $\omega'_{n_j}\xrightarrow[j\to\infty]{}\omega$.

So that \ $\exists\  \bar{j}\ \ \ \forall j>\bar{j}:\ \ \ \|\omega_{n_j}-\omega'_{n_j}\|_\Omega<\delta_n$,
\\
Hence \ $\varepsilon<\|\Phi(\omega_{n_j})-\Phi(\omega'_{n_j})\|_{L^\Sigma_2}
$

\hfill$\leq \|\Phi(\omega_{n_j})-\Phi(\omega)\|_{L^\Sigma_2}
+\|\Phi(\omega'_{n_j})-\Phi(\omega)\|_{L^\Sigma_2}\xrightarrow[j\to\infty]{(\,\ast\,)}0$, \ a contradiction.
\\
Therefore \ $\forall\varepsilon>0\ \ \ \exists\delta>0\ \ \ \forall\omega,\,\omega'\in\Omega_T$,\ such that \ $\|\omega-\omega'\|_\Omega<\delta$  \ and it follows \ $\|\Phi(\omega)-\Phi(\omega')\|_{L^\Sigma_2}<\varepsilon$.
\\
Take a partition $\pi_n=\big\{0=t_1^n<\ldots<t_{N_n}^n=T\big\}$,\ \ \ $|\pi_n|:=\max\limits_{1\leq i\leq N_n}|t_{i+1}^n-t_i^n|$.

Then for every $\omega\in K\ \ \ \pi_n\omega$ is a linear approximation of $\omega$ at the points of the given partion $\pi_n$, 
i.e.\ \ \ $\pi_n\omega(t_i^n)=\omega(t_i^n)$.

Define $\Phi_n(\omega):=\Phi(\pi_n\omega)=\widetilde{\Phi}(t_1^n,\ldots,t_{N_n}^n)\in BLip( \Omega_T )$.
\\
Owning to the Arzelà-Ascolì theorem we have that $K$ is an equicontinuous set: \ 
$\forall\omega\in K\ \ \ \forall\delta>0\ \ \ \exists\varkappa>0\ \ \ \forall t,\,t'\in[0,T]$,\ such that \  $|t-t'|<\varkappa$  \ and it follows \ $\|\omega(t)-\omega(t')\|_\Us<\delta$.\hfill$(\,\#\,)$
\\
Now for a fixed $t\in(t_i,t_{i+1})$ \ we have \ 

$\pi_n\omega(t)
=\dfrac{t-t_i}{t_{i+1}-t_i}\cdot\omega(t_{i+1})+\dfrac{t_{i+1}-t}{t_{i+1}-t_i}\cdot\omega(t_i)$.

Let \ \ \ $|\pi_n|:=\eta<\varkappa$ \ then let us calculate 
\\
$\|\pi_n\omega(t)-\omega(t)\|_\Us$

\hfill$
=\Big\|\dfrac{t-t_i}{t_{i+1}-t_i}\cdot\omega(t_{i+1})+\dfrac{t_{i+1}-t}{t_{i+1}-t_i}\cdot\omega(t_i)
-\dfrac{t-t_i}{t_{i+1}-t_i}\cdot\omega(t)-\dfrac{t_{i+1}-t}{t_{i+1}-t_i}\cdot\omega(t)\Big\|_\Us\\
\leq\dfrac{t-t_i}{t_{i+1}-t_i}\cdot\|\omega(t_{i+1})-\omega(t)\|_\Us
+\dfrac{t_{i+1}-t}{t_{i+1}-t_i}\cdot\|\omega(t_i)-\omega(t)\|_\Us$

\hfill$
\overset{(\,\#\,)}{<}\dfrac{t-t_i}{t_{i+1}-t_i}\cdot\delta
+\dfrac{t_{i+1}-t}{t_{i+1}-t_i}\cdot\delta
=\delta$.
\\
Hence \ $\|\pi_n\omega-\omega\|_\Omega
=\max\limits_{\substack{t\in(t_i,t_{i+1})\\ 1\leq i\leq N_n}}\|\pi_n\omega(t)-\omega(t)\|_\Us
<\delta$.

It means, $\ \ \forall\varepsilon>0\ \ \ \exists\eta>0$,\  such that \ $ |\pi_n|<\eta$  \ and it follows \ 

\hfill$
\|\Phi(\pi_n\omega)-\Phi(\omega)\|_{L^\Sigma_2}<\varepsilon$.

Since $|\pi_n|\xrightarrow[n\to\infty]{}0$ \ then \ $
\sup\limits_{\omega\in K}\|\Phi(\omega)-\Phi_n(\omega)\|_{L^\Sigma_2}\xrightarrow[n\to\infty]{}0$.
\\
Recall that $\bar{\E}[\pdot]:=\sup\limits_{P\in\mathcal{P}}E_P[\pdot]$\ \ \ and\ \ \ $\mathcal{P}$ is tight (\textbf{Prop.\ref{prop_P_tight}}).

It follows that for every \ $ n\geq1$ \ there exists a compact set $K_n\subset\Omega_n$, \ such that \ \ $c(K_n^c)<\dfrac{1}{n}$.

Therefore \ $\forall \Phi\in CB_{L_2^\Sigma}(\Omega_T) \ \ \ \exists\Phi_n\in BLip( \Omega_T )$, \ such that 

\hfill$\sup\limits_{\omega\in K_n}\|\Phi(\omega)-\Phi_n(\omega)\|_{L^\Sigma_2}<\dfrac{1}{n}$.

So that \  $\bar{\E}\big[\|\Phi-\Phi_n\|^p_{L_2^\Sigma}\big]
\leq\sup\limits_{\omega\in \Omega}\underbrace{\|\Phi(\omega)-\Phi_n(\omega)\|^p_{L^\Sigma_2}}_{<\infty}\cdot c(K_n^c)+\dfrac{1}{n}\cdot c(K_n)$

\hfill$
\leq \dfrac{1}{n}\cdot(C_0+1)
\xrightarrow[n\to\infty]{}0
$.

Hence \ $\Phi\in{^\Hs\! L_G^p \big( \Omega_T \big)}$.
\\
It is clear that $BLip( \Omega_T )\subset CB_{L_2^\Sigma}(\Omega_T)$.
\\
So we have the following inclusions:
\\
 \hspace*{50mm}$BLip( \Omega_T )\subset CB_{L_2^\Sigma}(\Omega_T)\subset{^\Hs\! L_G^p \big( \Omega_T \big)}$.
\hfill$(\,\circ\,)$
\\
Recall that\ \ \  $\L^p_{CB,L_2^\Sigma}=\mathrm{compl}\big(CB_{L_2^\Sigma}/\mathcal{N},\,\overline{\|\pdot\|}_{\Sigma,\,p}\big)$

$\overset{\textbf{Prop.\ref{prop_CB}}}{=}\big\{X\in \L^p_{L_2^\Sigma}(\Omega)\mid X \text{ has a q.c. vers., }
\lim\limits_{n\to\infty}\bar{\E}\big[\|X\|^p_{L_2^\Sigma}\cdot\ind_{\{\|X\|_{L_2^\Sigma}>n\}}\big]=~\!\!0\big\}
$;
\\
$\L^p_{L_2^\Sigma}=\mathcal{L}^p/\mathcal{N}
=\big\{X\in L^0_{L_2^\Sigma}(\Omega)\mid\bar{\E}\|X\|^p_{L_2^\Sigma}<\infty\big\}/\mathcal{N}$.
\\
But\ \ \ $\bar{\E}\big[\|X\|^p_{L_2^\Sigma}\cdot\ind_{\{\|X\|_{L_2^\Sigma}>n\}}\big]\xrightarrow[n\to\infty]{}0$ \ so that \ $
\bar{\E}\|X\|^p_{L_2^\Sigma}<\infty$.
\\
Then by $(\,\circ\,)$ we have \ 
${^\Hs\! L_G^p\big( \Omega_T \big)}
\equiv\mathrm{compl}\big(CB_{L_2^\Sigma}/\mathcal{N},\,\overline{\|\pdot\|}_{\Sigma,\,p}\big)\\
=\big\{X\in \mathcal{L}^p\mid X \text{ has q.c. version, }
\lim\limits_{n\to\infty}\bar{\E}\big[\|X\|^p_{L_2^\Sigma}\cdot\ind_{\{\|X\|_{L_2^\Sigma}>n\}}\big]=0\big\}\\
=\big\{X\in L^0_{L_2^\Sigma}(\Omega_T)\mid X \text{ has q.c. version, }
\lim\limits_{n\to\infty}\bar{\E}\big[\|X\|^p_{L_2^\Sigma}\cdot\ind_{\{\|X\|_{L_2^\Sigma}>n\}}\big]=0\big\}$.
\\
If $X\in Lip(\Omega)$ \ then \ $\E\big[\|X\|^p_{L_2^\Sigma}\big]=\bar{\E}\big[\|X\|^p_{L_2^\Sigma}\big]$.

Since\ \ \ ${^\Hs\! L_G^p\big( \Omega_T \big)}
=\mathrm{compl}\Big( BLip( \Omega_T ),\,\E\big[\|\pdot\|^p_{L_2^\Sigma}\big]\Big)$

\hfill$
=\mathrm{compl}\Big( BLip( \Omega_T ),\,\bar{\E}\big[\|\pdot\|^p_{L_2^\Sigma}\big]\Big)
$,

It follows that for every \ $X\in{^\Hs\! L_G^p\big( \Omega_T \big)}:\ \ \ \E\big[\|X\|^p_{L_2^\Sigma}\big]=\bar{\E}\big[\|X\|^p_{L_2^\Sigma}\big]$.

\end{proof}

Now we need speak a couple of words about convergence in ${^\Hs\! L_G^p\big( \Omega_T \big)}$ under the $G$-expectation. The biggest problem is that a dominated convergence theorem is not true in a given framework. But there is a result of a monotone convergence theorem (see \cite[Th.31]{denis_hu_peng1}). It has a trivial infinite dimensional extension which we will represent as the following theorem:

\begin{teo}\label{th_mon_conv_g_exp}
 Let a sequence $\big\{X_n\big\}\subset{^\Hs\! L_G^p\big( \Omega_T \big)}$ be such that  $X_n\downarrow X$ q.s. 
Then also holds \ $\E\big[\|X_n\|^p_{L_2^\Sigma}\big]\downarrow\E\big[\|X\|^p_{L_2^\Sigma}\big]$.
\end{teo}

So we are ready to define a space of the integrated processes:

Let ${^\Hs\!M^{p,0}_G\big( 0,T \big)}:=\Bigg\{\Phi(t)=\sum\limits_{k=0}^{N-1}\Phi_k(\omega) \ind_{[t_k,t_{k+1})}(t)\mid$

\hfill$
\Phi_k (\omega)\in {^\Hs}L_G^p( \Omega_{t_k} ),\ 0=t_0<t_1\ldots<t_N=T\Bigg\}$.

For elementary process $\Phi\in {^\Hs\!M^{p,0}_G\big( 0,T \big)}$ the stochastic integral we define as follows:
$$I_T(\Phi):=\int\limits_0^T \Phi(t) dB_t=\sum\limits_{k=0}^{N-1}\Phi_k ( B_{t_{k+1}}-B_{t_k})\ \ .$$

\begin{rem}\label{rem_st_int_shift}
It is clear that $\ \ \ \int\limits_0^T \Phi(t) dB_t=\int\limits_0^T \Phi(t) dB_{t-a}\,,\ \ a\in\R$.
\end{rem}

\begin{rem}
 For a fixed $t$ we have that \ $\Phi(t)\in L^0_{L_2^\Sigma}(\Omega)$ then we can see that $I_T(\Phi)$ is an $H$-valued random variable.
\end{rem}

\subsection{Itô's isometry and Burkholder–Davis–Gundy inequalities}

Let $B_t$ is a given $G$-Brownian motion with correspondent $G$-expectation $\E$ which coincides on the space $L^0_{L^\Sigma_2}$ with an upper expectation defined for the family of gaussian probability measures $\mathcal{P}$:
 $$\E[X]=\bar\E[X]:=\sup\limits_{P\in\mathcal{P}}E_P[X]\ .$$

If we fix a measure $P\in\mathcal{P}$ then we can define
$\E_P[X]:=\sup\limits_{P\in\mathcal{P}_P}E_P[X]$,\ \ where $\mathcal{P}_P:=\big\{P\big\}$. 

Actually, $\E_P[X]$ is a classical linear expectation, but we could treat it also as sublinear with the all properties for a sublinear expectation.

So, we can define the $G$-functional for this expectation:

$$G_P(A):=\dfrac{1}{2}\,\E_P\big[\langle  AX,X \rangle\big]=\dfrac{1}{2}\Tr[A\cdot Q_P],$$ 

for some non-negative symmetric trace-class operator $Q_P$, because functional $\E_P[X]$ is linear.

On the other hand $G_P(A)=\dfrac{1}{2}\sup\limits_{Q\in\Sigma_P}\Tr[A\cdot Q]$,\ \  where $\Sigma_P\subset\Sigma$.

Therefore \ $\Sigma_P\equiv\big\{Q_P\big\}$\ \  and\ \  $B_t$ is a classical $Q_P$-Wiener process under $P\in\mathcal{P}\,,\ Q_P\in\Sigma$.

\begin{teo}[Itô's isometry inequality]\label{th_ito_is_inq}
$\absz$\\
Let $\Phi\in{^\Hs\!M^{2,0}_G\big( 0,T \big)}$ then
\begin{equation}\label{eq_ito_is_inq}
\E\Big[\big\|\int\limits_0^T \Phi(t)\,dB_t\big\|^2_\Hs\Big]
\leq  \E\Big[\int\limits_0^T \|\Phi(t)\|^2_{L_2^\Sigma}\,dt\Big].
\end{equation}
\end{teo}

\begin{proof}
$\absz$\\
\textbf{1)} Firstly, we are going to prove a ``weaker''  version of the Itô's isometry inequality, namely that
\begin{equation}\label{eq_weak_iii}
 \E\Big[\big\|\int\limits_0^T \Phi(t)\,dB_t\big\|^2_\Hs\Big]
\leq  \int\limits_0^T\E\Big[ \|\Phi(t)\|^2_{L_2^\Sigma}\Big]\,dt.
\end{equation}
In fact,

$\E\Big[\big\|\int\limits_0^T \Phi(t)\,dB_t\big\|^2_\Hs\Big]=
\E\Big[\|\sum\limits_{k=0}^{N-1} \Phi_k\,\Delta B_{t_k}\|^2_\Hs\Big]$

\hfill$\leq\sum\limits_{k=0}^{N-1} \E\|\Phi_k\,\Delta B_{t_k}\|^2_\Hs+
2\!\!\!\sum\limits_{k<n=1}^{N-1} \E\langle\Phi_k\,\Delta B_{t_k},\Phi_n\,\Delta B_{t_n}\rangle_\Hs=:K$.
\\
\hspace*{6mm}\textbf{(a)}\parbox[t]{130mm}{$\mathbf{k<n}$:\\
Note that in this case for fixed $n$ and $k$ we have that $\Phi_k\,\Delta B_{t_k}$ and $\Delta B_{t_n}$ are independent from $\Omega_{t_n}$, but 
$\Phi_n$ is $\Omega_{t_n}$-measurable.\\
Therefore \ $\E\langle\Phi_k\,\Delta B_{t_k},\Phi_n\,\Delta B_{t_n}\rangle_\Hs=
\E\Big[\E\langle\Phi_k\,\Delta B_{t_k},\Phi_n\,\Delta B_{t_n}\rangle_\Hs\big|\,\Omega_{t_n}\Big]
$

\hfill$\phantom{\dfrac{1}{1}}
=\E\langle\Phi_k\,\Delta B_{t_k},\Phi_n\cdot\underbrace{\E[\Delta B_{t_n}]}_{=0}\rangle_\Hs=0.$
}

\hspace*{6mm}\textbf{(b)}\parbox[t]{130mm}{$\mathbf{k=n}$:\\
Analogously, $\Delta B_{t_k}$ is independent out of $\Omega_{t_k}$, but $\Phi_k$ is $\Omega_{t_k}$ measurable.\\
Therefore \ $\E\big[\|\Phi_k\,\Delta B_{t_k}\|^2_\Hs\big]=
\E\Big[\E\big[\|\Phi_k\,\Delta B_{t_k}\|^2_\Hs\big|\,\Omega_{t_k}\big]\Big]$

\hfill$
\overset{\substack{\textbf{Prop.\ref{prop_propert_cond_exp},1)}\\ \textbf{Rem.\ref{rem_prop_cond_G_exp}} \\ \textbf{Prop.\ref{prop_ar_op_gnd},\,3)} }}{=}
\E\Big[(t_{k+1}-t_{k})\cdot\sup\limits_{Q\in\Sigma}\Tr\big[\Phi_k Q \Phi_k^*\big]\Big].$
}

So we have \ $K=\sum\limits_{k=0}^{N-1} \E\Big[(t_{k+1}-t_{k})\cdot \|\Phi_k\|^2_{L_2^\Sigma}\Big]=
\int\limits_0^T \E\Big[\|\Phi(t)\|^2_{L_2^\Sigma}\Big] dt
$.
\\
\textbf{2)} If $\Phi\in{^\Hs\!M^{2,0}_G\big( 0,T \big)}$ \ then \ $ \Phi(t)=\sum\limits_{k=0}^{N-1}\Phi_k(\omega) \ind_{[t_k,t_{k+1})}(t),$

\hfill$ 0=t_0<t_1\ldots<t_N=T$ 

and for every $k\ \ \ \Phi_k \in {^\Hs}L_G^2( \Omega_{t_k} )$ \ then there exists a sequence

 $\big\{\Phi_k^{(n)}\big\}\subset Lip( \Omega_{t_k} )$, \ such that \ $\big\|\Phi_k-\Phi_k^{(n)}\big\|_\Sigma\xrightarrow[n\to\infty]{}0$.

Define $\Phi^{(n)}(t):=\sum\limits_{k=0}^{N-1}\Phi^{(n)}_k(\omega) \ind_{[t_k,t_{k+1})}(t)$.

Since $\Phi_k^{(n)}\in Lip( \Omega_{t_k} )$ \ then we can consider 

\hfill$\Phi_k^{(n)}=\fee_k^{(n)}(B_{u_1^k},..,B_{{u_{m_k}^k}})\in L_2^\Sigma\,,\ \ 
0\leq u_1^k\leq\ldots\leq u_{m_k}^k=t_k$.

Therefore \ $I_T(\Phi_k^{(n)})=\int\limits_0^T \Phi^{(n)} (t)\,dB_t=\sum\limits_{k=0}^{N-1}\Phi_k^{(n)} ( B_{t_{k+1}}-B_{t_k})
$

\hfill$=\fee(B_{u_1^0},..,B_{{u_{m_N}^N}})\in\Hs\,,\ \ 
\fee\in \Cp\,(\Us^{m_N}, \Hs)$.
\\
So, we have

$ \E\Big[\big\|\int\limits_0^T \Phi^{(n)}(t)\,dB_t\big\|^2_\Hs\Big]
=\E\Big[\|\fee\|^2_\Hs(B_{u_1^0},..,B_{{u_{m_N}^N}})\Big]
=\sup\limits_{\theta\in\A0}E_{P_\theta}\Big[\|\fee\|^2_\Hs(B_{u_1^0},..,B_{{u_{m_N}^N}})\Big]\\
=\sup\limits_{\theta\in\A0}E_{P_\theta}\Big[\big\|\int\limits_0^T \Phi^{(n)}(t)\,dB_t\big\|^2_\Hs\Big]\\
=\sup\limits_{\theta\in\A0}E_{P_\theta}\Big[\big\| \sum\limits_{k=0}^{N-1}\fee_k^{(n)}(B_{u_1^k},..,B_{{u_{m_k}^k}})\cdot( B_{t_{k+1}}-B_{t_k})\big\| ^2_\Hs\Big]\\
=\sup\limits_{\theta\in\A0}E_P\Big[\big\| \sum\limits_{k=0}^{N-1}\fee_k^{(n)}(B^{0,\theta}_{u_1^k},..,B^{0,\theta}_{{u_{m_k}^k}})\cdot( B^{0,\theta}_{t_{k+1}}-B^{0,\theta}_{t_k})\big\| ^2_\Hs\Big]$

\hfill$
=\Big\lfilet\Phi_*^{(n)}(t):=\sum\limits_{k=0}^{N-1}\fee_k^{(n)}(B^{0,\theta}_{u_1^k},..,B^{0,\theta}_{{u_{m_k}^k}})\cdot( t_{k+1}-{t_k})\Big\lfilet\\
=\sup\limits_{\theta\in\A0}E_P\Big[\big\|\int\limits_0^T \Phi_*^{(n)}(t)\,dB^{0,\theta}_t\big\|^2_\Hs\Big]
=\sup\limits_{\theta\in\A0}E_P\Big[\big\|\int\limits_0^T \Phi_*^{(n)}(t)\theta_t\,dW_t\big\|^2_\Hs\Big]\\
\overset{\text{classical Itô's isometry}}{=}
\sup\limits_{\theta\in\A0}E_P\Big[\int\limits_0^T\big\| \Phi_*^{(n)}(t)\cdot\theta_t\big\|^2_{L_2(\Us,\Hs)} dt\Big]\\
=\sup\limits_{\theta\in\A0}E_{P_\theta}\Big[\int\limits_0^T\big\| \Phi^{(n)}(t)\cdot\theta_t\big\|^2_{L_2(\Us,\Hs)} dt\Big]
\leq \sup\limits_{\theta\in\A0}E_{P_\theta}\Big[\int\limits_0^T
\sup\limits_{\theta\in\A0}\big\| \Phi^{(n)}(t)\cdot \theta\big\|^2_{L_2(\Us,\Hs)} dt\Big]\\
=\sup\limits_{\theta\in\A0}E_{P_\theta}\Big[\int\limits_0^T
\sup\limits_{\gamma\in\Theta}\big\| \Phi^{(n)}(t)\cdot \gamma\big\|^2_{L_2(\Us,\Hs)} dt\Big]\\
=\sup\limits_{\theta\in\A0}E_{P_\theta}\Big[\int\limits_0^T
\sup\limits_{Q\in\Sigma}\big\| \Phi^{(n)}(t)\cdot Q^{1/2}\big\|^2_{L_2(\Us,\Hs)} dt\Big]
=\bar{\E}\Big[\int\limits_0^T \|\Phi^{(n)}(t)\|^2_{L_2^\Sigma}\,dt\Big]$

\hfill$
=\E\Big[\int\limits_0^T \|\Phi^{(n)}(t)\|^2_{L_2^\Sigma}\,dt\Big]$.

For the finishing of proof we need to pass to the limit:

\textbf{(a):}$
\Bigg|\Bigg(\E\Big[\big\|\int\limits_0^T \Phi(t)\,dB_t\big\|^2_\Hs\Big]\Bigg)^{\frac{1}{2}}-
\Bigg(\E\Big[\big\|\int\limits_0^T \Phi^{(n)}(t)\,dB_t\big\|^2_\Hs\Big]\Bigg)^{\frac{1}{2}}\Bigg|$

\hfill$
\overset{\text{Minkowski ineq.}}{\leq}
\Bigg(\E\Big[\big\|\int\limits_0^T \Phi(t)\,dB_t-\int\limits_0^T \Phi^{(n)}(t)\,dB_t\big\|^2_\Hs\Big]\Bigg)^{\frac{1}{2}}\\
=\Bigg(\E\Big[\big\|\int\limits_0^T \big(\Phi(t)-\Phi^{(n)}(t)\big)\,dB_t\big\|^2_\Hs\Big]\Bigg)^{\frac{1}{2}}
\overset{\eqref{eq_weak_iii}}{\leq}\Bigg(\int\limits_0^T\E\Big[\big\| \Phi(t)-\Phi^{(n)}(t)\big\|^2_{L_2^\Sigma}\Big]\,dt\Bigg)^{\frac{1}{2}}\\
=\Bigg(\sum\limits_{k=0}^{N-1}\Big[\underbrace{\big\| \Phi_k-\Phi^{(n)}_k\big\|^2_\Sigma}_{\searrow0}\cdot (t_{k+1}-t_{k})\Big]\Bigg)^{\frac{1}{2}}
\xrightarrow[n\to\infty]{}0$.

\textbf{(b):}
$
\Bigg|\Bigg(\E\Big[\int\limits_0^T \|\Phi(t)\|^2_{L_2^\Sigma}\,dt\Big]\Bigg)^{\frac{1}{2}}-
\Bigg(\E\Big[\int\limits_0^T \|\Phi^{(n)}(t)\|^2_{L_2^\Sigma}\,dt\Big]\Bigg)^{\frac{1}{2}}\Bigg|\\
\overset{\text{Minkowski ineq.}}{\leq}
\Bigg(\E\Bigg[\Big|\Big(\int\limits_0^T \|\Phi(t)\|^2_{L_2^\Sigma}\,dt\Big)^{\frac{1}{2}}-
\Big(\int\limits_0^T \|\Phi^{(n)}(t)\|^2_{L_2^\Sigma}\,dt\Big)^{\frac{1}{2}} \Big|^2\Bigg]\Bigg)^{\frac{1}{2}}\\
\overset{\text{Minkowski ineq.}}{\leq}
\Bigg(\E\Big[\int\limits_0^T \|\Phi(t)-\Phi^{(n)}(t)\|^2_{L_2^\Sigma}\,dt \Big]\Bigg)^{\frac{1}{2}}
\leq\Bigg(\int\limits_0^T \E\Big[\|\Phi(t)-\Phi^{(n)}(t)\|^2_{L_2^\Sigma} \Big]\,dt\Bigg)^{\frac{1}{2}}\\
=\Bigg(\int\limits_0^T \|\Phi(t)-\Phi^{(n)}(t)\|^2_\Sigma\,dt\Bigg)^{\frac{1}{2}}
=\Bigg(\sum\limits_{k=0}^{N-1}\Big[\underbrace{\big\| \Phi_k-\Phi^{(n)}_k\big\|^2_\Sigma}_{\searrow0}\cdot (t_{k+1}-t_{k})\Big]\Bigg)^{\frac{1}{2}}
\xrightarrow[n\to\infty]{}0
$.

\end{proof}

The following result (so called the Burkholder-Davis-Gundy inequality) is a generalization of \textbf{Th.\ref{th_ito_is_inq}}.

\begin{teo}[BDG inequality]\label{th_BDG_inq}
$\absz$\\
Let $\Phi\in{^\Hs\!M^{p,0}_G\big( 0,T \big)}$ then
\begin{equation}\label{eq_BDG_iq}
\E\Big[\|\int\limits_0^T \Phi(t)\,dB_t\|^p_\Hs\Big]
\leq C_p\cdot\E\Bigg[\Big(\int\limits_0^T \|\Phi(t)\|^2_{L_2^\Sigma} \, dt\Big)^\frac{p}{2}\Bigg],
\end{equation}

\hfill where $C_p>0$.
\end{teo}

\begin{proof}
 $\absz$\\
The proof of the theorem is based on the proof of \textbf{Th.\ref{th_ito_is_inq}} and BDG inequality in the classical case, described for instance in 
\cite[ Lm.7.2]{da_prato_zabcz3}).

When $p=2$ and $C_p=1$ we just have the Itô isometry inequality.
\\
As in the proof of \textbf{Th.\ref{th_ito_is_inq}} we hold the same notations:

$\Phi\in{^\Hs\!M^{p,0}_G\big( 0,T \big)}$ \ then \ $ \Phi(t)=\sum\limits_{k=0}^{N-1}\Phi_k(\omega) \ind_{[t_k,t_{k+1})}(t),$

\hfill$ 0=t_0<t_1\ldots<t_N=T$ 

and for every $k\ \ \ \Phi_k \in {^\Hs}L_G^p( \Omega_{t_k} )$ \ then there exists a sequence

 $\big\{\Phi_k^{(n)}\big\}\subset Lip( \Omega_{t_k} )$, \ such that \ $\big\|\Phi_k-\Phi_k^{(n)}\big\|_{\Sigma,\,p}\xrightarrow[n\to\infty]{}0$.

$\Phi^{(n)}(t):=\sum\limits_{k=0}^{N-1}\Phi^{(n)}_k(\omega) \ind_{[t_k,t_{k+1})}(t)$.

Denote $I_T:=\int\limits_0^T \Phi(t)\,dB_t$ \ and \ 
$I^{(n)}_T:=\int\limits_0^T \Phi^{(n)}(t)\,dB_t$.

So, we need to show that \ \ \ 
$\E\Big[\|I_T\|^p_\Hs\Big]\leq C_p\cdot\E\Big(\int\limits_0^T \|\Phi(t)\|^2_{L_2^\Sigma} \, dt\Big)^\frac{p}{2}$.\hfill$(\,\ast\,)$

Firstly we show that\ \ \  $\forall m\geq2\ \ \ \ \E\Big[\|I_T\|^m_\Hs\Big]<\infty$:

$\E\Big[\|I_T\|^m_\Hs\Big]
=\E\Big[\big\|\int\limits_0^T \Phi(t)\,dB_t\big\|^m_\Hs\Big]
=\E\Big[\|\sum\limits_{k=0}^{N-1} \Phi_k\,\Delta B_{t_k}\|^m_\Hs\Big]\\
\leq C_m^1\cdot\sum\limits_{k=0}^{N-1} \E\Big[\|\Phi_k\,\Delta B_{t_k}\|^m_\Hs\Big]
=C_m^1\cdot\sum\limits_{k=0}^{N-1} \E\Big[\E\|\Phi_k\,\Delta B_{t_k}\|^m_\Hs\big|\,\Omega_{t_k}\Big]\\
\overset{\substack{\textbf{Prop.\ref{prop_moments}}\\ \textbf{Prop.\ref{prop_ar_op_gnd},\,2)} }}{\leq}
C_m^1\cdot\sum\limits_{k=0}^{N-1} \E\Big[(t_{k+1}-t_{k})^{\frac{m}{2}}\cdot 
C_m^2\cdot\sup\limits_{Q\in\Sigma}\Big(\Tr\big[\Phi_k Q \Phi_k^*\big]\Big)^{\frac{m}{2}}\Big]$

\hfill$
\leq C_m^0\cdot\sum\limits_{k=0}^{N-1} \E\Bigg[\Big((t_{k+1}-t_{k})\cdot\|\Phi(t)\|^2_{L_2^\Sigma}\Big)^{\frac{m}{2}}\Bigg]
<\infty
$.

It is clear also that \  $\E\Big[\|I^{(n)}_T\|^m_\Hs\Big]<\infty$.
\\
Then according to the proof of \textbf{Th.\ref{th_ito_is_inq}} we can get that

$
\E\Big[\big\|\int\limits_0^T \Phi^{(n)}(t)\,dB_t\big\|^p_\Hs\Big]
=\E\Big[\|\fee\|^p_\Hs(B_{u_1^0},..,B_{{u_{m_N}^N}})\Big]\\
=\sup\limits_{\theta\in\A0}E_{P_\theta}\Big[\|\fee\|^p_\Hs(B_{u_1^0},..,B_{{u_{m_N}^N}})\Big]
=\sup\limits_{\theta\in\A0}E_{P_\theta}\Big[\big\|\int\limits_0^T \Phi^{(n)}(t)\,dB_t\big\|^p_\Hs\Big]\\
=\sup\limits_{\theta\in\A0}E_{P_\theta}\Big[\big\| \sum\limits_{k=0}^{N-1}\fee_k^{(n)}(B_{u_1^k},..,B_{{u_{m_k}^k}})\cdot( B_{t_{k+1}}-B_{t_k})\big\| ^p_\Hs\Big]\\
=\sup\limits_{\theta\in\A0}E_P\Big[\big\| \sum\limits_{k=0}^{N-1}\fee_k^{(n)}(B^{0,\theta}_{u_1^k},..,B^{0,\theta}_{{u_{m_k}^k}})\cdot( B^{0,\theta}_{t_{k+1}}-B^{0,\theta}_{t_k})\big\| ^p_\Hs\Big]$

\hfill$
=\Big\lfilet\Phi_*^{(n)}(t):=\sum\limits_{k=0}^{N-1}\fee_k^{(n)}(B^{0,\theta}_{u_1^k},..,B^{0,\theta}_{{u_{m_k}^k}})\cdot( t_{k+1}-{t_k})\Big\lfilet\\
=\sup\limits_{\theta\in\A0}E_P\Big[\big\|\int\limits_0^T \Phi_*^{(n)}(t)\,dB^{0,\theta}_t\big\|^p_\Hs\Big] 
=\sup\limits_{\theta\in\A0}E_P\Big[\big\|\int\limits_0^T \Phi_*^{(n)}(t)\theta_t\,dW_t\big\|^p_\Hs\Big]\\
\overset{\text{classical BDG inequality}}{=}
C_p\cdot\sup\limits_{\theta\in\A0}E_P\Big[\Big(\int\limits_0^T\big\| \Phi_*^{(n)}(t)\cdot\theta_t\big\|^2_{L_2(\Us,\Hs)} dt\Big)^{\frac{p}{2}}\Big]\\
=C_p\cdot\sup\limits_{\theta\in\A0}E_{P_\theta}\Big[\Big(\int\limits_0^T\big\| \Phi^{(n)}(t)\cdot\theta_t\big\|^2_{L_2(\Us,\Hs)} dt\Big)^{\frac{p}{2}}\Big]\\
\leq C_p\cdot\sup\limits_{\theta\in\A0}E_{P_\theta}\Big[\Big(\int\limits_0^T\sup\limits_{\theta\in\A0}\big\| \Phi^{(n)}(t)\cdot\theta\big\|^2_{L_2(\Us,\Hs)} dt\Big)^{\frac{p}{2}}\Big]\\
\leq C_p\cdot\sup\limits_{\theta\in\A0}E_{P_\theta}\Big[\Big(\int\limits_0^T\sup\limits_{\gamma\in\Theta}\big\| \Phi^{(n)}(t)\cdot\gamma\big\|^2_{L_2(\Us,\Hs)} dt\Big)^{\frac{p}{2}}\Big]\\
=C_p\cdot\sup\limits_{\theta\in\A0}E_{P_\theta}\Big[\Big(\int\limits_0^T\sup\limits_{Q\in\Sigma}\big\| \Phi^{(n)}(t)\cdot Q^{1/2}\big\|^2_{L_2(\Us,\Hs)} dt\Big)^{\frac{p}{2}}\Big]\\
=C_p\cdot\bar{\E}\Big[\Big(\int\limits_0^T \|\Phi^{(n)}(t)\|^2_{L_2^\Sigma}\,dt\Big)^{\frac{p}{2}}\Big]
=C_p\cdot\E\Big[\Big(\int\limits_0^T \|\Phi^{(n)}(t)\|^2_{L_2^\Sigma}\,dt\Big)^{\frac{p}{2}}\Big]$.
\\
And now we pass to the limit.
\\
\textbf{(a):}$
\Bigg|\E\Big[\big\|I_T\big\|^p_\Hs\Big]-\E\Big[\big\|I^{(n)}_T\big\|^p_\Hs\Big]\Bigg|
\leq \E\Bigg[\Big|\big\|I_T\big\|^p_\Hs-\big\|I^{(n)}_T\big\|^p_\Hs\Big|\Bigg]\\
\leq \E\Bigg[\big\|I_T-I^{(n)}_T\big\|_\Hs\cdot\Big(\big\|I_T\big\|^{p-1}_\Hs+\big\|I_T\big\|^{p-2}_\Hs\cdot\big\|I^{(n)}_T\big\|_\Hs+\ldots+\big\|I^{(n)}_T\big\|^{p-1}_\Hs\Big)\Bigg]\\
\overset{\textbf{Prop.\ref{prop_cb_ineq}}}{\leq}
\Bigg(\E\Big[\big\|I_T-I^{(n)}_T\big\|^2_\Hs\Big]\Bigg)^{\frac{1}{2}}\times$

\hfill$\times
\underbrace{\Bigg(\E\Bigg[\Big(\big\|I_T\big\|^{p-1}_\Hs+\big\|I_T\big\|^{p-2}_\Hs\cdot\big\|I^{(n)}_T\big\|_\Hs+\ldots+\big\|I^{(n)}_T\big\|^{p-1}_\Hs\Big)^2\Bigg]\Bigg)^{\frac{1}{2}}}_{<\infty}\\
\overset{\textbf{Th.\ref{th_ito_is_inq}}}{\leq}C\cdot\vvvert\Phi-\Phi^{(n)}\vvvert_T
=C\cdot\Bigg(\E\Big[\int\limits_0^T \|\Phi(t)-\Phi^{(n)}(t)\|^2_{L_2^\Sigma}\,dt \Big]\Bigg)^{\frac{1}{2}}\\
\leq C\cdot\Bigg(\int\limits_0^T \E\Big[\|\Phi(t)-\Phi^{(n)}(t)\|^2_{L_2^\Sigma} \Big]\,dt\Bigg)^{\frac{1}{2}}
=C\cdot\Bigg(\int\limits_0^T \|\Phi(t)-\Phi^{(n)}(t)\|^2_\Sigma\,dt\Bigg)^{\frac{1}{2}}$

\hfill$
=C\cdot\Bigg(\sum\limits_{k=0}^{N-1}\Big[\underbrace{\big\| \Phi_k-\Phi^{(n)}_k\big\|^2_\Sigma}_{\searrow0}\cdot (t_{k+1}-t_{k})\Big]\Bigg)^{\frac{1}{2}}
\xrightarrow[n\to\infty]{}0$.
\\
\textbf{(b):}
$
\Bigg|\Bigg(\E\Big[\Big(\int\limits_0^T \|\Phi(t)\|^2_{L_2^\Sigma}\,dt\Big)^{\frac{p}{2}}\Big]\Bigg)^{\frac{1}{p}}-
\Bigg(\E\Big[\Big(\int\limits_0^T \|\Phi^{(n)}(t)\|^2_{L_2^\Sigma}\,dt\Big)^{\frac{p}{2}}\Big]\Bigg)^{\frac{1}{p}}\Bigg|\\
\overset{\text{Minkowski ineq.}}{\leq}
\Bigg|\Bigg(\E\Bigg[\Bigg(\Big(\int\limits_0^T \|\Phi(t)\|^2_{L_2^\Sigma}\,dt\Big)^{\frac{1}{2}}-
\Big(\int\limits_0^T \|\Phi^{(n)}(t)\|^2_{L_2^\Sigma}\,dt\Big)^{\frac{1}{2}}\Bigg)^p\Bigg]\Bigg)^{\frac{1}{p}}\Bigg|\\
\overset{\text{Minkowski ineq.}}{\leq}
\Bigg(\E\Big[\Big(\int\limits_0^T \|\Phi(t)-\Phi^{(n)}(t)\|^2_{L_2^\Sigma}\,dt \Big)^{\frac{p}{2}}\Big]\Bigg)^{\frac{1}{p}}\\
\leq\Bigg(\E\Big[\int\limits_0^T \|\Phi(t)-\Phi^{(n)}(t)\|^p_{L_2^\Sigma}\,dt \Big]\Bigg)^{\frac{1}{p}}
\leq\Bigg(\int\limits_0^T \E\Big[\|\Phi(t)-\Phi^{(n)}(t)\|^p_{L_2^\Sigma} \Big]dt\Bigg)^{\frac{1}{p}}\\
=\Bigg(\int\limits_0^T \|\Phi(t)-\Phi^{(n)}(t)\|^p_{\Sigma,\,p}\,dt\Bigg)^{\frac{1}{p}}$

\hfill$
=\Bigg(\sum\limits_{k=0}^{N-1}\Big[\underbrace{\big\| \Phi_k-\Phi^{(n)}_k\big\|^p_{\Sigma,\,p}}_{\searrow0}
\cdot (t_{k+1}-t_{k})\Big]\Bigg)^{\frac{1}{p}}
\xrightarrow[n\to\infty]{}0
$.

\end{proof}

Define  ${^\Hs\!M^p_G\big( 0,T \big)}$  as a completion of ${^\Hs\!M^{p,0}_G\big( 0,T \big)}$
under the norm 

\hfill$ \vvvert \Phi\vvvert^p_{T,\,p}:=\E\Bigg[\Big(\int\limits_0^T \|\Phi(t)\|^2_{L_2^\Sigma} \, dt\Big)^\frac{p}{2}\Bigg]
$.

Actually, such a norm is finite on ${^\Hs\!M^{p,0}_G\big( 0,T \big)}$, because

\parbox{140mm}{$ \vvvert \Phi\vvvert^p_{T,\,p}=\E\Bigg[\Big(\sum\limits_{k=0}^{N-1}  (t_k-t_{k+1})\cdot\|\Phi(t)\|^2_{L_2^\Sigma}\Big)^\frac{p}{2}\Bigg]$

\hfill$
\leq C_{p,N}\cdot\sum\limits_{k=0}^{N-1}\Big((t_k-t_{k+1})^\frac{p}{2}\cdot\underbrace{\E\Big[ \|\Phi(t)\|^p_{L_2^\Sigma}\Big]}_{<\infty}\Big)
<\infty
$.}

Also we denote \ \ $ \vvvert \Phi\vvvert^2_{T,\,2}:=\vvvert \Phi\vvvert^2_T=\E\Big[\int\limits_0^T \|\Phi(t)\|^2_{L_2^\Sigma}\,dt\Big]$.

\begin{teo}\label{th_ito_BDG_general}
 $I_T=\int\limits_0^T \Phi(t)\,dB_t$ can be extended on ${^\Hs\!M^p_G\big( 0,T \big)}$.

And for $\Phi\in{^\Hs\!M^p_G\big( 0,T \big)}$ the BDG inequality \eqref{eq_BDG_iq} holds.

In particular, if $p=2$ the Itô isometry inequality \eqref{eq_ito_is_inq} holds.
\end{teo}

\begin{proof}
 $\absz$\\
We have if   $\Phi\in{^\Hs\!M^p_G\big( 0,T \big)}$ \ then there exists a sequence 

$
\big\{\Phi^{(n)},\ n\geq1\big\}\subset{^\Hs\!M^{p,0}_G\big( 0,T \big)}$, \ such that \  $\vvvert\Phi-\Phi^{(n)}\vvvert_{T,\,p}\xrightarrow[n\to\infty]{}0$.

Let us define a norm \  $ \|I_T\|^p_{\Omega_T,\,p}:=\E\big[\|I_T\|_\Hs^p\big]  $.

And we have \  
$\|I_T(\Phi^{(n)})-I_T(\Phi^{(m)})\|_{\Omega_T,\,p}=
\|I_T(\Phi^{(n)}-\Phi^{(m)})\|_{\Omega_T,\,p}$

\hfill$
\overset{\textbf{Th.\ref{th_BDG_inq}}}{\leq}
C_p\cdot\vvvert \Phi^{(n)}-\Phi^{(m)}\vvvert_{T,\,p}$.
\\
So that, we can extend $I_T$ on ${^\Hs\!M^p_G\big( 0,T \big)}$ as a continuous mapping.
\\
Define $I_T(\Phi):=\lim\limits_{n\to\infty}I_T(\Phi^{(n)})$. 
\\
Letting $m\to\infty\ $ \ yields \ $0\leq\|I_T(\Phi^{(n)})-I_T(\Phi)\|_{\Omega_T,\,p}$

\hfill$
\leq C_p\cdot\vvvert \Phi^{(n)}-\Phi\vvvert_{T,\,p}
\xrightarrow[n\to\infty]{}0$.
\\
Therefore \ $\|I_T(\Phi^{(n)})\|_{\Omega_T,\,p}\xrightarrow[n\to\infty]{}\|I_T(\Phi)\|_{\Omega_T,\,p}$ \ 
and \ $\vvvert \Phi^{(n)}\vvvert_{T,\,p}
\xrightarrow[n\to\infty]{}\vvvert \Phi\vvvert_{T,\,p}$.
\\
So we can conclude that $\|I_T(\Phi)\|_{\Omega_T,\,p}
\leq C_p\cdot\vvvert \Phi\vvvert_{T,\,p}$.

\end{proof}

\subsection[Characterization of the space of integrand processes ${^\Hs\!M^{2}_G}$]{Characterization of the space of integrand processes ${^\Hs\!M^{2}_G\big( 0,T \big)}$}

\begin{prop}\label{prop_m2gh}
 $\Phi\in{^\Hs\!M^{2}_G\big( 0,T \big)}\ \ \ \Leftrightarrow\ \ \ $
\parbox[t]{100mm}{
1)$\vvvert \Phi\vvvert_T<\infty$;\\
2)$\Phi(t)\in{^\Hs\!L^{2}_G\big( \Omega_t \big)}$ for almost all $ t$.
}
\end{prop}

\begin{proof}
$\absz$\\
$(\Rightarrow)$\\
  If $\Phi\in{^\Hs\!M^{2}_G\big( 0,T \big)}$ \ then \ $\vvvert \Phi\vvvert_T<\infty$ \ and there exists a sequence 

 $\big\{\Phi_n,\ n\geq1\big\}\subset{^\Hs\!M^{2,0}_G\big( 0,T \big)}$, \ such that \ $\vvvert\Phi-\Phi_n\vvvert_T\xrightarrow[n\to\infty]{}0$.

It follows that for almost all $t$ \ we have \ $\|\Phi-\Phi_n\|_\Sigma\xrightarrow[n\to\infty]{}0$.

(Recall that the norm $\|\pdot\|_\Sigma$ is introduced in \ref{subs_def_stoch_int}).
\\
For such a fixed $t=t'$ \ implies that \ $\Phi_n(t')=Const\in Lip\big( \Omega_t \big)$.

Since $\Big({^\Hs\!L^{2}_G\big( \Omega_t \big)},\|\pdot\|_\Sigma\Big)$ is a Banach space, then for almost all $t$ we have that \ $\Phi(t)\in{^\Hs\!L^{2}_G\big( \Omega_t \big)}$.
\\
$(\Leftarrow)$\\
\textbf{1)} Let for almost all $t\ \ \Phi(t)\in Lip\big( \Omega_t \big)$ be continuous.

Take partition of $[0,T]:\ \lambda_n=\{0=t_0^n<t_1^n<\ldots<t_N^n=T\}\,,$

\hfill$d(\lambda_n)\xrightarrow[n\to\infty]{}0\,,\ N=N(n)\xrightarrow[n\to\infty]{}0$.

Define $\Phi_n:=\sum\limits_{k=0}^{N-1}\Phi(t_k^n)(\omega) \ind_{[t^n_k,t^n_{k+1})}(t)\in{^\Hs\!M^{2,0}_G\big( 0,T \big)}$.

We have \ $\Phi_n(t)\xrightarrow[n\to\infty]{}\Phi(t)$ \ for all $t$. 

Let us calculate \ 

$\vvvert \Phi-\Phi_n\vvvert_T^2
=\E\Big[\big\|\int\limits_0^T \big(\Phi(t)-\Phi_n(t)\big)\,dB_t\big\|^2_\Hs\Big]
\leq  \E\Big[\int\limits_0^T \big\|\big(\Phi(t)-\Phi_n(t)\big)\big\|^2_{L_2^\Sigma}\,dt\Big]$\,,

\hfill using Itô's inequality according to \textbf{Th.\ref{th_ito_BDG_general}}.
\\
Then there exists a point $t_\theta\in[0,T]$\ such that

$\E\Big[\int\limits_0^T \big\|\big(\Phi(t)-\Phi_n(t)\big)\big\|^2_{L_2^\Sigma}\,dt\Big]
=T\cdot\E\Big[ \big\|\big(\Phi(t_\theta,\omega)-\Phi_n(t_\theta,\omega)\big)\big\|^2_{L_2^\Sigma}\Big]$.

But the last term tends to $0$ according to \textbf{Th.\ref{th_mon_conv_g_exp}}.

So, we have that \ $\vvvert \Phi-\Phi_n\vvvert_T\xrightarrow[n\to\infty]{}0$.
\\
\textbf{2)} Let for almost all $t\ \ \Phi(t)\in Lip\big( \Omega_t \big)$ (it is not necessary continuous).

Define \ $\Phi_\varepsilon(t):=\dfrac{1}{\varepsilon}\int\limits_{-\infty}^{+\infty}\rho(\frac{t-s-\varepsilon}{\varepsilon})\Phi(s)ds=
\dfrac{1}{\varepsilon}\int\limits_{t-2\varepsilon}^{t}\rho(\frac{t-s-\varepsilon}{\varepsilon})\Phi(s)ds$

\hfill$=
\int\limits_{-1}^{1}\rho(v)\Phi(t-\varepsilon v-\varepsilon)dv$.
\\
Hence for all $t\ \  \Phi_\varepsilon(t+)=\Phi_\varepsilon(t-)=\Phi_\varepsilon(t)$, \ so that \ $\Phi_\varepsilon(\pdot)$ is continuous.

Consider $A_\varepsilon:=\vvvert \Phi-\Phi_\varepsilon\vvvert_T^2=\int\limits_0^T \|\Phi(t)-\Phi_\varepsilon(t)\|^2_\Sigma \, dt\\
\hspace*{28,5mm}=
\int\limits_0^T \|\int\limits_{-1}^{1}\rho(v)\big(\Phi(t)-\Phi(t-\varepsilon v-\varepsilon)\big)dv\|^2_\Sigma \, dt\\
\hspace*{28,5mm}
 \leq
\int\limits_0^T \Big(\int\limits_{-1}^{1}\rho(v)\|\Phi(t)-\Phi(t-\varepsilon v-\varepsilon)\|_\Sigma\, dv\Big)^2 dt\\
\\
\hspace*{28,5mm} \leq
\int\limits_0^T \Big(\int\limits_{-1}^{1}\rho(v)dv\cdot\int\limits_{-1}^{1}\rho(v)\|\Phi(t)-\Phi(t-\varepsilon v-\varepsilon)\|^2_\Sigma\, dv\Big) dt$.
\\
But for every separable Banach space $B$ we have such a dense inclusion: 
$C\big([0,T],\,B\big)\subset L_2\big([0,T],\,B\big) $,

because every $f\in L_2\big([0,T],\,B\big)$ can be approximated by $\sum\limits_{i=1}^m f_i(t) b_i$, 

\hfill where $(b_i)_{i\geq1}\subset B$ - densely, $f_i\in L_2\big([0,T],\,B\big)$;

and such $f_i\in L_2\big([0,T],\,B\big)$ can be approximated by $g_i\in C\big([0,T],\,B\big)$, because it is well-known that 
$C\big([0,T],\,B\big)\subset L_2\big([0,T],\,B\big) $ densely.

So, $f$ can be approximated by $\sum\limits_{i=1}^m g_i(t) b_i\in C\big([0,T],\,B\big)$.

Using this fact we have that:

For $\Phi:[0,T]\to{^\Hs\!L^{2}_G\big( \Omega_T \big)}=:B$, where $\big(B,\|\pdot\|_B\big)$ is a Banach separable space
and \ $\int\limits_0^T \|\Phi(t)\|^2_B dt<\infty$, \ there exists \ 
$\Psi_\delta\in C\big([0,T],\,B\big)$, \ such that \ $\int\limits_0^T \|\Phi(t)-\Psi_\delta(t)\|^2_B dt<\delta$.
\\
Therefore \ 
$\big|A_\varepsilon\big|\leq
\int\limits_0^T \int\limits_{-1}^{1}\rho(v)\|\Phi(t)-\Phi(t-\varepsilon v-\varepsilon)\|^2_\Sigma\, dv dt
\\
\leq
3\int\limits_0^T \int\limits_{-1}^{1}\rho(v)\|\Phi(t)-\Psi_\delta(t)\|^2_\Sigma\, dv dt
+
3\int\limits_0^T \int\limits_{-1}^{1}\rho(v)\|\Phi(t-\varepsilon v-\varepsilon)$

\hfill$-\Psi_\delta(t-\varepsilon v-\varepsilon)\|^2_\Sigma\, dv dt+
3\int\limits_0^T \int\limits_{-1}^{1}\rho(v)\|\Psi_\delta(t)-\Psi_\delta(t-\varepsilon v-\varepsilon)\|^2_\Sigma\, dv dt$

\hfill$ \leq
6\delta+3\int\limits_0^T \int\limits_{-1}^{1}\rho(v)\|\Psi_\delta(t)-\Psi_\delta(t-\varepsilon v-\varepsilon)\|^2_\Sigma\, dv dt$.

Hence \ $\overline{\lim\limits_{\varepsilon\to0}}\big|A_\varepsilon\big|\leq6\delta$.

So that \ $\vvvert \Phi-\Phi_\varepsilon\vvvert_T\xrightarrow[\varepsilon\to0]{}0$.
\\
\textbf{3)} Let for almost all $t \ \ \Phi(t)\in{^\Hs\!L^{2}_G\big( \Omega_t \big)}$.

By the definition there exists a sequence \ $\big\{\Phi_m,\ m\geq1\big\}\subset Lip\big( \Omega_t \big)$, \ such that \ $ \|\Phi-\Phi_m\|_\Sigma\xrightarrow[n\to\infty]{}0$ 
for almost all $t$.

As in the \textbf{1)} part we can get that 

\hfill $\vvvert \Phi-\Phi_m\vvvert_T^2
\leq T\cdot\E\Big[ \big\|\big(\Phi(t_\theta,\omega)-\Phi_n(t_\theta,\omega)\big)\big\|^2_{L_2^\Sigma}\Big]\xrightarrow[n\to\infty]{}0$.

Hence \ $\vvvert \Phi-\Phi_m\vvvert_T\leq\vvvert \Phi-(\Phi_m)^n\vvvert_T+\vvvert \Phi_m-(\Phi_m)^n\vvvert_T$,

\hfill
where $(\Phi_m)^n\in{^\Hs\!M^{2,0}_G\big( 0,T \big)}$.

From \textbf{2)} it follows that \ $\vvvert \Phi-\Phi_m\vvvert_T\xrightarrow[m, n\to\infty]{}0$
.

So that \ $\Phi_m\in{^\Hs\!M^{2,0}_G\big( 0,T \big)}$.

\end{proof}

\begin{rem}\label{rem_m2gh}
 If $\Phi(t)$ is nonrandom then condition  \textbf{2)} of \textbf{Prop.\ref{prop_m2gh}} can be omitted:
$$\vvvert \Phi\vvvert_T<\infty\ \ \Leftrightarrow\ \ \Phi\in{^\Hs\!M^{2}_G\big( 0,T \big)}.$$
\end{rem}

\begin{proof}
$\absz$\\
$\vvvert \Phi\vvvert_T=
\int\limits_0^T \|\Phi(t)\|^2_\Sigma \, dt =
\int\limits_0^T \|\Phi(t)\|^2_{L^\Sigma_2} \, dt<\infty$ \ then \ $\|\Phi(t)\|^2_{L^\Sigma_2}<\infty$ 

\hfill for almost all $t$.

From here we can conclude that \ 
$\Phi(t)\in L^\Sigma_2$, \ hence that \ 
$\Phi(t)\in Lip\big( \Omega_t \big)$ (since $\Phi(t)$ is nonrandom) and finally that \
$\Phi(t)\in{^\Hs\!L^{2}_G\big( \Omega_t \big)}$ 

\hfill for almost all $t$.

And by \textbf{Prop.\ref{prop_m2gh}} we get that \ 
$\Phi(t)\in{^\Hs\!M^{2}_G\big( 0,T \big)}$.
\\
\end{proof}

\subsection{Fubini theorem}

Let $(\mathcal{X},\mathcal{E},\mu)$ is a measurable space, $\mu(\mathcal{X})<\infty$.

Consider 
${^\Hs\!M^{2,0}_G\big( 0,T ;\mathcal{X}\big)}
:=\Bigg\{\Phi(t,x)=\!\!\sum\limits_{k=0}^{N-1}\sum\limits_{j=0}^{M-1}\Phi_{kj}(\omega) \ind_{[t_k,t_{k+1})}(t)\ind_{A_j}(x)\ \Bigg|\ $

\hfill$ 
\Phi_k (\omega)\in {^\Hs}L_G^2( \Omega_{t_k} ),\ 0=t_0<t_1\ldots<t_N=T,\ A_j\in\mathcal{E}\Bigg\}$.

Let ${^\Hs\!M^{2}_G\big( 0,T ;\mathcal{X}\big)}$ be a completion of ${^\Hs\!M^{2,0}_G\big( 0,T ;\mathcal{X}\big)}$ under the norm 

$ \vvvert \Phi\vvvert_{T,\mathcal{X}}:=\int\limits_\mathcal{X}\vvvert\Phi(\pdot,x)\vvvert_T\, \mu(dx)$.

\begin{teo}\label{th_fubini}
 If $\Phi(t,x)\in{^\Hs\!M^{2}_G\big( 0,T ;X\big)}$, then:
$$\int\limits_\mathcal{X}\int\limits_0^T \Phi(t,x)\, dB_t\,  \mu(dx)=\int\limits_0^T\int\limits_\mathcal{X} \Phi(t,x)\, \mu(dx)\, dB_t  \ \ \ \text{q.s.}$$

\end{teo}

\begin{proof}
 $\absz$\\
Let  $\big\{\Phi_n,\ n\geq1\big\}\subset{^\Hs\!M^{2,0}_G\big( 0,T;\mathcal{X} \big)}$.

 Then in the same way as in \textbf{Prop.\ref{prop_m2gh}} using  \textbf{Th.\ref{th_ito_BDG_general}} and  \textbf{Th.\ref{th_mon_conv_g_exp}} we can conclude that \ 
 $\vvvert\Phi(\pdot,x)-\Phi_n(\pdot,x)\vvvert_T\xrightarrow[n\to\infty]{}0$.

And due to the dominated convergence theorem have that

\hfill $\int\limits_\mathcal{X}\vvvert\Phi(\pdot,x)-\Phi_n(\pdot,x)\vvvert_T\, \mu(dx)\xrightarrow[n\to\infty]{}0$.
\\
For every $ n$ \ set \ $\Phi_n(t,x):=\!\!\sum\limits_{k=0}^{N_n-1}\sum\limits_{j=0}^{M_n-1}\Phi^n_{kj}(\omega) \ind_{[t^n_k,t^n_{k+1})}(t)\ind_{A^n_j}(x)$.

And define the following random variables:

$\xi_n(x):=\int\limits_0^T \Phi_n(t,x)\, dB_t\,,\ \  \xi(x):=\int\limits_0^T \Phi(t,x)\, dB_t$.

$\eta_n(t):=\int\limits_\mathcal{X} \Phi_n(t,x)\,\mu(dx)\,,\ \  \eta(t):=\int\limits_\mathcal{X} \Phi(t,x)\,\mu(dx)$.

Then \ $\int\limits_\mathcal{X} \xi_n(x)\,\mu(dx)=\int\limits_0^T \eta_n(t)\, dB_t=
\!\!\sum\limits_{k=0}^{N_n-1}\sum\limits_{j=0}^{M_n-1}\Phi^n_{kj} ( B_{t^n_{k+1}}-B_{t^n_k})\mu(A^n_j)$.

For the later calculations we will use a Cauchy-Schwarz-Bunyakovsky inequality (\textbf{Prop.\ref{prop_cb_ineq}}) in a following form \  $\E[X]\leq\big(\E[X^2]\big)^\frac{1}{2}$:

\begin{enumerate}[(a)]
 \item $\E\Big[\big\|\int\limits_\mathcal{X} \xi(x)\,\mu(dx)-\int\limits_\mathcal{X} \xi_n(x)\,\mu(dx)\big\|_\Hs\Big]
\\
=
\E\Big[\big\| \int\limits_\mathcal{X}\int\limits_0^T\big(\Phi(t,x)-\Phi_n(t,x)\big)\,dB_t\,\mu(dx)\big\|_\Hs\Big]\\ \leq 
\int\limits_\mathcal{X}\E\Big[\big\| \int\limits_0^T\big(\Phi(t,x)-\Phi_n(t,x)\big)\,dB_t\big\|_\Hs\Big]\mu(dx)
\\
\leq
\int\limits_\mathcal{X}\Big(\E\Big[\big\| \int\limits_0^T\big(\Phi(t,x)-\Phi_n(t,x)\big)\,dB_t\big\|^2_\Hs\Big]\Big)^\frac{1}{2}\mu(dx)\\=
\int\limits_\mathcal{X}\vvvert\Phi(\pdot,x)-\Phi_n(\pdot,x)\vvvert_T\, \mu(dx)\xrightarrow[n\to\infty]{}0$.
  \item $\E\Big[\big\|\int\limits_0^T \eta(t)\, dB_t-\int\limits_\mathcal{X} \xi_n(x)\,\mu(dx)\big\|_\Hs\Big]
\\
=
\E\Big[\big\| \int\limits_0^T \eta(t)\, dB_t-\int\limits_0^T \eta_n(t)\, dB_t\big\|_\Hs\Big]\\ =
\E\Big[\big\| \int\limits_0^T\int\limits_\mathcal{X}\big(\Phi(t,x)-\Phi_n(t,x)\big)\,\mu(dx)\,dB_t\big\|_\Hs\Big]
\\
 \leq 
\vvvert\int\limits_\mathcal{X}\big(\Phi(\pdot,x)-\Phi_n(\pdot,x)\big)\vvvert_T\, \mu(dx)\\ \leq
\int\limits_\mathcal{X}\vvvert\Phi(\pdot,x)-\Phi_n(\pdot,x)\vvvert_T\, \mu(dx)\xrightarrow[n\to\infty]{}0$.
\end{enumerate}

Therefore \ $\E\big\|\int\limits_\mathcal{X} \xi(x)\,\mu(dx)-\int\limits_0^T \eta(t)\, dB_t\big\|_\Hs$.

And by \textbf{Lm.\ref{lm_X0_qs}} we get that \ $\int\limits_\mathcal{X} \xi(x)\,\mu(dx)=\int\limits_0^T \eta(t)\, dB_t$ \ quasi surely.

\end{proof}

\subsection[Distribution of the stochastic integral]{Distribution of the stochastic integral with nonrandom integrand}

Let us consider a stochastic integral \ $I(\Phi)=\int\limits_0^T \Phi(t)dB_t$, with respect to $G$-Brownian motion \ 
$B_t\sim N_G\big(0,t\cdot\Sigma\big)$.

Assume that  $\Phi$ is non-random, then  $\Phi:[0,T]\to L_2^\Sigma$.
\\
Note that also in such a case $\ \ \ {^\Hs\!L^{2}_G\big( \Omega\big)}\equiv Lip(\Omega)\equiv L_2^\Sigma$.

\textbf{1)} Firstly we consider a case with elementary processes:

If $\Phi\in{^\Hs\!M^{2,0}_G\big( 0,T \big)}$ \ then \ $\Phi(t)=\sum\limits_{k=0}^{N-1}\Phi_k\ind_{[t_k,t_{k+1})}(t)\,,\ \ 
\Phi_k\in L_2^\Sigma$, and the norm
$\vvvert\Phi\vvvert_T^2=\int\limits_0^T\|\Phi(t)\|^2_{L_2^\Sigma}dt<\infty$.

By the definition we have \ $ I(\Phi)=\sum\limits_{k=0}^{N-1}\Phi_k( B_{t_{k+1}}-B_{t_k})$.

We know that a random variable $\dfrac{B_{t_{k+1}}-B_{t_k}}{\sqrt{t_{k+1}-t_k}}\sim B_1\sim N_G\big(0,\Sigma\big)$.

Then using \textbf{Prop.\ref{prop_ar_op_gnd}} we get that \ $ I(\Phi)\sim N_G\big(0,\Sigma_I\big)$,

\hfill where$
\Sigma_I=\Big\{\sum\limits_{k=0}^{N-1}(t_{k+1}-t_k)\Phi_k Q \Phi_k^*\big|\, Q\in\Sigma\Big\}$.

Note that in this case \ $\sum\limits_{k=0}^{N-1}(t_{k+1}-t_k)\Phi_k Q \Phi_k^*=\int\limits_0^T\Phi(t) Q \Phi^*(t)dt$.
\\
\textbf{2)} In a general case we have that for
$\Phi\in{^\Hs\!M^2_G\big( 0,T \big)}$ \ there exists a sequence \ $
\big\{\Phi_n,\ n\geq1\big\}\subset{^\Hs\!M^{2,0}_G\big( 0,T \big)}$, \ such that  \ $\vvvert\Phi-\Phi_n\vvvert_T\xrightarrow[n\to\infty]{}0$.

And by \textbf{Th.\ref{th_ito_BDG_general}} we get that \ $\|I(\Phi)-I(\Phi_n)\|_{\Omega_T}\leq \vvvert \Phi-\Phi_n\vvvert_T$.

Since $\Phi\in L_2^\Sigma$ \ then for \ $p>0\ \ \ \Phi\in{^\Hs\!M^p_G\big( 0,T \big)}$ \ and by \textbf{Th.\ref{th_ito_BDG_general}} it may be concluded that \ $\E\big[\|I\|^p_\Hs\big]
\leq C_p\cdot\Big(\int\limits_0^T \|\Phi(t)\|^2_{L_2^\Sigma} \, dt\Big)^\frac{p}{2}<\infty$.

\begin{lm}\label{lm_converg_cov_set_elements}
 $\int\limits_0^T\Phi_n(t) Q \Phi_n^*(t)dt\xrightarrow[n\to\infty]{}\int\limits_0^T\Phi(t) Q \Phi^*(t)dt$ \ in the trace-class topology.
\end{lm}

\begin{proof}
 $\absz$\\
$ \vvvert\Phi-\Phi_n\vvvert_T^2
=\int\limits_0^T\|\Phi(t)-\Phi_n(t)\|^2_{L_2^\Sigma}dt\\
=\int\limits_0^T\sup\limits_{Q\in\Sigma}\Tr\Big[\big(\Phi(t)-\Phi_n(t)\big)Q\big(\Phi(t)-\Phi_n(t)\big)^*\Big]dt$

\hfill$
=\int\limits_0^T\sup\limits_{Q\in\Sigma}\big\|\big(\Phi(t)-\Phi_n(t)\big)Q^{1/2}\big\|^2_{L_2(\Hs)}dt
\xrightarrow[n\to\infty]{}0
$.

Denote \  $A_n:=\Phi_n(t)\cdot Q^{1/2}$ \ and \ $ A:=\Phi(t)\cdot Q^{1/2}$.
\\
We have that \ \ \ 

 $\Big\|\int\limits_0^T\Phi_n(t) Q \Phi_n^*(t)dt-\int\limits_0^T\Phi(t) Q \Phi^*(t)dt\Big\|_{L_1(\Hs)}
=\Big\|\int\limits_0^T \big(A_n A_n^*-AA^*\big)dt\Big\|_{L_1(\Hs)}\\
\leq\int\limits_0^T \|A_n A_n^*-AA^*\|_{L_1(\Hs)}dt
=\int\limits_0^T \|A_n A_n^*-A A_n^*+A A_n^*-AA^*\|_{L_1(\Hs)}dt\\
\leq\int\limits_0^T\Big( \|A^*_n\|_{L_2(\Hs)} \cdot \|A-A_n\|_{L_2(\Hs)}+\|A\|_{L_2(\Hs)} \cdot \|(A-A_n)^*\|_{L_2(\Hs)}\Big)dt\\
\leq\Big(\int\limits_0^T \|A_n\|^2_{L_2(\Hs)}dt\Big)^{1/2} \cdot \Big(\int\limits_0^T \|A-A_n\|^2_{L_2(\Hs)}dt\Big)^{1/2}$

\hfill$
+\Big(\int\limits_0^T \|A\|^2_{L_2(\Hs)}dt\Big)^{1/2} \cdot \Big(\int\limits_0^T \|A-A_n\|^2_{L_2(\Hs)}dt\Big)^{1/2}\\
\leq\Bigg(\Big(\int\limits_0^T \sup\limits_{Q\in\Sigma}\big\|\big(\Phi_n(t)\big)Q^{1/2}\big\|^2_{L_2(\Hs)}dt\Big)^{1/2}
+\Big(\int\limits_0^T \sup\limits_{Q\in\Sigma}\big\|\big(\Phi(t)\big)Q^{1/2}\big\|^2_{L_2(\Hs)}dt\Big)^{1/2}\Bigg)\times$

\hfill$
\times\int\limits_0^T\sup\limits_{Q\in\Sigma}\big\|\big(\Phi(t)-\Phi_n(t)\big)Q^{1/2}\big\|^2_{L_2(\Hs)}dt\\
=\Big(\vvvert\Phi_n\vvvert_T+\vvvert\Phi\vvvert_T\Big)\cdot\vvvert\Phi-\Phi_n\vvvert_T
\xrightarrow[n\to\infty]{}0
$.

\end{proof}

\begin{teo}\label{th_dist_stoch_integr}
Stochastic integral $I(\Phi)=\int\limits_0^T \Phi(t)dB_t$
with nonrandom integrand $\Phi(t)$ is $G$-normal distributed, , where 
$B_t\sim N_G\big(0,t\cdot\Sigma\big)$ is a $G$-Brownian motion. I.e.,

\ \ \ \ \ \ \ $I(\Phi)\sim N_G\big(0,\Sigma_I\big)$\,,\ \ \ where 
$ 
\Sigma_I=\Big\{\int\limits_0^T\Phi(t) Q \Phi^*(t)dt\mid\, Q\in\Sigma\Big\}$.
\end{teo}

\begin{proof}
 $\absz$\\
\textbf{1)} From the first part of this section we have got that

 $I_n:=I(\Phi_n)\sim N_G\big(0,\Sigma_{I_n}\big)$, \
 where \ 
$ 
\Sigma_{I_n}=\Big\{\int\limits_0^T\Phi_n(t) Q \Phi_n^*(t)dt\mid\, Q\in\Sigma\Big\}$.

\textbf{2)} $G_{I_n}(A)=\dfrac{1}{2}\E\big[\langle AI_n,I_n\rangle\big]$;
$G_{I}(A)=\dfrac{1}{2}\E\big[\langle AI,I\rangle\big]$.
\\
Then \ $2\cdot\Big|G_{I_n}(A)-G_{I}(A)\Big|
=\Big|\E\big[\langle AI_n,I_n\rangle\big]-\E\big[\langle AI,I\rangle\big]\Big|$

\hfill$
=\Big|\E\big[\langle AI_n,I_n\rangle\big]-\E\big[\langle AI_n,I\rangle\big]+\E\big[\langle AI_n,I\rangle\big]-\E\big[\langle AI,I\rangle\big]\Big|\\
\leq\Big|\E\big[\langle AI_n,I_n\rangle-\langle AI_n,I\rangle\big]\Big|+\Big|\E\big[\langle AI_n,I\rangle-\langle AI,I\rangle\big]\Big|\\
\leq \E\Big[\big|\langle AI_n,I_n-I\rangle\big|\Big]+\E\Big[\big|\langle A(I_n-I),I\rangle\big|\Big]\\
\leq \Bigg(\E\Big[\big\|AI_n\big\|^2_\Hs\Big]\Bigg)^{1/2}\cdot\Bigg(\E\Big[\big\|I_n-I\big\|^2_\Hs\Big]\Bigg)^{1/2}
+ \Bigg(\E\Big[\big\|A(I_n-I)\big\|^2_\Hs\Big]\Bigg)^{1/2}\cdot\Bigg(\E\Big[\big\|I\big\|^2_\Hs\Big]\Bigg)^{1/2}\\
\leq \big\|A\big\|_{L(\Hs)}\cdot\Bigg[\Bigg(\E\Big[\big\|I_n\big\|^2_\Hs\Big]\Bigg)^{1/2}
+\Bigg(\E\Big[\big\|I\big\|^2_\Hs\Big]\Bigg)^{1/2}\Bigg]\cdot\Bigg(\E\Big[\big\|I_n-I\big\|^2_\Hs\Big]\Bigg)^{1/2}\\
=\big\|A\big\|_{L(\Hs)}\cdot\Big(\|I_n\|_{\Omega_T}+\|I\|_{\Omega_T}\Big)\cdot\|I_n-I\|_{\Omega_T}
\xrightarrow[n\to\infty]{}0
$.
\\
So that \ $G_{I_n}(A)\xrightarrow[n\to\infty]{}G_{I}(A)$.
\\
We know that $G_I$ defines a covariation set $\Sigma_I$ of the $I(\Phi)$.
\\
\textbf{3)} Now we are going to proof that \ $I(\Phi)\sim N_G\big(0,\Sigma_I\big)$.
\\
In order to show such a fact we consider $u_n(t,x):=\E\big[f(x+\sqrt{T-t}\,I_n)\big]$ \ and \ $u(t,x):=\E\big[f(x+\sqrt{T-t}\,I)\big]$, \ with a $B$-continuous function 

\hfill$f\in \Cp(\Hs)$;

Since\ \ $I_n\sim N_G\big(0,\Sigma_{I_n}\big)$ \ then by \textbf{Th.\ref{th_exist_vs_p0}} we have that \ 
$u_n$ is a unique viscosity solution of the equation
\ \ \ $\begin{cases}
\partial_t u+G_{I_n}(D^2_{xx} u)=0 \,;\\
u(T,x)=f(x) \,.
\end{cases}$\hfill $(\,\ast\,)$
\\
So, we need to show that $u$ is a viscosity solution of the equation
\\
$\begin{cases}
\partial_t u+G_{I}(D^2_{xx} u)=0 \,;\\
u(T,x)=f(x) \,.
\end{cases}$\hfill $(\,\#\,)$
\\
\textbf{(i)} We claim that for every fixed point $(t,x):\ \ \ u_n(t,x)\xrightarrow[n\to0]{}u(t,x)$.

In fact,\ \ \ $\big|u_n(t,x)-u(t,x)\big|
=\Big|\E\big[f(x+\sqrt{T-t}\,I_n)\big]-\E\big[f(x+\sqrt{T-t}\,I)\big]\Big|\\
\leq\E\Big[\big|f(x+\sqrt{T-t}\,I_n)-f(x+\sqrt{T-t}\,I)\big|\Big]\\
\leq\E\Big[C\cdot\big(1+\|x+\sqrt{T-t}\,I_n\|^m_\Hs+\|x+\sqrt{T-t}\,I\|^m_\Hs\big)\cdot\big\|\sqrt{T-t}(I_n-I)\big\|_\Hs\Big]\\
\leq C\cdot\Bigg(\E\Big[\Big(1+\|x+\sqrt{T-t}\,I_n\|^m_\Hs+\|x+\sqrt{T-t}\,I\|^m_\Hs\Big)^2\,\Big]\Bigg)^{1/2}\times$

\hfill$
\times\Bigg((T-t)\cdot\E\Big[\big\|I_n-I\big\|^2_\Hs\Big]\Bigg)^{1/2}\\
\leq \widetilde{C}\cdot\|I_n-I\|_{\Omega_T}
\xrightarrow[n\to\infty]{}0
$.
\\
\textbf{(ii)} Also we claim that $u,u_n$ are continuous at $(t,x)\in[0,T]\tdot\Hs$.
\\
In fact, let us show that \ \ $\E\big[f(x+\sqrt{T-t-\delta}\,I)\big]\xrightarrow[\delta\to0]{}\E\big[f(x+\sqrt{T-t}\,I)\big]$:
\\
$\Big|\E\big[f(x+\sqrt{T-t-\delta}\,I)\big]-\E\big[f(x+\sqrt{T-t}\,I)\big]\Big|\\
\leq\E\Big[\big|f(x+\sqrt{T-t-\delta}\,I)-f(x+\sqrt{T-t}\,I)\big|\Big]\\
\leq\E\Big[C\cdot\big(1+\|x+\sqrt{T-t-\delta}\,I)\|^m_\Hs+\|x+\sqrt{T-t}\,I)\|^m_\Hs\big)\times$

\hfill$\times\big\|(\sqrt{T-t-\delta}-\sqrt{T-t})I\big\|_\Hs\Big]\\
\leq C\cdot\Big(\sqrt{T-t-\delta}-\sqrt{T-t}\Big)\times$

\hfill$\times\Bigg(\E\Big[\Big(1+\|x+\sqrt{T-t-\delta}\,I)\|^m_\Hs+\|x+\sqrt{T-t}\,I)\|^m_\Hs\Big)^2\,\Big]\Bigg)^{1/2}\times
\Bigg(\E\Big[\big\|I\big\|^2_\Hs\Big]\Bigg)^{1/2}\\
\leq \widetilde{C}\cdot\Big(\sqrt{T-t-\delta}-\sqrt{T-t}\Big)\cdot\|I\|_{\Omega_T}
\xrightarrow[\delta\to0]{}0
$.
\\
So that $u$ is continuous. It is clear that $u_n$ is continuous too.
\\
\textbf{(iii)} Let $\psi$ be a test function, such that: \ 
\parbox[t][\height]{5cm}{$u(t,x)\leq \psi(t,x)\,;\\ 
\phantom{.}\!u(t_0,x_0)=\psi(t_0,x_0)\,.$}
\\
Since for every fixed point $(t,x):\ \ \ u_n(t,x)\xrightarrow[n\to0]{}u(t,x)$ \ then there exists a sequence of test functions $\big\{\psi_n\big\}$, such that: \ 
\parbox[t][\height]{5cm}{$u_n(t,x)\leq \psi_n(t,x)\,;$\\
$u_n(t_0,x_0)=\psi_n(t_0,x_0)\,;$\\
$\psi_n(t,x)\xrightarrow[n\to0]{}\psi(t,x)\,.$}
\\
In order to show it we can take\ \ $\widetilde{\psi}_n:=\psi+u_n-u$ \ \ that satisfies above written required properties, and in the points where it is not enough smooth we need to alter it in the proper way to get the test function $\psi_n$.
\\
We know that $u_n$ is a viscosity sub- (and super-) solution of equation $(\,\ast\,)$.

So that \ $u_n(T,x)\leq f(x)$;

\hspace{18mm}$\Big[\partial_t\psi_n+G_{I_n}(D^{2}_{xx}\psi_n)\Big](t_0,x_0)\geq0$.
\\
Hence \ $u(T,x)\leq f(x)$\,,\ \ \  since $u_n(t,x)\xrightarrow[n\to0]{}u(t,x)$.

And \ $\Big[\partial_t\psi_n+G_{I_n}(D^{2}_{xx}\psi_n)\Big](t_0,x_0)
\xrightarrow[n\to0]{\textbf{2)}}\Big[\partial_t\psi+G_I(D^{2}_{xx}\psi)\Big](t_0,x_0)$.
\\
Therefore \ $\Big[\partial_t\psi+G_I(D^{2}_{xx}\psi)\Big](t_0,x_0)\geq0$.
\\
So, we have that $u$ is a viscosity subsolution of equation $(\,\#\,)$.

And in the same way we can show that $u$ is a viscosity supersolution.

So, we can conclude that  $u$ is a viscosity solution of equation $(\,\#\,)$.
\\
\textbf{3)} Now we are going to describe the structure of the covariation set $\Sigma_I$.
\\
Let us define a set \  $\Sigma:=\overline{\mathrm{conv}(\Sigma_I^\prime)}$\ \ in the trace-class topology, where 

$\Sigma_I^\prime:=
\Big\{B\in C_1(\Hs)\,,\ \  B=B^*\geq0 \mid \,\forall\, \varepsilon>0$

\hfill$
\dfrac{1}{2}\Tr\big[A B\big]\leq G_I(A)<\dfrac{1}{2}\Tr\big[A B\big]+\varepsilon\Big\}.
$

Therefore \ $G_I(A)=\dfrac{1}{2}\sup\limits_{B\in\Sigma_I}\Tr\big[A B\big]
=\dfrac{1}{2}\sup\limits_{B\in\Sigma_I^\prime}\Tr\big[A B\big]$.

Analogously, $G_{I_n}(A)=\dfrac{1}{2}\sup\limits_{B\in\Sigma_{I_n}}\Tr\big[A B\big]
=\dfrac{1}{2}\sup\limits_{B\in\Sigma_{I_n}^\prime}\Tr\big[A B\big]$.

Let us show that:

\hspace{10mm}\textbf{(a)} If $\{B_n\}\subset\Sigma_{I_n}^\prime$, \ such that \ $B_n\xrightarrow[n\to\infty]{}B$ \ then \ $B\in\Sigma_I^\prime$.

\hspace{10mm}\textbf{(b)} For every $B\in\Sigma_{I}^\prime$ \ there exists a sequence \ $\{B_n\}\subset\Sigma_{I_n}^\prime$, 

\hfill such that \ $B_n\xrightarrow[n\to\infty]{}B$.

In fact, let us fix an operator $A$ then:

\textbf{(a)} If $\{B_n\}\subset\Sigma_{I_n}^\prime$, \ such that \ $B_n\xrightarrow[n\to\infty]{}B$ \ then we have:

$\dfrac{1}{2}\Tr\big[A B_n\big]\leq G_{I_n}(A)<\dfrac{1}{2}\Tr\big[A B_n\big]+\varepsilon$;

Letting $n\to\infty$ \ yields \ $\dfrac{1}{2}\Tr\big[A B\big]\leq G_I(A)<\dfrac{1}{2}\Tr\big[A B\big]+\varepsilon$.
\\
Hence \ $B\in\Sigma_I^\prime$.

\textbf{(b)} Let there exists $B\in\Sigma_{I}^\prime$, and a sequence \ $\{B_n\}\subset\Sigma_{I_n}^\prime$ converges to the operator $C\neq B: \ \ \ B_n\xrightarrow[n\to\infty]{}C\in\Sigma_{I}^\prime$. Then we have:

$\dfrac{1}{2}\Tr\big[A B\big]\leq G_I(A)<\dfrac{1}{2}\Tr\big[A B\big]+\varepsilon$;
\\
$\dfrac{1}{2}\Tr\big[A C\big]\leq G_I(A)<\dfrac{1}{2}\Tr\big[A C\big]+\varepsilon$.
\\
Therefore for every $A: \ \ \ \Tr\big[A B\big]=\Tr\big[A C\big]$.
\\
So, it easy to check that \ $B=C$, \ a contradiction.
\\
So, from \textbf{(a)} and \textbf{(b)} we have that

$\Sigma_{I}^\prime=\big\{B\mid \,\ B_n\xrightarrow[n\to\infty]{C_1(\Hs)}B\,,\ \ B_n\in\Sigma_{I_n}^\prime\big\}$,\ \ \ 
where $B_n=\int\limits_0^T\Phi_n(t) Q \Phi_n^*(t)dt$.

Hence \ $\Sigma_{I}=\big\{B\mid \,\ B_n\xrightarrow[n\to\infty]{C_1(\Hs)}B\,,\ \ B_n\in\Sigma_{I_n}\big\}$.

Applying \textbf{Lm.\ref{lm_converg_cov_set_elements}} we can conclude that \ $
\Sigma_I=\Big\{\int\limits_0^T\Phi(t) Q \Phi^*(t)dt\mid\, Q\in\Sigma\Big\}$.

\end{proof}

\subsection{The continuity property of stochastic convolution}

Define a \textbf{stochastic convolution} as the integral 
$$I_t:=\int\limits_0^t e^{(t-s)A}dB_s\,,$$
where $A:D(A)\to\Hs$ is the infinitesimal generator of $C_0$-semigroup $\big( e^{tA}\big)$.
 \begin{teo}
The integral \ $I_t:=\int\limits_0^t e^{(t-s)A}dB_s$ \ is continuous for quasi every $\omega\ $
if there exists $\beta>0$, \ such that \ 
$$\int\limits_0^T \big\|e^{tA}\|^2_{L_2^\Sigma}\cdot t^{-\beta} dt<\infty\ .$$
\end{teo}
 \begin{proof}
 $\absz$\\
We shall use the factorization method (see \cite{da_prato_kwapien_zabcz1}).

For this reason we will use the following elementary inequality:

Let $\alpha\in(0,1)$ \ then \ 

$\int\limits_0^1(1-r)^{\alpha-1}\cdot r^{-\alpha}\,dr
=B(\alpha,1-\alpha)
=\dfrac{\Gamma(\alpha)\Gamma(1-\alpha)}{\Gamma(1)}
=\dfrac{\pi}{\sin\pi\alpha}$.

It follows that
\begin{equation}\label{eq_fact_meth}
\int\limits_\sigma^t(t-s)^{\alpha-1}\cdot (s-\sigma)^{-\alpha}\,ds=\dfrac{\pi}{\sin\pi\alpha}\ \ ,
\end{equation}
\hfill $0\leq\sigma\leq s\leq t$,\ \ \ where $s:=r(t-\sigma)+\sigma$,

because
 
$\int\limits_\sigma^t(t-s)^{\alpha-1}\cdot (s-\sigma)^{-\alpha}\,ds
=\int\limits_\sigma^t\big((1-r)(t-\sigma)\big)^{\alpha-1}\cdot \big(r(t-\sigma)\big)^{-\alpha}\,d[r(t-\sigma)+\sigma]\\
=\int\limits_0^1(1-r)^{\alpha-1}\cdot r^{-\alpha}\cdot(t-\sigma)^{\alpha-1}\cdot(t-\sigma)^{-\alpha}\cdot(t-\sigma)\,dr$

\hfill$
=\int\limits_0^1(1-r)^{\alpha-1}\cdot r^{-\alpha}\,dr
=\dfrac{\pi}{\sin\pi\alpha}
$.
\\
Let $\alpha\in(0,\frac{1}{2})$ be fixed, and $m>\dfrac{1}{2\alpha}$ then we have

$I_t
=\dfrac{\sin\pi\alpha}{\pi}\ \int\limits_0^t e^{(t-s)A}\ \int\limits_\sigma^t(t-s)^{\alpha-1}\cdot (s-\sigma)^{-\alpha}\,ds\ dB_s$.
\\
From the Fubini theorem (\textbf{Th.\ref{th_fubini}})  we get that 

$I_t
=\dfrac{\sin\pi\alpha}{\pi}\ \int\limits_0^t e^{(t-s)A}\cdot (t-s)^{\alpha-1}Y(s)\,ds$\ \ \ quasi surely,

\hfill where $Y(s)=\int\limits_0^s e^{(s-\sigma)A}\, (s-\sigma)^{-\alpha}\,dB_\sigma$.

Then by \textbf{Th.\ref{th_dist_stoch_integr}} we have for every $s \ \ \ Y(s)\sim N_G(0,\Sigma_{I_s})$
\\

\hfill where $\Sigma_{I_s}
=\Big\{\int\limits_0^s e^{(s-\sigma)A}\,Q\, e^{(s-\sigma)A^*}\,(s-\sigma)^{-2\alpha} d\sigma\,,\ Q\in\Sigma\Big\}$.
\\
Therefore 
\\
$\E\Big[\big\|Y(s)\big\|_\Hs^2\Big]
\overset{\textbf{Th.\ref{th_ito_is_inq}}}{\leq}\int\limits_0^s\big\|Y(\sigma)\|^2_{L_2^\Sigma}\,d\sigma
=\int\limits_0^s\sup\limits_{Q\in\Sigma}\Tr\big[e^{(s-\sigma)A}\,Q\, e^{(s-\sigma)A^*}\,(s-\sigma)^{-2\alpha}\Big]d\sigma\\
=\int\limits_0^s\big\|e^{(s-\sigma)A}\|^2_{L_2^\Sigma}\,(s-\sigma)^{-2\alpha}\,d\sigma
=\Big\lfilet s-\sigma=t\in(s,0)\Big\lfilet
=\int\limits_0^s\big\|e^{tA}\|^2_{L_2^\Sigma}\,t^{-2\alpha}\,dt<\infty
$.

Then by \textbf{Prop.\ref{prop_moments}} it follows \ $\E\Big[\big\|Y(s)\big\|_\Hs^{2m}\Big]\leq C_m\,, \ \ s\in[0,T].$

Hence  \ $\E\Big[\int\limits_0^T\big\|Y(s)\big\|_\Hs^{2m}\,ds\Big]\leq C_m\cdot T$, \
so that \ $Y\in L^{2m}(0,T;\Hs)$.
\\
Let us consider \ $z(t)=\int\limits_0^t e^{(t-s)A}\, (t-s)^{\alpha-1}y(s)\,ds$.

Set $z_\varepsilon(t):=\int\limits_0^{t-\varepsilon} e^{(t-s)A}\, (t-s)^{\alpha-1}y(s)\,ds$, \ for a small enough $\varepsilon>0$.

So, we have \ $\big|z(t)-z_\varepsilon(t)\big|
=\int\limits_{t-\varepsilon}^t e^{(t-s)A}\, (t-s)^{\alpha-1}y(s)\,ds\\
\overset{\textbf{Hölder ineq.}}{\leq} 
\Big(\int\limits_{t-\varepsilon}^t \big\|e^{(t-s)A}\big\|_{L(\Hs)}^{\frac{2m}{2m-1}}\cdot(t-s)^{\frac{2m(\alpha-1)}{2m-1}}\,ds\Big)^{\frac{2m-1}{2m}}
\cdot\Big(\int\limits_{t-\varepsilon}^t \big\|y(s)\big\|_{\Hs}^{2m}\,ds\Big)^{\frac{1}{2m}}\\
\leq M\cdot e^{\frac{2m\varepsilon a}{2m-1}}\cdot
\Big(\int\limits_0^\varepsilon r^{\frac{2m(\alpha-1)}{2m-1}}\,dr\Big)^{\frac{2m-1}{2m}}\cdot\|y\|_{L^{2m}(0,\,t;\,\Hs)}
\leq K_\varepsilon\cdot\|y\|_{L^{2m}(0,\,t;\,\Hs)}$,

\hfill$K_\varepsilon\xrightarrow[\varepsilon\to0]{}0\ .
$
\\
So that \ $z(\pdot)$ is continuous if $y(\pdot)\in L^{2m}(0,T;\Hs)$.
\\
And we have that $I_t$ is continuous for quasi every $\omega$.\\[-3mm]
\end{proof}

\newpage
\section[Viscosity solution for other parabolic PDEs]{Existence of viscosity solution for parabolic PDE with a linear unbounded first order term}

\subsection{Ornstein-Uhlenbeck process}

Consider the following SDE:
\begin{equation}\tag{S}\label{eq_sde_s}
\begin{cases}
d X_\tau=AX_\tau\,d\tau +dB_\tau\ ,\ \ \ \ \ \ \tau\in[t,T]\subset[0,T]\\
X_t=x\ .
\end{cases}
\end{equation}

where $X_t:[0,T]\tdot\Omega\rightarrow\Hs;$

$B_t$ is a $G$-Brownian motion in $\Hs;$ 

$A:D(A)\to\Hs$ is an infinitesimal generator of $C_0$-semigroup $\big( e^{tA}\big)$.

\begin{df}
 The process
$$X_\tau:=X_\tau^{t,\,x}=e^{(\tau-t)A}x+\int\limits_t^\tau e^{(\tau-s)A}dB_s$$
will be called a mild solution of \eqref{eq_sde_s}.
\end{df}

\begin{df}
Stochastic process \begin{equation}I_t=\int\limits_0^t e^{(t-s)A}dB_s\end{equation} will be called Ornstein-Uhlenbeck process (or a stochastic convolution as we have already mentioned above). 

Ornstein-Uhlenbeck process is well defined under the condition
\begin{equation}\label{cond_OU}
 \int\limits_0^t\sup\limits_{Q\in\,\Sigma}\big\|e^{sA}Q^{1/2}\|^2_{L_2(\Hs)} ds<\infty.
\end{equation}
\end{df}

\begin{rem}\label{rem_cond_OU}
 The condition \eqref{cond_OU} holds true if\ \  $\sup\limits_{Q\in\,\Sigma}\Tr Q<\infty$.
\end{rem}

In other words the condition \eqref{cond_OU} (or the condition of \textbf{Rem.\ref{rem_cond_OU}}) implies that 
the map $s\mapsto e^{(t-s)A}$ belongs to the ${^\Hs\!M^{2}_G\big( 0,t \big)}$, what shows us the following proposition.

\begin{prop}
 If\ \  $\sup\limits_{Q\in\,\Sigma}\Tr Q<\infty\,,\ \Phi(s):=e^{(t-s)A}$ \ then \ 

\hfill$\Phi(s)\in {^\Hs\!M^{2}_G\big( 0,t \big)}$.
\end{prop}

\begin{proof}
 $\absz$\\
In order to prove the statement of proposition according to \textbf{Rem.\ref{rem_m2gh}} we need to show 
that $\vvvert \Phi\vvvert_t<\infty$.

According to theory of $C_0$-semigroup it is known that 

$e^{sA}\in L(\Hs)\,, \ \ \|e^{sA}\|^2_{L(\Hs)}\leq M\cdot e^{as}\, \ \ M\geq1\,,\ a\in\R.\phantom{///}$

Therefore \ $\vvvert \Phi\vvvert_t^2=\E\Big[\int\limits_0^t \sup\limits_{Q\in\,\Sigma}\|\Phi(s)Q^{1/2}\|^2_{L_2(\Hs)} \, ds\Big]=
\int\limits_0^t \sup\limits_{Q\in\,\Sigma}\|e^{(t-s)A}Q^{1/2}\|^2_{L_2(\Hs)} \, ds
\\=
\underbrace{\int\limits_0^t \sup\limits_{Q\in\,\Sigma}\|e^{sA}Q^{1/2}\|^2_{L_2(\Hs)} \, ds}_{\text{condition} \eqref{cond_OU}} \leq 
\int\limits_0^t \sup\limits_{Q\in\,\Sigma}\Big[\|ne^{sA}\|^2_{L(\Hs)}\cdot\|Q^{1/2}\|^2_{L_2(\Hs)}\Big] \, ds$

\hfill$ \leq 
\int\limits_0^t M\cdot e^{as}\cdot\sup\limits_{Q\in\,\Sigma}\Tr Q \, ds\leq 
M\cdot \dfrac{e^{at}-1}{a}\cdot\sup\limits_{Q\in\,\Sigma}\Tr Q<\infty$.

\end{proof}

\begin{prop}\label{prop_ch_tm_OU}
 $X_\tau^{t,\,x}=X_\tau^{s,\,X_s^{t,\,x}},\ \ 0\leq t\leq s\leq \tau\leq T$.
\end{prop}

\begin{proof}
$\absz$\\
$
X_\tau^{t,\,x}=e^{(\tau-t)A}x+\int\limits_t^\tau e^{(\tau-\sigma)A}dB_\sigma
=e^{(\tau-s)A}\cdot e^{(s-t)A} x+\int\limits_t^s e^{(s-\sigma)A}\cdot e^{(\tau-s)A}dB_\sigma\\
+
\int\limits_s^\tau e^{(\tau-\sigma)A}dB_\sigma
=e^{(\tau-s)A}\Big(e^{(s-t)A}x+\int\limits_t^s e^{(s-\sigma)A}dB_\sigma\Big)+\int\limits_s^\tau e^{(\tau-\sigma)A}dB_\sigma$

\hfill$
=e^{(\tau-s)A}\cdot X_s^{t,\,x}+\int\limits_s^\tau e^{(\tau-\sigma)A}dB_\sigma
=X_\tau^{s,\,X_s^{t,\,x}}.
$

\end{proof}

\subsection{Solving the equation \eqref{eq_p}}

\begin{lm}\label{lm_conv_g_f}
Let $B_t$ \ be a $G$-Brownian motion and $A$ \  be an infinitesimal generator of $C_0$ semigroup.
A mapping \ $\psi:\R,\Hs\to\R$ \  is twice Fréchet differentiable by $x$.

For the small $\delta>0$ define the following random variable as:

$L_\delta:=\dfrac{1}{\delta}\Big\langle D^{2}_{xx}\psi(t,x)[\int\limits_0^{\delta} e^{(\delta -s)A}dB_s],
 \int\limits_0^{\delta} e^{(\delta -s)A}dB_s\Big\rangle$.

Then\qquad $\E[L_\delta]\xrightarrow[\delta\to0]{}\E\Big[\big\langle  D^{2}_{xx}\psi(t,x)B_1,B_1 \big\rangle\Big]
\equiv2\,G\big(D^{2}_{xx}\psi(t,x)\big)$.
\end{lm}

\begin{proof}
 $\absz$\\
Let $L_0:=\big\langle  D^{2}_{xx}\psi(t,x)B_1,B_1 \big\rangle$.
\\
Note that 
\begin{multline}\label{eq_expect_ineq}
 \E\Big[\langle AX,Y\rangle\Big]\leq\E\Big[\|AX\|_\Hs\cdot\|Y\|_\Hs\Big]\leq\E\Big[\|A\|_{L(\Hs)}\cdot\|X\|_\Hs\cdot\|Y\|_\Hs\Big]\\
\overset{\textbf{Prop.\ref{prop_cb_ineq}}}{\leq}
\|A\|_{L(\Hs)}\Big(\E\Big[\|X\|_\Hs^2\Big]\Big)^{\frac{1}{2}}\Big(\E\Big[\|Y\|_\Hs^2\Big]\Big)^{\frac{1}{2}}.
\end{multline}

Consider $K_\delta:=\dfrac{1}{\delta}\big\langle D^{2}_{xx}\psi(t,x)\int\limits_0^{\delta} (e^{(\delta -s)A}-I)dB_s,B_\delta \big\rangle$ 

\hspace{.8cm} and $M_\delta:=\dfrac{1}{\delta}\big\langle D^{2}_{xx}\psi(t,x)\int\limits_0^{\delta} e^{(\delta -s)A}dB_s,\int\limits_0^{\delta} (e^{(\delta -s)A}-I)dB_s \big\rangle$.
\\
Then we have:

$0\leq\E\big[\big|K_\delta\big|\big]\overset{\textbf{\eqref{eq_expect_ineq}}}{\leq}
\dfrac{1}{\delta}\|D^{2}_{xx}\psi(t,x)\|_{L(\Hs)}
\Big(\E\Big[\|\int\limits_0^{\delta} (e^{(\delta -s)A}-I)dB_s\|_{\Hs}^2\Big]\Big)^{\frac{1}{2}}
\Big(\E\Big[\|B_\delta\|_{\Hs}^2\Big]\Big)^{\frac{1}{2}}\\
\overset{\textbf{Th.\ref{th_ito_BDG_general}}}{\leq}
\dfrac{1}{\delta}\|D^{2}_{xx}\psi(t,x)\|_{L(\Hs)}\cdot C 
\Big(\E\Big[\int\limits_0^{\delta} \|e^{(\delta -s)A}-I\|^2_{L^\Sigma_2} ds\Big]\Big)^{\frac{1}{2}}
\Big(\E\Big[\|B_\delta\|_{\Hs}^2\Big]\Big)^{\frac{1}{2}}\\
\overset{\textbf{Rem.\ref{rem_lin_gf_gbm}}}{=}
\dfrac{1}{\delta}\|D^{2}_{xx}\psi(t,x)\|_{L(\Hs)}\cdot C 
\Big(\int\limits_0^{\delta} \underset{Q\in\,\Sigma}{\sup}\|e^{sA}Q^{1/2}-Q^{1/2}\|^2_{L_2(\Hs)} ds\Big)^{\frac{1}{2}}
\cdot\sqrt{\delta G(I)}\\
\leq\|D^{2}_{xx}\psi(t,x)\|_{L(\Hs)}\cdot C \cdot\delta\cdot
\max\limits_{0\leq\theta\leq\delta}\underset{Q\in\,\Sigma}{\sup}\|e^{\theta A}Q^{1/2}-Q^{1/2}\|_{L_2(\Hs)}\cdot\sqrt{G(I)}
\xrightarrow[\delta\to0]{}0$.
\parbox{155mm}{
In much the same way:

$0\leq\E\big[\big|M_\delta\big|\big]
\overset{\textbf{\eqref{eq_expect_ineq}}}{\leq}
\dfrac{1}{\delta}\|D^{2}_{xx}\psi(t,x)\|_{L(\Hs)}
\Big(\E\Big[\|\int\limits_0^{\delta} e^{(\delta -s)A}dB_s\|_{\Hs}^2\Big]\Big)^{\frac{1}{2}}
\Big(\E\Big[\|\int\limits_0^{\delta} (e^{(\delta -s)A}-I)dB_s\|_{\Hs}^2\Big]\Big)^{\frac{1}{2}}\\
\overset{\textbf{Th.\ref{th_ito_BDG_general}}}{\leq}
\dfrac{1}{\delta}\|D^{2}_{xx}\psi(t,x)\|_{L(\Hs)}\cdot C 
\Big(\E\Big[\int\limits_0^{\delta} \|e^{(\delta -s)A}\|^2_{L^\Sigma_2} ds\Big]\Big)^{\frac{1}{2}}
\Big(\E\Big[\int\limits_0^{\delta} \|e^{(\delta -s)A}-I\|^2_{L^\Sigma_2} ds\Big]\Big)^{\frac{1}{2}}\\
=\dfrac{1}{\delta}\|D^{2}_{xx}\psi(t,x)\|_{L(\Hs)}\cdot C 
\Big(\int\limits_0^{\delta} \underset{Q_1\in\,\Sigma}{\sup}\|e^{sA}Q_1^\frac{1}{2}\|_{L_2(\Hs)}^2 ds\Big)^{\frac{1}{2}}
\Big(\int\limits_0^{\delta} \underset{Q_2\in\,\Sigma}{\sup}\|e^{sA}Q_2^\frac{1}{2}-Q_2^\frac{1}{2}\|_{L_2(\Hs)}^2 ds\Big)^{\frac{1}{2}}\\
\leq\|D^{2}_{xx}\psi(t,x)\|_{L(\Hs)}\cdot C \cdot
\delta\cdot\max\limits_{0\leq\theta_1\leq\delta}\underset{Q_1\in\,\Sigma}{\sup}\|e^{\theta_1 A}Q_1^\frac{1}{2}\|_{L_2(\Hs)}\cdot\delta
\max\limits_{0\leq\theta_2\leq\delta}\underset{Q_2\in\,\Sigma}{\sup}\|e^{\theta_2 A}Q_2^\frac{1}{2}-Q_2^\frac{1}{2}\|_{L_2(\Hs)}$

\hfill$
\xrightarrow[\delta\to0]{}0$.}
\\
Therefore 

$0\leq\Big|\E[L_\delta]-\E[L_0]\Big|
\overset{\textbf{Rem.\ref{rem_lin_gf_gbm}}}{=}
\Bigg|\dfrac{1}{\delta}\,\E\Big[\Big\langle D^{2}_{xx}\psi(t,x)\Big[\int\limits_0^{\delta} e^{(\delta -s)A}dB_s\Big],
 \int\limits_0^{\delta} e^{(\delta -s)A}dB_s\Big\rangle\Big]$

\hfill$
-\dfrac{1}{\delta}\,\E\Big[\big\langle  D^{2}_{xx}\psi(t,x)B_\delta,B_\delta \big\rangle\Big]\Bigg|\\
\overset{\textbf{Prop.\ref{prop_elem_prop},\,2)}}{\leq}
\dfrac{1}{\delta}\,\E\Bigg[\Big|\Big\langle D^{2}_{xx}\psi(t,x)\Big[\int\limits_0^{\delta} e^{(\delta -s)A}dB_s\Big],
 \int\limits_0^{\delta} e^{(\delta -s)A}dB_s\Big\rangle$

\hfill$
-\big\langle  D^{2}_{xx}\psi(t,x)B_\delta,B_\delta \big\rangle\Big|\Bigg]\\
\leq\dfrac{1}{\delta}\,\E\Bigg[\Big|\Big\langle D^{2}_{xx}\psi(t,x)\Big[\int\limits_0^{\delta} e^{(\delta -s)A}dB_s\Big],
 \int\limits_0^{\delta} e^{(\delta -s)A}dB_s\Big\rangle$

\hfill$
-\big\langle  D^{2}_{xx}\psi(t,x)\Big[\int\limits_0^{\delta} e^{(\delta -s)A}dB_s\Big],B_\delta \big\rangle\Big|\Bigg]\\
+\dfrac{1}{\delta}\,\E\Bigg[\Big|\Big\langle D^{2}_{xx}\psi(t,x)\Big[\int\limits_0^{\delta} e^{(\delta -s)A}dB_s\Big],
 B_\delta\Big\rangle
-\big\langle  D^{2}_{xx}\psi(t,x)B_\delta,B_\delta \big\rangle\Big|\Bigg]$

\hfill$
=\E\big[\big|K_\delta\big|\big]+\E\big[\big|M_\delta\big|\big]\xrightarrow[\delta\to0]{}0$.

\end{proof}

Now let us turn back to equation \eqref{eq_p}:

\begin{equation}\tag{P}\label{eq_pp}
\hspace{20mm}\begin{cases}
\partial_t u+\langle Ax,D_x u\rangle+G(D^2_{xx} u)=0 \,, \ \ \ t\in[0,T)\,, \ x\in\Hs;\\
u(T,x)=f(x) \,.
\end{cases}
\end{equation}

$u:[0,T]\times\Hs\to\R$;

$f\in\Cp(\Hs)$;

$G:K_S(\Hs)\rightarrow\R$ \ is a \ $G$-functional;

$A:D(A)\to\Hs$ is a generator of $C_0$-semigroup $\big(e^{tA}\big)$.

$B_t$ is a $G$-Brownian motion with correspondent $G$-functional $G(\pdot)$,\ \ \ \ 

\hfill i.e. $G(A)=\dfrac{1}{2t}\,\E\Big[\langle A B_t, B_t\rangle\Big]$;

$X_\tau^{t,\,x}=e^{(\tau-t)A}x+\int\limits_t^\tau e^{(\tau-s)A}dB_s$\ \ \  be a mild solution of equation \eqref{eq_sde_s}:
\begin{equation}\tag{S}\label{eq_sde_ss}
\begin{cases}
d X_\tau=AX_\tau\,d\tau +dB_\tau\ ,\ \ \ \ \ \ \tau\in[t,T]\subset[0,T]\\
X_t=x\ .
\end{cases}
\end{equation}
 
\begin{teo} \label{th_exist_vs_p}
Let $f$ is a $B$-continuous of \ $\Cp(\Hs)$-class real function. Then $u(t,x):=\E \big[f(X_T^{t,\,x})\big]$ is a unique viscosity solution of equation \eqref{eq_pp}:
\begin{equation}\tag{P}\label{eq_ppp}
\hspace{20mm}\begin{cases}
\partial_t u+\langle Ax,D_x u\rangle+G(D^2_{xx} u)=0 \,, \ \ \ t\in[0,T)\,, \ x\in\Hs;\\
u(T,x)=f(x) \,.
\end{cases}
\end{equation}

\end{teo}

\begin{proof}
$\absz$\\
Let $\psi$ be a test function, and for every fixed point $(t,x)\in[0,T]\tdot\Hs$ we have:  
\parbox[t][\height]{5cm}{$u\leq \psi\,;\\ 
\phantom{|}\!u(t,x)=\psi(t,x)\,.$}
\\
Taking a small enough $\delta$ yields:
\\
$\psi(t,x)=u(t,x)
=\E[ f(X_T^{t,\,x})]
\overset{\textbf{Prop.\ref{prop_ch_tm_OU}}}{=}
\E[ f(X_T^{s,\,X_s^{t,\,x}})]$

\hfill$=
\E\Bigg[\E\Big[ f(X_T^{s,\,y})\Big]_{y=X_s^{t,\,x}}\Bigg]=
\E\Big[u(s,X_s^{t,\,x})\Big]
\leq\E\big[\psi(s,X_s^{t,\,x})\big]$.
\\
Then putting \ $s:=t+\delta,\ \delta>0$ \ by the Taylor formula (\textbf{Lm.\ref{lm_taylor_formula}}) we have:
\\
$\psi(s,X_s^{t,\,x})=\psi(t+\delta,X_{t+\delta}^{t,\,x})
=\psi(t+\delta,e^{\,\delta A}x+\int\limits_t^{t+\delta} e^{(t+\delta-s)A}dB_s)\\
\overset{\textbf{Rem.\ref{rem_st_int_shift}}}{=}\psi\Big(t+\delta,x+(e^{\,\delta A}x-x)+\int\limits_0^{\delta} e^{(\delta-s)A}dB_s\Big)
=\psi(t,x)+\delta\,\partial_t\psi(t,x)\\
+\Big\langle D_x\psi(t,x), e^{\,\delta A}x-x+\int\limits_0^{\delta} e^{(\delta -s)A}dB_s\Big\rangle
+\dfrac{1}{2}\,\delta^{\,2}\,\partial^{\,2}_{tt}\psi(t,x)\\
+\delta\cdot\partial_t\Big\langle D_x\psi(t,x),e^{\delta A}x-x+\int\limits_0^{\delta} e^{(\delta -s)A}dB_s\Big\rangle\\
+\dfrac{1}{2}\,\delta\Big\langle D^{2}_{xx}
\psi(t,x)\Big[e^{\delta A}x-x+\int\limits_0^{\delta} e^{(\delta -s)A}dB_s\Big],
e^{\delta A}x-x+\int\limits_0^{\delta} e^{(\delta -s)A}dB_s\Big\rangle$

\hfill$
+o\Big(\delta^2+\big\|e^{\,\delta A}x-x+\int\limits_0^{\delta} e^{(\delta-s)A}dB_s\big\|_{\Hs}^2\Big)$.
\\
We can say that \ \ $\E\Big[o\Big(\delta^2+\big\|e^{\,\delta A}x-x
+\int\limits_0^{\delta} e^{(\delta-s)A}dB_s\big\|_{\Hs}^2\Big)\Big]=o\Big(\delta^2\Big)$,
\\
because \ \ 
$\E\Big[\big\|e^{\,\delta A}x-x+\int\limits_0^{\delta} e^{(\delta-s)A}dB_s\big\|_{\Hs}^2\Big]
\leq 2\,\E\Big[\big\|e^{\,\delta A}x-x\big\|_{\Hs}^2\Big]$

\hfill$
+2\,\E\Big[\big\|\int\limits_0^{\delta} e^{(\delta-s)A}dB_s\big\|_{\Hs}^2\Big]\\
\overset{\textbf{Th.\ref{th_ito_BDG_general}}}{\leq}2\,\big\|e^{\,\delta A}x-x\big\|_{\Hs}^2
+2\,\int\limits_0^{\delta} \sup\limits_{Q\in\Sigma}\big\|e^{(\delta-s)A}\,Q^{1/2}\big\|_{L_2(\Hs)}^2 ds
\xrightarrow[\delta\to0]{}0$.
\\
Then we have
\\
$0\leq\dfrac{1}{\delta}\Big(\E\Big[\psi(t+\delta,X_{t+\delta}^{t,\,x})\Big]-\psi(t,x)\Big)
\leq\dfrac{1}{\delta}\E\Big[\psi(t+\delta,X_{t+\delta}^{t,\,x})-\psi(t,x)\Big]\\
=\E\Bigg[\partial_t\psi(t,x)
+\Big\langle D_x\psi(t,x), \dfrac{e^{\,\delta A}x-x}{\delta}\Big\rangle
+\dfrac{1}{\delta}\Big\langle D_x\psi(t,x), \int\limits_0^{\delta} e^{(\delta -s)A}dB_s\Big\rangle\\
+\dfrac{1}{2}\,\delta\,\partial^{\,2}_{tt}\psi(t,x)
+\delta\cdot\partial_t\Big\langle D_x\psi(t,x), \dfrac{e^{\,\delta A}x-x}{\delta}\Big\rangle
+\partial_t\Big\langle D_x\psi(t,x), \int\limits_0^{\delta} e^{(\delta -s)A}dB_s\Big\rangle$
\\
$
+\dfrac{1}{2}\Big\langle D^{2}_{xx}
\psi(t,x)\Big[\dfrac{e^{\,\delta A}x-x}{\delta}\Big],\int\limits_0^{\delta} e^{(\delta -s)A}dB_s\Big\rangle$

\hfill$
+\dfrac{1}{2}\Big\langle D^{2}_{xx}
\psi(t,x)\int\limits_0^{\delta} e^{(\delta -s)A}dB_s,\dfrac{e^{\,\delta A}x-x}{\delta}\Big\rangle\\
+\dfrac{\delta}{2}\Big\langle D^{2}_{xx}
\psi(t,x)\Big[\dfrac{e^{\,\delta A}x-x}{\delta}\Big],\dfrac{e^{\,\delta A}x-x}{\delta}\Big\rangle$

\hfill$
+\dfrac{1}{2\,\delta}\Big\langle D^{2}_{xx}
\psi(t,x)\int\limits_0^{\delta} e^{(\delta -s)A}dB_s,\int\limits_0^{\delta} e^{(\delta -s)A}dB_s\Big\rangle\Bigg]
+o\Big(\delta\Big)$
\\
$
\overset{\textbf{Prop.\ref{prop_elem_prop}.5).(a)-(b)}}{=}
\partial_t\psi(t,x)
+\Big\langle D_x\psi(t,x), \dfrac{e^{\,\delta A}x-x}{\delta}\Big\rangle
+\dfrac{1}{2}\,\delta\,\partial^{\,2}_{tt}\psi(t,x)\\
+\delta\cdot\partial_t\Big\langle D_x\psi(t,x), \dfrac{e^{\,\delta A}x-x}{\delta}\Big\rangle
+\dfrac{\delta}{2}\Big\langle D^{2}_{xx}
\psi(t,x)\Big[\dfrac{e^{\,\delta A}x-x}{\delta}\Big],\dfrac{e^{\,\delta A}x-x}{\delta}\Big\rangle\\
+\dfrac{1}{2\,\delta}\,\E\Bigg[\Big\langle D^{2}_{xx}
\psi(t,x)\int\limits_0^{\delta} e^{(\delta -s)A}dB_s,\int\limits_0^{\delta} e^{(\delta -s)A}dB_s\Big\rangle\Bigg]
+o\Big(\delta\Big)$

\hfill$
\xrightarrow[\delta\to0]{\textbf{Lm.\ref{lm_conv_g_f}}}
\partial_t\psi(t,x)+\big\langle x,\,A^* D_x\psi(t,x)\big\rangle+G\big(D^{2}_{xx}\psi(t,x)\big)
$.
\\
Letting $\delta\to0$ \ yields \ $\Big[\partial_t\psi+\big\langle x,\,A^* D_x\psi(t,x)\big\rangle+ G(D^{2}_{xx}\psi)\Big](t,x)\geq0$.
\\
Note, that $u$ is continuous at $(t,x)\in[0,T]\tdot\Hs$.
\\
In fact, let us show that $\ \  \E\big[f(x+B_s)\big]\xrightarrow[s\to t]{}\E\big[f(x+B_t)\big]\ ,\ \ \ t\in[0,T]\ :$

$0\leq\Big|\E[ f(X_T^{t+\delta,\,x})]-\E[ f(X_T^{t,\,x})]\Big|
\leq\E\Big[\big|f(X_T^{t+\delta,\,x})-f(X_T^{t,\,x})\big|\Big]\\
\leq C\cdot\E\Big[\big(1+\|X_T^{t+\delta,\,x}\|_{\Hs}^m+\|X_T^{t,\,x}\|_{\Hs}^m\big)\cdot\|X_T^{t+\delta,\,x}-X_T^{t,\,x}\|_{\Hs}\Big]\\
\leq C\cdot\Bigg(\E\Big[\big(1+\|X_T^{t+\delta,\,x}\|_{\Hs}^m+\|X_T^{t,\,x}\|_{\Hs}^m\big)^2\Big]
\cdot\E\Big[\|X_T^{t+\delta,\,x}-X_T^{t,\,x}\|_{\Hs}^2\Big]\Bigg)^{\frac{1}{2}}\\
\leq 2C\cdot\Bigg(\Big(1+\E\Big[\|X_T^{t+\delta,\,x}\|_{\Hs}^{2m}\Big]+\E\Big[\|X_T^{t,\,x}\|_{\Hs}^{2m}\Big]\Big)
\cdot\E\Big[\|X_T^{t+\delta,\,x}-X_T^{t,\,x}\|_{\Hs}^{2}\Big]\Bigg)^{\frac{1}{2}}
\xrightarrow[\delta\to 0]{}0,
$
\\
this convergence is true because:
\\
\textbf{(a)} $\E\Big[\|X_T^{t,\,x}\|_{\Hs}^{2m}\Big]= 
\E\Big[\|e^{(T-t)A}x+\int\limits_t^Te^{(T-s)A}dB_s\|_{\Hs}^{2m}\Big]\\
\leq 2m\cdot\Bigg(\|e^{(T-t)A}x\|_{\Hs}^{2m}+\E\Big[\|\int\limits_t^Te^{(T-s)A}dB_s\|_{\Hs}^{2m}\Big]\Bigg)\\
\overset{\textbf{Th.\ref{th_ito_BDG_general}}}{\leq}2m\cdot\Bigg(\|e^{(T-t)A}x\|_{\Hs}^{2m}
+C_m\cdot\Big(\int\limits_t^T\sup\limits_{Q\in\Sigma}\big\|e^{(T-s)A}\,Q^{1/2}\big\|_{L_2(\Hs)}^2 ds\Big)^m\Bigg)
<\infty
$.

\textbf{(b)} $\E\Big[\|X_T^{t+\delta,\,x}-X_T^{t,\,x}\|_{\Hs}^{2}\Big]
=\E\Big[\|e^{(T-t)A}\big(e^{-\delta A}x-x\big)+\int\limits_{t+\delta}^Te^{(T-s)A}dB_s
$

\hfill$
-\int\limits_t^Te^{(T-s)A}dB_s\|_{\Hs}^2\Big]\\
\leq\Big(\|e^{(T-t)A}\big(e^{-\delta A}x-x\big)\|^2_\Hs+\E\Big[\|\int\limits_t^{t+\delta}e^{(T-s)A}dB_s\|_{\Hs}^2\Big]\Big)\\
\overset{\textbf{Th.\ref{th_ito_BDG_general}}}{\leq}\Big(\|e^{(T-t)A}\big(e^{-\delta A}x-x\big)\|^2_\Hs
+\|\int\limits_t^{t+\delta}\sup\limits_{Q\in\Sigma}\big\|e^{(T-s)A}\,Q^{1/2}\big\|_{L_2(\Hs)}^2 ds\Big)
\xrightarrow[\delta\to 0]{}0$.
\\
Also \ $u(T,x)=\E\big[f(x)\big]=f(x)\leq f(x)$.
\\
So we see that $u$ is a viscosity subsolution of equation \eqref{eq_ppp}.

In the same way one can prove that $u$ is a viscosity supersolution, and the existence is proved.
\\
 It is clear that if $f$ is $B$-continuous and has a polynomial growth that $u$ is also $B$-continuous and has a polynomial growth, because a sublinear expectation $\E$ does not influence on it. So we can conclude that $u$ is a unique viscosity solution by \textbf{Th.\ref{th_uniq_vs}}.

\end{proof}

\end{document}